\documentclass[10pt]{amsart} 

\usepackage[psamsfonts]{amssymb} 
\usepackage{amsfonts,amsmath} 
\usepackage{graphicx, color}
\usepackage{enumerate}
\usepackage{url}
\usepackage{hyperref}
\usepackage{tikz-cd}
\usepackage{ytableau}

\usepackage[all]{xy}

\newtheorem{theorem}{Theorem}[section]

\newtheorem{proposition}[theorem]{Proposition}
\newtheorem{lemma}[theorem]{Lemma}
\newtheorem{corollary}[theorem] {Corollary}

\newtheorem*{claim*}{Claim}

\theoremstyle{remark}
\newtheorem{remark}[theorem]{Remark}

\theoremstyle{definition}
\newtheorem{definition}[theorem]{Definition}

\def\Z{\mathbb Z}
\def\R{\mathbb R}
\def\N{\mathbb N}

\def\coker{{\mathrm{Coker}}}
\def\hom{{\mathrm{Hom}}}

\def\co{\colon}
\def\rewriting{\mathfrak{R}}

\newcommand\abs[1]{\lvert#1\rvert}

\newcommand\sg[1]{\mathfrak{S}_{#1}}
\newcommand\presby[2]{\langle#1\mid #2\rangle}
\newcommand\ramp[2]{\rho(#1,#2)}

\DeclareMathOperator{\im}{im}

\title[A $3$-skeleton for a $K(\sg{n},1)$]{A $3$-skeleton for a classifying space for the symmetric group}
\author{Matthew B. Day and Trevor Nakamura}   

\date{\today}
\begin{document}

\begin{abstract}
We construct a $3$-dimensional cell complex that is the $3$-skeleton for an Eilenberg--MacLane classifying space for the symmetric group $\mathfrak{S}_n$.
Our complex starts with the presentation for $\mathfrak{S}_n$ with $n-1$ adjacent transpositions with squaring, commuting, and braid relations, and adds seven classes of $3$-cells that fill in certain $2$-spheres bounded by these relations.
We use a rewriting system and a combinatorial method of K. Brown to prove the correctness of our construction.
Our main application is a computation of the second cohomology of $\mathfrak{S}_n$ in certain twisted coefficient modules; we use this computation in a companion paper to study splitting of extensions related to braid groups.
As another application, we give a concrete description of the third homology of $\mathfrak{S}_n$ with untwisted coefficients in $\mathbb{Z}$. 
\end{abstract}

	\maketitle

\section{Introduction}
Let $n$ be a positive integer and let $\sg{n}$ be the symmetric group on $\{1,2,\dotsm,n\}$.
It is well known that $\sg{n}$ is given by the presentation 
\begin{equation}\label{eq:Snpresentation} \sg{n}=\presby{s_1,s_2,\dotsc,s_{n-1}}{R},
\end{equation}
with relations $R$ consisting of:
\begin{itemize}
\item $s_i^2=1$ for $i=1,\dotsc,n-1$,
\item $s_is_{i+1}s_i=s_{i+1}s_is_{i+1}$ for $i=1,\dotsc,n-2$, and
\item $s_is_j=s_js_i$ for $1\leq i,j\leq n-1$ with $\abs{i-j}\geq 2$.
\end{itemize}
Here each $s_i$ represents the transposition $(i,i+1)$.
From a topological perspective, (using the Seifert-Van Kampen theorem) this tells us that $\sg{n}$ is the fundamental group of a cell complex $X$ whose $1$-skeleton is a rose with $n-1$ edges indexed by the $s_i$, and whose $2$-cells are glued in with their boundaries following the three classes of relations given above.
For homological computations, it is useful to describe the $3$-cells in a complex $X$ with the same $2$-skeleton, but with $X$ being $2$-connected; such a description is the main theorem of this paper.

There are several presentations for $\sg{n}$ to choose from.
We use the presentation above because it comes from a convenient rewriting system, and this rewriting system gives us a methodical way to verify our $3$-skeleton, following an algorithm of Brown~\cite{BrownRewriting}.
We describe this rewriting system in Section~\ref{Section: Rewriting system for Sn} below.

The applications of this construction are computations of specific homology and cohomology groups.
Our main interest came from studying the second cohomology of $\sg{n}$ in specific twisted modules that are related to extensions of $\sg{n}$ that arise in the study of the braid group.
These modules are examples of permutation modules of Young tableaux with two rows, and subquotients of these modules, working over the integers.
We explore this connection in Day--Nakamura~\cite{DN1}; in that paper we describe the low-dimensional cohomology of $\sg{n}$ in Specht modules for the partitions $(n-1,1)$ and $(n-2,2)$.

The homology and cohomology of the symmetric group with untwisted coefficients is fairly well understood.
Working over the field $\mathbb{F}_p$ of prime order $p$, there are algorithms to give bases for $H_k(\sg{n};\mathbb{F}_p)$ and $H^k(\sg{n};\mathbb{F}_p)$ for any $k$.
This goes back to Nakaoka~\cite{Nakaoka}, and a modern reference is Chapter~VI of Adem--Milgram~\cite{AMCFG}.
We use these ideas in a limited way to show that our construction of $H_3(\sg{n};\Z)$ is complete; see Section~\ref{ss:untwisted3rd} below.
There are only a few references discussing homology and cohomology of $\sg{n}$ in twisted modules; these include Burichenko--Kleschev--Martin~\cite{BKM}, Nagpal~\cite{Nagpal}, and Weber~\cite{Weber}.

\subsection{Statement of results}
In Section~\ref{ss:cxdefn} below, we construct a $3$-dimensional chain complex $P_*$ of $\Z \sg{n}$-modules.
This is the cellular chain complex of the universal cover $\widetilde X$ of a $3$-dimensional cell complex $X$ that we define at the same time as $P_*$.
Our main theorem:
\begin{theorem}\label{th:main}
    The complex $P_*$ is a $3$-dimensional chain complex of $\Z \sg{n}$-modules with an augmentation map $\epsilon\co P_0\to \Z$, and with $H_i(P_*)=0$ for $i\in\{0,1,2\}$.
In particular, $P_*$ is the $3$-dimensional truncation of a free resolution for $\Z$ as a $\Z \sg{n}$-module, and $X$ is the $3$-skeleton of a $K(\sg{n},1)$-space.
\end{theorem}
So for any $\Z \sg{n}$-module $M$, we may compute low-dimensional homology and cohomology by
\[H_i(\sg{n};M)=H_i(P_*\otimes_{\Z \sg{n}} M)\quad\text{ and }\quad H^i(\sg{n};M)=H^i(\hom_{\Z \sg{n}}(P_*,M))\]
for $i\in\{0,1,2\}$.

Fix a commutative ring $R$ and an integer $k$ with $1\leq k\leq n/2$, and consider the permutation module $M^k_R$, the free $R$-module whose basis is indexed by the set of $k$-subsets of $\{1,2,\dotsc,n\}$ (see Definition~\ref{de:permmod}).
This is an $\sg{n}$-module by the usual action of $\sg{n}$ on the basis.
We use the notation $R[2]$ for the $2$-torsion of $R$.
\begin{proposition}
We have the following computations:
\begin{itemize}
\item
For $k\geq 2$, $H_1(\sg{n};M^k)\cong (R/2R)^2$.
\item
Again for $k\geq 2$, $H^1(\sg{n};M^k)\cong R[2]^2$.
\item
For $k\geq 4$, $H_2(\sg{n};M^k)\cong R[2]^2\oplus (R/2R)^3$.
\item
Again for $k\geq 4$, $H^2(\sg{n};M^k)\cong (R/2R)^2\oplus R[2]^3$.
\end{itemize}
Further, there are explicit descriptions of cycles and cocycles that generate these modules in terms of $P_*$.
\end{proposition}
\begin{proof}
The cohomology statements are in Proposition~\ref{pr:Mkcohomology} below, and the homology statements are in Proposition~\ref{pr:Mkhomology} below. 
Both propositions include the explicit generators.
\end{proof}
These computations are essential for the application in Day--Nakamura~\cite{DN1}.

Our last application is a description of the third homology of $\sg{n}$ with untwisted coefficients in $\Z$.
\begin{theorem}\label{th:untwistedthird}
Let $n\geq 6$.
We have 
\[H_3(\sg{n})\cong(\Z/2\Z)^2\oplus (\Z/3\Z) \oplus (\Z/4\Z).\]
We have explicit descriptions of cycles representing generators, in terms of $P_*$.
\end{theorem}
\begin{proof}
This is proven in Corollary~\ref{co:tightH3}, which also gives the generating cycles.
\end{proof}
Certainly $H_3(\sg{n};\mathbb{F}_p)$ has been previously described for all primes $p$ (as explained above), but we believe that the explicit description of generators in terms of a cell complex to be new.

\subsection{The definition of the $3$-complex}\label{ss:cxdefn}
We give some reminders about basic definitions in cohomology of groups in Section~\ref{se:preliminaries1}; some readers will find it helpful to review that section before reading this one.
We include this section here because we thought it would helpful to include the definition of our $3$-complex in the introduction.

The starting point for our construction is presentation~\eqref{eq:Snpresentation} for $\sg{n}$.
The chain complex is the cellular chain complex of the universal cover of a $3$-dimensional cell complex $X$ with $\pi_1(X)=\sg{n}$.
Then our chain complex is $P_*=C_*(\widetilde X)$.
The deck transformations of $\widetilde X\to X$ make $P_*$ into a chain complex of $\Z \sg{n}$-modules.

Instead of describing the gluing maps for the cells of $X$, we give the cellular boundaries in $\widetilde X$; this is equivalent to describing $P_*$ by generators and boundaries.
The interested reader can construct the gluing maps using these boundary formulas and Figures~\ref{fi:2cells}, \ref{fi:most3cells}, and~\ref{fi:cellc37}.

\begin{figure}[h!]
\includegraphics{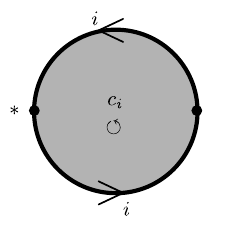}
\includegraphics{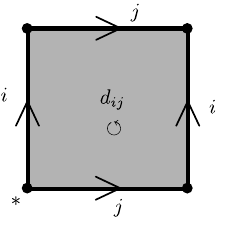}
\includegraphics{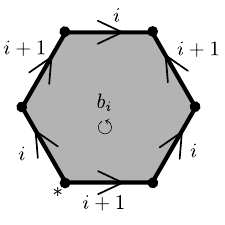}
\caption{The three kinds of $2$-cells in $X$.  Here we label edges with ``$i$" instead of ``$e_i$", etc.  The condition that $i<j$ is part of the definition for $d_{ij}$.}\label{fi:2cells}
\end{figure}

\begin{definition}[The complex $P_*$]\label{de:Snres}
Let $P_0$ be the free $\Z \sg{n}$-module on a single generator $*$.
Let $P_1$ be the free $\Z \sg{n}$-module on $n-1$ generators $e_1,\dotsc, e_{n-1}$.
For all $i$, $\partial e_i=(s_i-1)*$; this defines the chain boundary map $\partial\co P_1\to P_0$.
Now $P_2$ will be a free $\Z \sg{n}$-module on the following set of three classes of free generators.
We define $\partial\co P_2\to P_1$ simultaneously by giving the boundaries of the generators.
\begin{itemize}
    \item For each $i$ from $1$ to $n-1$, we have a generator $c_i$ for $P_2$, with $\partial c_i=(s_i+1)e_i$.
    \item For each $i$ from $1$ to $n-2$, we have a generator $b_i$ for $P_2$, with 
    \[\partial b_i= (1-s_i+s_{i+1}s_i)e_{i+1}-(1-s_{i+1}+s_is_{i+1})e_i.\]
    \item For each $i$ and $j$ with $1\leq i<j\leq n-1$ and $j\geq i+2$, we have a generator $d_{ij}$ for $P_2$, with $\partial d_{ij}=(s_j-1)e_i-(s_i-1)e_j.$
\end{itemize}

Similarly, $P_3$ is a free $\Z \sg{n}$-module on seven classes of generators, and we give free generators and their boundaries here.
\begin{itemize}
    \item For each $i$ from $1$ to $n-1$, there is a generator $c^{3,1}_i$ with $\partial c^{3,1}_i=(s_i-1)c_i$.
    \item For all $i$ and $j$ with $1\leq i, j\leq n-1$ and $\abs{j-i}\geq 2$, there is a generator $c^{3,2}_{ij}$ with 
    \[\partial c^{3,2}_{ij}=
    \left\{
    \begin{array}{cc}
    (s_j-1)c_i-(s_i+1)d_{ij} & \text{if $i<j$} \\
    (s_j-1)c_i+(s_i+1)d_{ji} & \text{if $j<i$}
    \end{array}
    \right.\]
    \item For all $i,j,k$ with $1\leq i<j<k\leq n-1$, with $j>i+1$ and $k>j+1$, there is a generator $c^{3,3}_{ijk}$ with
    \[\partial c^{3,3}_{ijk}=(s_i-1)d_{jk}-(s_j-1)d_{ik}+(s_k-1)d_{ij}.\]
    \item For all $i$ and $j$ with $1\leq 1,j\leq n-1$ and $j<i-1$ or $j>i+2$, there is a generator $c^{3,4}_{ij}$ with 
    \[\partial c^{3,4}_{ij}=
    \left\{
    \begin{array}{cc}
    (s_j-1)b_i-(1-s_i+s_{i+1}s_i)d_{i+1,j} &  \\
    \qquad+(1-s_{i+1}+s_is_{i+1})d_{ij} & \text{if $i+1<j$} \\
    (s_j-1)b_i+(1-s_i+s_{i+1}s_i)d_{j,i+1} & \\
    \qquad -(1-s_{i+1}+s_is_{i+1})d_{ji} & \text{if $j<i$} 
    \end{array}
    \right.
    \]
    \item For all $i$ with $1\leq i\leq n-2$, there is a generator $c^{3,5}_i$ with
    \[\partial c^{3,5}_i=(s_{i+1}+1)b_i - c_{i+1} + s_is_{i+1}c_i.\]
    \item For all $i$ with $1\leq i\leq n-2$, there is a generator $c^{3,6}_i$ with
    \[\partial c^{3,6}_i=(s_is_{i+1}-1)b_i + (s_{i+1}-1)(s_ic_{i+1}+c_i).\]
    \item For all $i$ with $1\leq i\leq n-3$, there is a generator $c^{3,7}_i$ with
    \[
    \begin{split}
    \partial c^{3,7}_i = &(1-s_{i+1}+s_is_{i+1}-s_{i+2}s_is_{i+1}+s_{i+2}s_{i+1}+s_{i+1}s_{i+2}s_is_{i+1})d_{i,i+2}    \\
    & +(s_{i+2}-1+s_is_{i+1}s_{i+2}-s_{i+1}s_{i+2})b_i \\
    &+(s_i-1+s_{i+2}s_{i+1}s_i-s_{i+1}s_i)b_{i+1}. \\
    \end{split}
    \]
\end{itemize}
\end{definition}

\begin{remark}
The generators $e_1,\dotsc,e_{n-1}$ are lifts of the edges in $X$, which are the generators in the presentation~\eqref{eq:Snpresentation}.
The generators for $P_2$ are lifts of the $2$-cells in $X$ glued in by the relations in the presentation; this means that the boundaries are describing loops in $\widetilde X$ that are lifts of the relations in the presentation.
We note that with the generators $d_{ij}$, we require that $i<j$ (for other objects indexed by a pair in this paper, we do not require this).

Each generator for $P_3$ describes the lift of some $3$-cell in $X$, and these $3$-cells are chosen to fill in $2$-spheres in the presentation complex for $\sg{n}$.
It is possible to visualize the cells; see Figures~\ref{fi:2cells}, \ref{fi:most3cells}, and~\ref{fi:cellc37}.
Each $c^{3,1}_i$ is a ball, each $c^{3,2}_{ij}$ and $c^{3,5}_i$ is a tube, each $c^{3,3}_{ijk}$ is a cube, each $c^{3,4}_{i,j}$ is a hexagonal prism, and each $c^{3,7}_i$ is a truncated octahedron.
The cells of the form $c^{3,6}_i$ are harder to describe; each one is a hexagonal pillow with four bigon faces.
It is worth noting that the $b_i$-cells and $c^{37}_i$-cells are fillings of Cayley graphs of $\sg{3}$ and $\sg{4}$.
\end{remark}

\subsection{Acknowledgments}
We would like to thank the following people for helpful conversations: Wade Bloomquist, Susan Hermiller, Dan Margalit, Daniel Minahan, Peter Patzt, Andrew Putman, and Nick Salter.

\begin{figure}[hp]
\begin{tabular}{ccc}
$c^{31}_i$ & $c^{32}_{ij}$, front faces & $c^{32}_{ij}$, back faces \\
\includegraphics{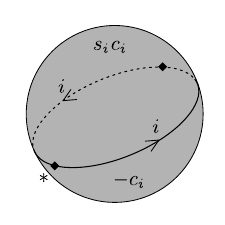} & 
\includegraphics{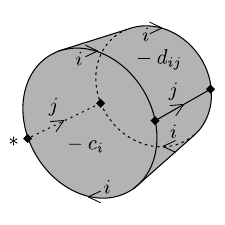} & 
\includegraphics{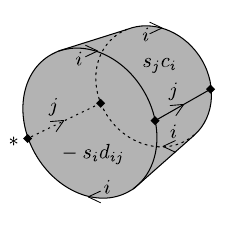}  
\end{tabular}
\begin{tabular}{ccc}
$c^{33}_{ijk}$, front faces & $c^{33}_{ijk}$, back faces & $c^{34}_{ij}$, front faces \\
\includegraphics{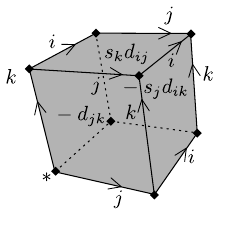} & 
\includegraphics{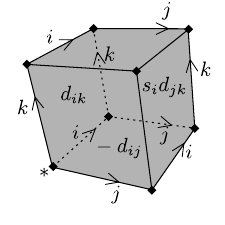} & 
\includegraphics{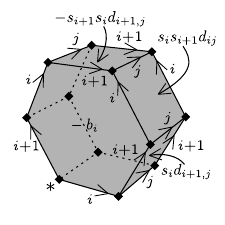}  
\end{tabular}
\begin{tabular}{ccc}
$c^{35}_{i}$, front faces & $c^{35}_{i}$, back faces & $c^{34}_{ij}$, back faces \\
\includegraphics{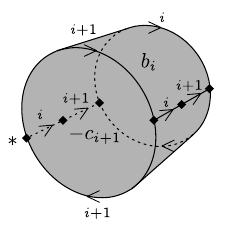} & 
\includegraphics{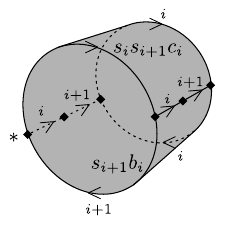} & 
\includegraphics{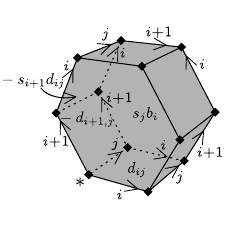}  
\end{tabular}
\begin{tabular}{cc}
$c^{36}_{i}$, front faces & $c^{36}_{i}$, back faces \\
\includegraphics{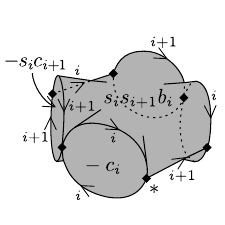} & 
\includegraphics{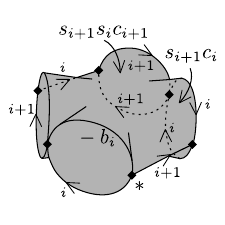}   
\end{tabular}
\caption{The $3$-cells of $X$, except for $c^{37}_i$, which is in another figure.
Again we label edges with ``$i$" instead of ``$e_i$".
The condition that $i<j<k$ is part of the definition for $c^{33}_{ijk}$.
The cells $c^{32}_{ij}$ and $c^{34}_{ij}$ are illustrated with the assumption that $i<j$, but there are also versions of these cells where the reverse inequality holds.}\label{fi:most3cells}
\end{figure}

\begin{figure}[htbp]
\begin{tabular}{cc}
$c^{37}_{i}$, front faces & $c^{37}_{i}$, back faces \\
\includegraphics[scale=0.8]{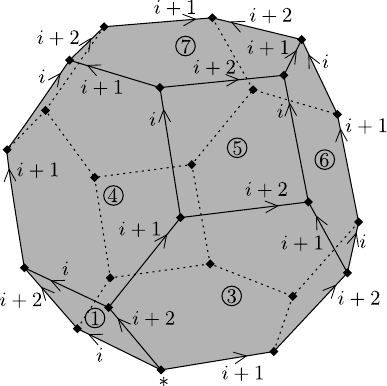} & 
\includegraphics[scale=0.8]{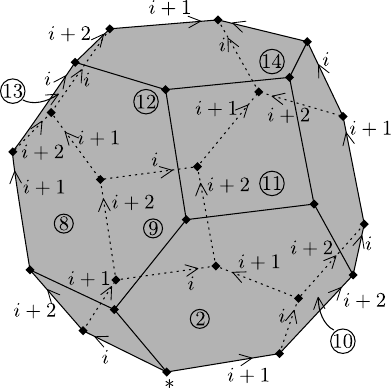}
\end{tabular}
\begin{tabular}{|cc|cc|}
\hline
label & face & label & face \\
\hline
1 & $d_{i,i+2}$ & 8  & $s_i b_{i+1}$\\
2 & $-b_i$ & 9  & $s_is_{i+1}d_{i,i+2}$ \\
3 & $-b_{i+1}$ & 10  & $-s_{i+1}d_{i,i+2}$ \\
4 & $s_{i+2}b_i$ & 11  & $-s_{i+1}s_ib_{i+1}$ \\
5 & $s_{i+2}s_{i+1}d_{i,i+2}$ & 12  & $s_is_{i+1}s_{i+2}b_i$ \\
6 & $-s_{i+1} s_{i+2} b_i$ & 13  & $-s_{i+2}s_is_{i+1}d_{i,i+2}$ \\
7 & $s_{i+2}s_{i+1} s_i b_{i+1}$ & 14  & $s_{i+1}s_{i+2}s_is_{i+1}d_{i,i+2}$ \\
\hline
\end{tabular}

\caption{The $3$-cell $c^{37}_i$, with faces labeled by the key here.  Again we label edges with ``$i$" instead of ``$e_i$".}\label{fi:cellc37}
\end{figure}

\section{Preliminaries to the $3$-complex}\label{se:preliminaries1}
\subsection{Group homology}
We start by reviewing the basics of homology of groups; Brown~\cite{Brown} is a standard reference.
Recall that, given a group $G$ and $\Z G$-module $A$, there are several ways to compute the homology groups $H_n(G;A)$ and cohomology groups $H^n(G;A)$.
First of all, if $X$ is an Eilenberg-Mac Lane classifying space for $G$, denoted a $K(G,1)$ space, then $H_n(G;A)$ and $H^n(G;A)$ may be identified with the topological homology and cohomology of $X$ using a local coefficient system associated with $A$.
(A $K(G,1)$ space is a connected topological space $X$ with $\pi_1(X)=G$ and with contractible universal cover.)

A second approach to group homology is to recognize $H_n(G;A)$ as $\mathrm{Tor}^{\Z G}_n(\Z,A)$ and $H^n(G;A)$ as $\mathrm{Ext}^n_{\Z g}(\Z, A)$, where $\Z$ is the untwisted $\Z G$-module.
Then computation is done using projective resolutions of $\Z$ as a $\Z G$-module.
This dovetails with the topological approach, since if $X$ is a cellular $K(G,1)$-space, then the cellular cochains of the universal cover of $X$ form a free resolution of $\Z$ as a $\Z G$-module.
However, there are many choices of projective resolution, and there are constructions with better functorial properties than resolutions from topology.
One such construction is the \emph{bar resolution}.

\subsection{Bar resolutions} \label{section: resolutions}
We review the construction of the normalized bar resolution and the cellular chain complex obtained from constructing the Cayley complex. The interested reader is referred to Chapter I.5 of Brown's text, \cite{Brown}, for a complete description of the normalized bar resolution. 

As above, $G$ is a group. Topologically, construct a simplicial complex, $BG$ where the vertices are group elements and $G$ acts on the vertices by left translation with each finite subset of $G$ a simplex of $BG$. (This $BG$ is an infinite-dimensional simplex with vertex set $G$.)
For each nonnegative integer $t$, we let $F_t$ denote the simplicial $t$-chains of $BG$.
Then $F_t$ is a free $\Z$-module, freely generated by the $t$-dimensional simplices, which are $(t+1)$-tuples of the form $(g_{0},g_{1},\ldots,g_{t})$ with $g\cdot(g_{0},g_{1},\ldots, g_{t})=(g\cdot g_{0},g\cdot g_{1},\ldots, g\cdot g_{t})$. 
We give $F_*$ the standard boundary operator, $\partial F_{t}\to F_{t-1}$ as $\sum_{i=0}^{t}(-1)^i d_{i}$ where $d_{i}(g_{0},\ldots,g_{t})=(g_{0},\ldots,\hat{g}_{i},\ldots g_{t})$.
We give $F_*$ the augmentation $\epsilon\co F_0\to \Z$ by $\epsilon(g)=1$.
This $F_*$ is the \emph{bar resolution} and is a free resolution of $\Z$ over $\Z G$.

The generators of $F_{t}$ are usually written using \emph{bar notation}:
\[[g_{1}\mid g_{2}\mid\cdots\mid g_{t}]=(1,g_{1}, g_{1}g_{2},\ldots, g_{1}\cdots g_{t}).\]
 The boundary operator that we gave above becomes $\partial_{t}^{F}([g_{1}\mid\cdots\mid g_{t}])=\sum\limits_{i=0}^{t}(-1)^i d_{i}$ where $d_{i}$ is defined by:
$$d_{i}([g_{1}\mid\cdots\mid g_{t}])=\begin{cases}
g_{1}\cdot[g_{1}\mid\cdots\mid g_{t}]\hspace{1em}&i=0\\
[g_{1}\mid\cdots\mid g_{i-1}\mid g_{i}g_{i+1}\mid g_{i+2}\mid\cdots\mid g_{t}]\hspace{1em}&0<i<t\\
[g_{1}\mid\cdots\mid g_{t-1}]\hspace{1em}&i=t.
\end{cases}$$
To obtain the \textit{normalized bar resolution}, take the quotient of the bar resolution by the submodule generated by elements $[g_{1}\mid\cdots\mid g_{t}]$ such that $g_{i}=1$ for some $1\leq i\leq t$.  
It is an exercise to show that this submodule is closed under the boundary map, so that this quotient is actually a subcomplex.
We will denote the normalized bar resolution by $\bar{F}_{*}$.
 Topologically, $\bar{F}_{*}$ is simplicial chain complex of the quotient of the space $BG$ given by collapsing those degenerate simplices $(g_0,g_1,\dotsc,g_n)$ in which there is some $i$ with $g_i=g_{i+1}$.

\subsection{Resolutions from presentations}
Suppose $G$ has presentation $\presby{S}{R}$. 
Build a cell complex $X$ as follows: there is single $0$-cell $*$, the $1$-cells are oriented and indexed by $S$, and the $2$-cells are indexed by $R$, and for each $r\in R$, the $2$-cell of $r$ is glued in with its boundary following the directed edges for the letters in $S\cup S^{-1}$ in sequence as they appear in $r$.
This complex is the \emph{presentation complex}.
The \emph{Cayley complex} is the universal cover $\widetilde X$ of the presentation complex.
To describe the Cayley complex, we choose a $0$-cell $*$ and use the left action of $G$ by deck transformations to describe the other cells.
The action of $G$ on the $0$-cells is faithful and transitive, so we have a $0$-cell $g*$ in $\widetilde X$ for each $g\in G$.
For each $s\in S$, let $x_s$ denote the $1$-cell in $\widetilde X$ that starts at $*$ and descends to the $1$-cell indexed by $s$ in $X$.
Again, the action by deck transformations is faithful, and is transitive on the preimages of each $1$-cell in $X$, so we have $1$-cells $g x_s$ in $\widetilde X$ for each $g\in G$ and $s\in S$, and every $1$-cell in $\widetilde X$ has exactly one such label.
Finally we describe the $2$-cells in $\widetilde X$.
A relation $r$ in $R$ can be given as a word $r=s_1^{b_1}s_1^{b_2}\dotsm s_n^{b_n}$ for some $s_1,\dotsc,s_n\in S$ and $b_1,\dotsc,b_n\in \{1,-1\}$.
By the construction of the presentation complex, there is a unique $2$-cell in $X$ labeled by this relation and glued in with its boundary following the path in the $1$-skeleton indicated by $s_1^{b_1}s_1^{b_2}\dotsm s_n^{b_n}$.
Let $e_r$ denote the lift in $\widetilde X$ of this $2$-cell that is glued in starting at $*$.
Then the $2$-cells of $\widetilde X$ are uniquely labeled as $g e_r$ for $g\in G$ and $r\in R$.

The action of $G$ by deck transformations on $\widetilde X$ gives the cellular homology chains of $\widetilde X$ the structure of a chain complex of $\Z G$-modules.
This chain complex is free in each dimension; it is freely generated by the basepoint $*$ in dimension $0$, by the $1$-cells $\{x_s\}_{s\in S}$ in dimension $1$, and by the $2$-cells $\{e_r\}_{r\in R}$ in dimension $2$.
It is immediate from the construction that boundaries are given by $\partial *=0$, $\partial x_s=s*-*$, and if $r=s_1^{b_1}s_1^{b_2}\dotsm s_n^{b_n}$, then
$\partial e_r= y_1+s_1^{b_1}y_2+s_1^{b_1}s_2^{b_2}y_3+\dotsm+s_1^{b_1}s_2^{b_2}\dotsm s_{n-1}^{b_{n-1}}y_n$, where
\[y_i=
\left\{
\begin{array}{cc}
x_{s_i} & \text{ if $b_i=1$} \\
-s_i^{-1} x_{s_i} & \text{ if $b_i=-1$.} 
\end{array}
\right.
\]

\subsection{Normal forms, rewriting systems, and collapsing schemes}
\label{Section: Collapsing Schemes}
We describe the application of rewriting systems Brown used to compute the homology of groups and monoids. 
The interested reader is referred to Brown's paper~\cite{BrownRewriting}. 
The majority of results required for this paper are found in sections~3 and~5. 

Let $G$ be a group with presentation $\presby{S}{R}$ and let $F(S)$ denote the free group generated by $S$. 
A set of \textit{normal forms} for $G$ is a subset $I\subseteq F(S)$ which maps bijectively onto $G$ under the canonical projection $\pi\co F(S)\to G$.
A \emph{rewriting rule} is an ordered pair $(r_1,r_2)$, written as $r_1\to r_2$, of words $r_1, r_2\in F(S)$, such that $r_1=r_2$ in $G$.
A \emph{rewriting system} is a set of rewriting rules.
Given a rewriting system, we write $w_1\to w_2$ to mean that $w_1$, $w_2$ are words in $F(S)$, and there is a rewriting rule $r_1\to r_2$ in our rewriting system such that there are words $u,v\in F(S)$ with $w_1=ur_1v$ and $w_2=ur_2v$, as reduced products.
If $w_1\to w_2$, we say that $w_2$ is a \emph{reduction} or \emph{rewriting} of $w_1$.
Furthermore, a word $w\in F(S)$ is called \textit{reducible} if  $w\to w'$ for some $w'\in F(S)$. 
Otherwise, $w$ is called \textit{irreducible}.
\begin{definition}\label{de:complete}
    A rewriting system, $\to$, on $G$ is called \textit{complete} if it satisfies the following two conditions:
    \begin{enumerate}
        \item The set of irreducible words, $I$, is a set of normal forms for $G$.
        \item There is no infinite chain $w\to w_{1}\to w_{2}\to\cdots$ of reductions.
    \end{enumerate}
    Since $S$ is a generating set for $G$, we also assume the following condition for simplicity:
    \begin{enumerate}
        \setcounter{enumi}{2}
        \item Every $s\in S$, considered as a word of length 1, is irreducible. 
    \end{enumerate}
\end{definition}

Following Brown, we use our normal forms to label the generators of the normalized bar resolution.
Let $X$ denote the chain complex whose cells are the generators of the normalized bar resolution $\bar{F}_*$ (so $X$ is a quotient of $BG$ as described in section \ref{section: resolutions}).
Our goal is define a \emph{collapsing scheme} on $X$.
This consists of a partitioning of the cells of $X$ into \emph{essential}, \emph{redundant}, and \emph{collapsible} cells, and a \emph{collapsing function} for $X$ that selects, for each redundant cell, a collapsible cell that it is a face of.
We think of the collapsing scheme as instructions for retracting $X$ onto the subcomplex of essential cells by pulling the redundant cells across the collapsible cells.

As usual, each basis element $[g_{1}\mid g_{2}\mid\cdots\mid g_{t}]$ of $F_{t}$ corresponds to a $t$-cell in $X$.
We label the cells of $X$ as \emph{essential}, \emph{collapsible}, or \emph{redundant}, as follows:
\begin{definition} \label{de:essential cells}
These definitions are made with respect to the complete rewriting system $\to$ on $G$ on the generators $S$.
Let $\tau=[g_{1}\mid g_{2}\mid\cdots\mid g_{t}]$ be a simplex of the bar resolution of $G$, with each $g_i\neq 1$.
Assume each $g_i$ is given as an irreducible word in $S$.
Then $\tau$ is \emph{essential} if 
\begin{itemize}
\item $g_1\in S$,
\item $g_ig_{i+1}$ is reducible for each $i$ with $1\leq i<t$, and
\item every proper initial subword of $g_ig_{i+1}$ is irreducible, for each $i$ with $1\leq i<t$.
\end{itemize}
We say $\tau$ is \emph{collapsible} if 
\begin{itemize}
\item $g_1\in S$, and
\item for the first $i$ such that $[g_1|g_2|\dotsm|g_{i+1}]$ not essential, we have that $g_ig_{i+1}$ is irreducible.
\end{itemize}
Finally, we say $\tau$ is \emph{redundant} if
\begin{itemize}
\item $g_1\notin S$, or
\item $g_1\in S$, and for the first $i$ with $[g_1|g_2|\dotsm|g_{i+1}]$ not essential, we have that $g_ig_{i+1}$ has a proper initial subword that is reducible.
\end{itemize}
\end{definition}
Note that every cell of $X$ satisfies exactly one of these conditions.

\begin{remark}
    We have subtly altered the classification of essential, redundant, and collapsible cells from Brown~\cite{BrownRewriting} to use a left action instead of a right action. The natural change is to write words in reverse order, however we altered the classification setup to maintain the ordering on the normalized bar resolution and classify essential cells with the same ordering as Brown.
\end{remark}

\begin{remark}\label{remark:lowdessential}
We can characterize the low-dimensional simplices as follows.
First, the $0$-simplex $[\hspace{.5em}]$ is essential.
For all $s\in S$ we label the cells corresponding to the bar chains $[s]$ as essential and all other $1$-simplices are redundant.
The redundant $1$-simplices may be considered as $[g]$ where $g$ is a nontrivial, irreducible word in $F(S)$ with $g\not\in S$. 
Thus $g$ admits a nontrivial factorization $g=su$ where $s\in S$ and $u$ is an irreducible word in $F(S)$. 
Therefore $[g]$ is the face of a $2$-simplex $[s\mid u]$. 
The $2$-simplices of the form $[s\mid u]$ are collapsible.

In the $2$-skeleton, we have already described the collapsible simplices. 
The remaining $2$-simplices are of the form $[g_{1}\mid g_{2}]$ where $g_{1}\not\in S$ or $[s\mid g]$ where $s\in S$ and $sg$ is reducible. 
If $g_{1}\not\in S$, then $g_{1}=sg_{1}'$ is a nontrivial factorization and $[g_{1}\mid g_{2}]$ is a face of the $3$-simplex $[s\mid g_{1}'\mid g_{2}]$. 
Therefore, in this case, we classify $[g_{1}\mid g_{2}]$ as redundant and the corresponding $3$-simplex $[s\mid g_{1}'\mid g_{2}]$ as collapsible. 
Now, suppose $sg$ is reducible with factorization $sg=sg'g''$ where $sg'$ is reducible and $g''$ is a possibly empty word. 
If $g''$ is a nonempty word, we label $[s\mid g]$ redundant and can obtain $[s\mid g]$ as the face of the collapsible cell $[s\mid g'\mid g'']$. 
If $g''$ is empty, then we label $[s\mid g]$ as essential. 
Alternatively, $[s\mid g]$ is essential if and only if $sg$ is a rewriting rule.
\end{remark}

A set of normal forms corresponding to some complete rewriting system is often referred to as a set of \textit{good normal forms}. 
We restate a Theorem~2 of Brown~\cite{BrownRewriting} for groups to obtain a quotient complex of the normalized bar resolution.

\begin{theorem}[Brown]\label{Theorem: Quotient Complex} 
Let $G$ be a group with a set of good normal forms and let $\bar{F}_{*}$ be the normalized bar resolution of $G$. 
Let $Q_*$ be the quotient of $F_*$ by the subcomplex generated by the collapsible simplices.
Then:
\begin{itemize}
\item the images of the essential simplices are a free basis for $Q_*$ as a $\Z G$-module,
\item $Q_*$ is a free resolution of $\Z$ as a $\Z G$-module, and 
\item the quotient map $q\co F_*\to Q_*$ is an augmentation-preserving chain homotopy equivalence.
\end{itemize}
\end{theorem}

To finish describing the collapsing scheme, we define the \emph{collapsing function} $\tau\mapsto c(\tau)$.
\begin{definition}
For each non-negative integer $t$, let $c$ be the function from the set of redundant $t$-simplices in $F_t$ to the set of collapsible $(t+1)$-simplices in $F_{t+1}$, defined as follows.
Let $\tau=[g_{1}\mid g_{2}\mid\cdots\mid g_{t}]\in F_t$ be redundant, with each $g_i$ given by its irreducible representative.
\begin{itemize}
\item
If $g_1\notin S$, factor $g_1$ as $g_{1}=sg_{1}'$ with $s\in S$.
Then 
\[c(\tau)=[s\mid g_{1}'\mid g_{2}\mid\cdots\mid g_{t}].\]
\item
If $g_{1}\in S$, let $i$ be the largest integer such that $[g_{1}\mid g_{2}\mid\cdots \mid g_{i-1}]$ is an essential $(i-1)$-cell.
Since $\tau$ is redundant, $g_{i-1}g_{i}$ contains a proper initial subword which is reducible.
Let $g'$ and $g''$ be the subwords such that we have a factorization $g_{i-1}g_{i}=g_{i-1}g'g''$ and $g_{i-1}g'$ is the smallest reducible initial subword of $g_{i-1}g_i$.
Then 
\[c(\tau)=[g_{1}\mid\cdots \mid g_{i-1}\mid g'\mid g''\mid g_{i+1}\mid\cdots g_{t}].\]
\end{itemize}
\end{definition}
In either case, $c(\tau)$ is collapsible, and one of the faces of $c(\tau)$ is $\pm\tau$.

Our next goal is to explain how to express a redundant $t$-cell $\tau$ in $Q_*$ as a linear combination of essential cells.
Since $c(\tau)$ is collapsible, $q(c(\tau))=0$. Furthermore, since $\tau$ is the $i^{th}$ face of $c(\tau)$, taking the quotient of $\partial c(\tau)$ we obtain:
$$0=\pm\tau+\sum\limits_{\substack{j=0\\ j\neq i}}^{t}(-1)^{j}d_{j}c(\tau).$$
Therefore we may determine $\tau$ as a finite $\Z G$-linear combination of essential, redundant, and collapsible cells:
$$\tau= \mp\sum\limits_{\substack{j=0\\ j\neq i}}^{t}(-1)^{j}d_{j}c(\tau).$$
Since the quotient of collapsible cells are $0$, $\tau$ is a $\Z G$-linear combination of essential or redundant cells. Furthermore, the redundant cells can be identified as the face of another collapsible cell. 
We note that the redundant cells resulting from this replacement are closer to being essential, in the following sense: the largest $i$ such that $[g_1\mid g_1\mid\dotsm\mid g_{i-1}]$ is essential is greater for the resulting cells than the original one.
Since the rewriting system is complete, each redundant cell can be written as a $\Z G$-linear combination of essential cells in finitely many application of this process. We include the following lemmas to reduce the computations in Section~\ref{Section: Computation of Essential cells}.
\begin{lemma}\label{Lemma: irreducible tuples quotients to 0}
Let $\tau=[g_{1}\mid g_{2}\mid\cdots\mid g_{t}]$ be a $t$-cell for some $t\geq 2$. If $g_{1}g_{2}$ is irreducible, then $q(\tau)=0$.
\end{lemma}
\begin{proof}
    We prove this by induction on the length of the irreducible word $g_{1}$. If $g_{1}$ is of length $1$, then $g_{1}\in S$ and $\tau$ is collapsible by Definition~\ref{de:essential cells}.
    Therefore, we may assume $g_{1}=sg'$ for some $s\in S$ and some irreducible word $g'$ with length one less than $g_{1}$.    
    Then $\tau$ is a redundant cell with the corresponding collapsible simplex $c(\tau)=[s\mid g'\mid g_{2}\mid\cdots\mid g_{t}]$.
    Therefore applying the method above:
    $$\tau=s[g'\mid g_{2}\mid\cdots\mid g_{t}] +[s\mid g'g_{2}\mid g_{3}\mid\cdots \mid g_{t}]+\sum\limits_{i=3}^{t}(-1)^{i}d_{i}c(\tau).$$
    For $i\geq 3$, $d_i(c(\tau))$ satisfies the assumption that the product of the first two entries is irreducible.
    Similarly, so do $[g'\mid g_{2}\mid\cdots\mid g_{t}]$ and $[s\mid g'g_{2}\mid g_{3}\mid\cdots \mid g_{t}]$ since $g_1g_2$ is irreducible and $g_1g_2=sg'g_2$.  Therefore $q(\tau)=0$.
\end{proof}
Lemma \ref{Lemma: irreducible tuples quotients to 0} is a restatement of Proposition 2 from Section 5 of Brown~\cite{BrownRewriting}. 
\begin{lemma}\label{lemma: rewriting [w|v] for w not a generator}
    Suppose $w=s_{i_{1}}s_{i_{2}}\cdots s_{i_{k}}$ is irreducible and $uv\to r$ with $wu$ irreducible. Let $\tau=[wu\mid v\mid g_{3}\mid\cdots\mid g_{t}]$, then:  
\[
\begin{split}
q(\tau)=&w q([u\mid v\mid g_{3}\mid\cdots\mid g_{t}])\\
&+\sum\limits_{j=0}^{k-1}s_{i_{1}}s_{i_{2}}\cdots s_{i_{j}}q([s_{i_{j+1}}\mid s_{i_{j+2}}\cdots s_{i_{k}}r\mid g_{3}\mid\cdots\mid g_{t}]).
\end{split}
\]
    Furthermore, if $s_{i_{j}}\cdots s_{i_{k}}r$ is irreducible for all $1\leq j\leq k$, then $q(\tau)=wq([u\mid v\mid g_{3}\mid\cdots\mid g_{t}])$.
\end{lemma}
\begin{proof}
We induct on $k$. If $k=1$, then $\tau=[s_{i}u\mid v\mid g_{3}\mid\cdots\mid g_{t}]$ is redundant and the corresponding $(t+1)$-cell $c(\tau)=[s_{i}\mid u\mid v\mid g_{3}\mid\cdots\mid g_{t}]$ has boundary:
$$\partial c(\tau)=s_{i}[u\mid v\mid g_{3}\mid\cdots\mid g_{t}]-\tau+[s_{i}\mid r\mid g_{3}\mid\cdots\mid g_{t}]+\sum\limits_{j=3}^{t}(-1)^{j}d_{j}c(\tau).$$
Since $s_{i+1}u$ is irreducible, $q(d_{j}c(\tau))=0$ for all $j\geq 3$ by Lemma \ref{Lemma: irreducible tuples quotients to 0}. Thus:
$$q(\tau)=s_{i}q([u\mid v\mid g_{3}\mid\cdots\mid g_{t}])+q([s_{i}\mid r\mid g_{3}\mid\cdots\mid g_{t}]).$$
Now, suppose $w=s_{i_{1}}s_{i_{2}}\cdots s_{i_{k}}$ for some $k>1$. Let $w'=s_{i_{2}}s_{i_{3}}\cdots s_{i_{k}}$, then $\tau$ is redundant with the corresponding collapsible $(t+1)$-cell $c(\tau)=[s_{i_{1}}\mid w'u\mid v\mid g_{3}\mid\cdots\mid g_{t}]$, where boundary:
$$\partial c(\tau)=s_{i_{1}}[w'u\mid v\mid g_{3}\mid\cdots\mid g_{t}]-\tau+[s_{i_{1}}\mid w'r\mid g_{3}\mid\cdots\mid g_{t}]+\sum\limits_{j=3}^{t}(-1)^{j}d_{j}c(\tau).$$
Since $wu$ is irreducible, $q(d_{j})c(\tau)=0$ for all $j\geq 3$ by Lemma \ref{Lemma: irreducible tuples quotients to 0}. Applying the inductive hypothesis to $[w'u\mid v\mid g_{3}\mid\cdots \mid g_{t}]$ completes the first statement. To obtain the second statement, apply Lemma \ref{Lemma: irreducible tuples quotients to 0} to $q([s_{i_{j+1}}\mid s_{i_{j+2}}\cdots s_{i_{k}}r\mid g_{3}\mid\cdots\mid g_{t}])$.
\end{proof}

\section{Boundaries in the $3$-complex}

\subsection{Rewriting in the symmetric group}\label{Section: Rewriting system for Sn}
Our proof of Theorem~\ref{th:main} has three parts.
First, we define a rewriting system for our presentation of $\sg{n}$, and prove that it is complete.
In the second step, we consider the complex $Q_*$ given by taking the quotient of the normalized bar resolution $\bar F_*$ by the subcomplex spanned by inessential simplices.
We take the boundaries of the essential simplices of dimension $2$ and $3$ in $Q_*$, and rewrite these as sums of essential simplices.
In the third step, we build chain maps between $P_*$ and $Q_*$, and quote Brown's theorem to deduce that $P_*$ is a truncation of a free resolution of $\Z$ as a $\Z\sg{n}$-module.

As above, $S=\{s_1,s_2,\dotsc, s_{n-1}\}$ is our generating set for $\sg{n}$.
One of our rewriting rules is easier to state with the following definition, which we also use to describe our normal form.
For $i$ and $j$ with $1\leq i\leq j\leq n-1$, define the \emph{ramp} from $s_i$ to $s_j$ to be 
\[\rho(i,j)=s_i s_{i+1} s_{i+2}\dotsm s_j.\]

\begin{definition}\label{de:symmetricrewriting}
Our rewriting system is defined as follows:
\begin{itemize}
\item $\rewriting 0$: for each $i$ with $1\leq i\leq n-1$, we have $s_i^{-1}\to s_i$.
\item $\rewriting 1$: for each $i$ with $1\leq i\leq n-1$, we have $s_i^2\to 1$.
\item $\rewriting 2$: for each $i$ and $j$, with $1\leq i< j-1$ and $j\leq n-1$, we have
$s_is_j\to s_js_i$.
\item $\rewriting 3$: for each $i$ and $j$, with $1\leq i<j\leq n-1$, we have
\[s_j\rho(i,j) \to \rho(i,j)s_{j-1},\]
\end{itemize}
\end{definition}

We note that these rules are valid in $\sg{n}$, meaning that for each rule, the right and left words are equal in $\sg{n}$.
This is immediate in all cases except that $s_j\rho(i,j) = \rho(i,j)s_{j-1}$, which is a consequence of a single braid relation and $j-i-1$ commuting relations.

From this definition of a rewriting system, we get the following characterization of the irreducible elements (the normal forms).
\begin{lemma}\label{le:normalform}
Let $w\in F(S)$ be a nontrivial reduced word in the free group on $S$.
Then $w$ is irreducible if and only if there is a positive integer $m$ and sequences $j_1,j_2,\dotsc,j_m$ and $k_1, k_2,\dotsc,k_m$, with $\{k_i\}_i$ strictly decreasing, and $1\leq j_i\leq k_i\leq n-1$ for each $i$, such that
\[w=\rho(j_1,k_1)\rho(j_2,k_2)\rho(j_3,k_3)\dotsm\rho(j_m,k_m).\] 
\end{lemma}

\begin{remark}\label{remark:creepandjump}
Lemma~\ref{le:normalform} says that $w$ is in normal form if and only if it is a product of ramps of strictly decreasing height.
Another characterization is that $w$ is normal form if and only if
\begin{itemize}
\item only positive generators appear in $w$ (there are no inverse generators $s_i^{-1}$),
\item for a subword $s_is_j$ of length $2$ in $w$, we must have $j=i+1$ or $j<i$ (we can creep up or jump down, but we can't jump up or stay level), and
\item if $s_i$ appears as the top of a ramp, then any $s_j$ with $j\geq i$ cannot appear to the right of this $s_i$.
\end{itemize}
We think of this third condition as a ``no valleys" condition: if the indices of the generators decrease and then increase, it cannot increase to the same height or higher than it started.

Here are some examples of irreducible words:
\begin{itemize}
\item For $s_i\in S$, we have $s_i=\rho(i,i)$, so $s_i$ is irreducible.
\item The word $s_1s_2s_3s_2s_1=\rho(1,3)\rho(2,2)\rho(1,1)$ is in normal form because the heights of the ramps are decreasing.
\item The word $s_2s_3s_4s_5s_6s_1s_2s_3s_4s_3s_1s_2=\rho(2,6)\rho(1,4)\rho(3,3)\rho(1,2)$ might be a more typical example of a word in normal form; in particular, there are no restrictions on the starting points of the ramps, and the only restriction on the heights of the ramps is that they be strictly decreasing.
\end{itemize}
\end{remark}

\begin{proof}[Proof of Lemma~\ref{le:normalform}]
Let $w$ be a nontrivial reduced word in $F(S)$.
First suppose that $w$ is reducible.
Then $w$ has a subword that is the left side of one of the reduction rules in Definition~\ref{de:symmetricrewriting}.
Suppose for contradiction that $w$ is written as a product of ramps with decreasing height.
If $\rewriting0$ applies to $w$, then $w$ has an inverse generator and cannot be written as a product of ramps.
If $\rewriting1$ applies to $w$ then $w$ contains a product of the form $\rho(j,i)\rho(i,k)$, and since $i\leq k$, this product does not have decreasing height.
Similarly, if $\rewriting2$  applies to $w$, with some $i<j-1$, then $w$ contains a product of the form $\rho(k,i)\rho(j,l)$, and this does not have decreasing height.
If $\rewriting3$ applies for some $i<j$, then $w$ contains a product of the form $\rho(k,j)\rho(i,l)$ with $j\leq l$, which does not have decreasing height.
In each case we have a contradiction, proving this direction.

Now suppose that $w$ is irreducible.
Then none of the rewriting rules apply to $w$, and in particular, $w$ contains no inverse generators.
So $w$ can be written as a product of ramps, and we do this in a way that uses as few ramps as possible.
Suppose for contradiction that the ramps are not decreasing in height.
So we have a subword $\rho(i,j)\rho(k,l)$ with $j\leq l$.
If $k<j$, then $\rewriting3$ applies.
If $k=j$, then $\rewriting1$ applies.
If $k=j+1$, then $\rho(i,j)\rho(k,l)=\rho(i,l)$, which contradicts our assumption that we have expressed $w$  as a product of as few ramps as possible.
If $k>j+1$, then $\rewriting2$ applies.
In each case, we get a contradiction, so $w$ can be written as a product of ramps of decreasing height.
\end{proof}

The following theorem is a necessary condition for using Brown's rewriting method for chains.
\begin{theorem}\label{th:rwscomplete}
The rewriting system from Definition~\ref{de:symmetricrewriting} is complete.
\end{theorem}

\begin{proof}
We need to verify the three conditions in Definition~\ref{de:complete}.
As explained Remark~\ref{remark:creepandjump}, each generator $s_i$ is a ramp, and is irreducible by Lemma~\ref{le:normalform}.
The other conditions are less trivial to show.

Our next goal is to show that $\to$ has no infinite sequences of reductions.
We define the following complexity function $C\co F(S)\to \N^{n+1}$.
For a word $w\in F(S)$, $C(w)$ is the $(n+1)$-tuple of natural numbers such that:
\begin{itemize}
\item the first entry of $C(w)$ is the number of inverse generators in $w$,
\item for each $i$ from $2$ to $n$, the $i$th entry of $C(w)$ is the number of instances of $s_{n-i+1}$ in $w$, and
\item the last entry of $C(w)$ is the number of pairs of indices $i<j$ such that the $i$th letter in $w$ is $s_k$, the $j$th letter of $w$ is $s_l$, and $k<l$.
\end{itemize}
We lexicographically order $\N^{n+1}$ and apply rewriting rules to reduce $C(w)$.
This means that our first priority is to remove the inverse generators from $w$; our second priority is to reduce the number of instances of generators $s_i$ with high values of $i$, and our last priority is to reorder elements to replace increasing indices with decreasing indices.
We note that applying any rewriting rule decreases $C(w)$, so that the naive algorithm of applying rules freely will always move $w$ closer to normal form.
Specifically:
\begin{itemize}
\item $\rewriting0$ reduces $C(w)$ since it reduces the first entry in the tuple;
\item $\rewriting1$ reduces $C(w)$ since it reduces entry $n-i+1$ in the tuple (by decreasing the count of $s_i$ generators), but leaves the preceding entries the same;
\item with $i<j$, $\rewriting3$ reduces $C(w)$ since it reduces entry $n-j+1$ in the tuple (by decreasing the count of $s_j$ generators), but  leaves the preceding entries the same (the count of $s_{j-1}$ is a subsequent entry, so it is okay that it increases);
\item with $i<j-1$, $\rewriting2$ reduces $C(w)$ since it reduces the final entry, but keeps all the other entries the same.
\end{itemize}
Since $\N^{n+1}$ is well ordered under the lexicographical ordering, any sequence of applications of rewriting rules must have finite length.

Next, we claim that each permutation has at least one irreducible representative.
Given a permutation $g\in\sg{n}$, we express it as a product of generators from $S$ in any way whatsoever, to get a word $w\in F(S)$ that maps to $g$.
Then we apply rewriting moves to $w$ repeatedly. 
(If there is a choice of different moves at some point in this reduction process, it doesn't ultimately matter which move we choose, because of the uniqueness that we explain in the next paragraph.)
By the previous paragraph, $C(w)$ decreases with each move, so eventually we arrive at an irreducible form for $w$.
Since the rewriting moves are valid, this is an irreducible word that represents $g$.

Finally, we want to show that the irreducible representatives are unique.
Suppose $w$ is a product of ramps in normal form: $w=\rho(j_1,k_1)\rho(j_2,k_2)\dotsm\rho(j_m,k_m)$.
Then $s_{k_1}$ appears exactly once in $w$, and no generators with indices higher than $k_1$ appear.
In particular, considered as an element of $\sg{n}$, $w$ sends $k_1+1$ to $j_1$.
(This uses our convention that permutations compose like functions and act on the left.)
So $w$ acts nontrivially, and therefore the unique irreducible representative of $(1)\in\sg{n}$ is $1\in F(S)$.
Further, the left-most ramp in $w$ is determined by its action on $\{1,2,\dotsc,n\}$.
An induction argument on the length of $w$ then shows that $w$ is the unique irreducible representative of its image in $\sg{n}$.
\end{proof}

\subsection{Rewriting boundaries as sums of essential cells}
\label{Section: Computation of Essential cells}
In Section \ref{Section: Rewriting system for Sn} we constructed a complete rewriting system which yields a collapsing scheme on the simplicial complex described by $F_{*}$. 
As in Brown's theorem, let $Q_*$ be the quotient of the normalized bar resolution $F_*$ by the subcomplex spanned by collapsible cells.
We have the following map of resolutions:
$$\begin{tikzcd}
F_{3}\arrow{r}{\partial^{F}_{3}}\arrow{d}{q_{3}}&F_{2}\arrow{r}{\partial^{F}_{3}}\arrow{d}{q_{2}}&F_{1}\arrow{r}{\partial_{1}}\arrow{d}{q_{1}}&F_{0}\arrow{r}{\epsilon}\arrow{d}{q_{0}}&\Z\arrow{r}\arrow{d}&1\\
Q_{3}\arrow{r}{\partial_{3}^{Q}}&Q_{2}\arrow{r}{\partial_{2}^{Q}}&Q_{1}\arrow{r}{\partial_{1}^{Q}}&Q_{0}\arrow{r}&\Z\arrow{r}&1
\end{tikzcd}$$
The boundary maps $\partial_{i}^{Q}$ are defined by the composition $q_{i-1}\circ\partial_{i}^{F}$. 
So the usual boundary in $F_*$ determines the boundary in $Q_*$, and we can write boundaries in $Q_*$ by referring to their images under $q$.
However, but it is better to understand the boundary in terms of the free basis for $Q_*$ given by the essential cells.
Thus, to describe the boundary we must rewrite the boundaries of essential cells entirely in terms of essential cells, modulo the redundant and collapsible cells.
For $0\leq i\leq 2$, the essential $i$-cells have already been explicitly described in Remark~\ref{remark:lowdessential}.
Unpacking Definition~\ref{de:essential cells} to the rewriting system for $\sg{n}$ described above, we describe the essential $t$-cells with the following proposition.
\begin{proposition}\label{pr: g_{i}g_{i+1} is a rewriting rule}
    For $t\geq 1$, the essential cells of $F_{t}$ which generate $Q_{i}$ are of the form $[g_{1}\mid g_{2}\mid\cdots\mid g_{t}]$ where $g_{1}\in S$ and $g_{i}g_{i+1}$ has precisely the form of the reducible word on the left side of $\rewriting1$, $\rewriting2$, or $\rewriting3$, for all $1\leq i\leq t-1$. 
\end{proposition}
We include the following proposition to summarize the essential cells described in Section \ref{Section: Collapsing Schemes}.
\begin{proposition}\label{prop: essential 1 and 2 cells}
For $S_{n}$ with the rewriting system described by $\rewriting0-\rewriting3$, there is one 0-cell, $[\hspace{.25em}]$, $n-1$ 1-cells denoted by $[s_{i}]$ with $\partial_{1}^{Q}[s_{i}]=(s_{i}-1)[\hspace{.5em}]$. Furthermore, there are three classes of 2-cells described by:
\begin{itemize}
    \item $[s_{i}\mid s_{i}]$ for all $1\leq i\leq n-1$ with $\partial_{2}^{Q}[s_{i}\mid s_{i}]=(s_{i}+1)[s_{i}]$.
    \item $[s_{i}\mid s_{j}]$ for all $1\leq i<j\leq n-1$ with $i\leq j-2$ with $\partial_{2}^{Q}[s_{i}\mid s_{j}]=(s_{i}-1)[s_{j}]-(s_{j}-1)[s_{i}]$.
    \item $[s_{j}\mid \ramp{i}{j}]$ for all $1\leq i<j\leq n-1$ with:
    $$\partial_{2}^{Q}[s_{j}\mid \ramp{i}{j}]=(s_{j}-1)\sum\limits_{\ell=i}^{j}\ramp{i}{\ell-1}[s_{\ell}]+[s_{j}]-\ramp{i}{j}[s_{j-1}].$$
\end{itemize}
\end{proposition}
\begin{proof}
    The boundaries of the 1-cells and $[s_{i}\mid s_{i}]$ are determined directly by the definition of the boundary for the normalized bar resolution. Suppose $1\leq j-2$, then:
    $$\partial_{2}^{F}[s_{i}\mid s_{j}]=s_{i}[s_{j}]-[s_{j}s_{i}]+[s_{i}].$$
    Let $\tau=[s_{j}s_{i}]$, then $c(\tau)=[s_{j}\mid s_{i}]$ and $\partial^{F}c(\tau)=s_{j}[s_{i}]-\tau+[s_{i}]$. Therefore $q(\tau)=s_{j}[s_{i}]+[s_{j}]$ and the boundary of $[s_{i}\mid s_{j}]$ follows. 
    
    Now, suppose $i<j$, then: $$\partial_{2}^{F}[s_{j}\mid \ramp{i}{j}]=s_{j}[\ramp{i}{j}]-[\ramp{i}{j}s_{j-1}]+[s_{j}].$$
    Applying Lemma \ref{lemma: rewriting [w|v] for w not a generator} to $[\ramp{i}{j}]$ and $[\ramp{i}{j}s_{j-1}]$ we have:
    \begin{align*}
        \partial_{2}^{F}[s_{j}\mid \ramp{i}{j}]&=s_{j}\sum\limits_{\ell=i}^{j}\ramp{i}{\ell-1}[s_{\ell}]-\ramp{i}{j}[s_{j-1}]-\sum\limits_{\ell=i}^{j}\ramp{i}{\ell-1}[s_{\ell}]+[s_{j}]\\
        &=(s_{j}-1)\sum\limits_{\ell=i}^{j}\ramp{i}{\ell-1}[s_{\ell}]+[s_{j}]-\ramp{i}{j}[s_{j-1}]. 
    \end{align*}
\end{proof}
\begin{remark} \label{remark:Q3gens}
Since there are three kinds of rewriting rules for words of length greater than one, there are $3^{t-1}$ types of essential $t$-cell. Each type of $t$-cell is determined by the rewriting rules applied to the product $g_{i}g_{i+1}$ for all $1\leq i\leq t-1$.

We record the nine classes of essential $3$-cell, giving the general form and the rewriting rules.  We then explain how to express their boundaries using essential $2$-cells in several propositions over the following pages.
\begin{itemize}
\item Using $\rewriting1$ and $\rewriting1$: $[s_i\mid s_i\mid s_i]$.
\item Using $\rewriting1$ and $\rewriting2$: $[s_i\mid s_i\mid s_j]$, with $i<j-1$.
\item Using $\rewriting1$ and $\rewriting3$: $[s_i\mid s_i\mid \ramp{j}{i}]$, with $j<i$.
\item Using $\rewriting2$ and $\rewriting1$: $[s_i\mid s_j\mid s_j]$, with $i<j-1$.
\item Using $\rewriting2$ and $\rewriting2$: $[s_i\mid s_j\mid s_k]$, with $i<j-1$ and $j<k-1$.
\item Using $\rewriting2$ and $\rewriting3$: $[s_i\mid s_j\mid \ramp{k}{j}]$, with  $i<j-1$ and $k<j$.
\item Using $\rewriting3$ and $\rewriting1$: $[s_i\mid \ramp{j}{i}\mid s_i]$, with $j<i$.
\item Using $\rewriting3$ and $\rewriting2$: $[s_i\mid \ramp{j}{i}\mid s_k]$, with $j<i$ and $i<k-1$.
\item Using $\rewriting3$ and $\rewriting3$: $[s_i\mid \ramp{j}{i}\mid \ramp{k}{i}]$, with $j<i$ and $k<i$.
\end{itemize}
\end{remark}
\begin{proposition}\label{prop: boundary of s,s}
    For each $1\leq i\leq n-1$,  $\partial_{3}^{Q}[s_{i}\mid s_{i}\mid s_{i}]=(s_{i}-1)[s_{i}\mid s_{i}]$.
\end{proposition}
\begin{proof}
    We have:
    \begin{align*}
        \partial_{3}^{F}[s_{i}\mid s_{i}\mid s_{i}]&=s_{i}[s_{i}\mid s_{i}]-[s_{i}^{2}\mid s_{i}]+[s_{i}\mid s_{i}^{2}]-[s_{i}\mid s_{i}]\\
        &=(s_{i}-1)[s_{i}\mid s_{i}]
    \end{align*}
    since $s_{i}^{2}\to 1$ by $\rewriting1$ and both $[1\mid s_{i}]$ and $[s_{i}\mid 1]$ are degenerate cells. 
\end{proof}
\begin{proposition}\label{prop: boundary of s,c}
    For each $1\leq i<j-1\leq n-1$, $\partial_{3}^{F}[s_{i}\mid s_{i}\mid s_{j}]=(s_{i}+1)[s_{i}\mid s_{j}]+(s_{j}-1)[s_{i}\mid s_{i}]$.
\end{proposition}
\begin{proof}
    Evaluating the boundary:
    \begin{align*}
        \partial_{3}^{F}[s_{i}\mid s_{i}\mid s_{j}]&=s_{i}[s_{i}\mid s_{j}]-[s_{i}^{2}\mid s_{j}]+[s_{i}\mid s_{i}s_{j}]-[s_{i}\mid s_{i}]\\
        &=s_{i}[s_{i}\mid s_{i}]+[s_{i}\mid s_{j}s_{i}]-[s_{i}\mid s_{i}]
    \end{align*}
    since $s_{i}^{2}\to 1$ and $[1\mid s_{j}]$ is degenerate. Let $\tau=[s_{i}\mid s_{j}s_{i}]$, then $\tau$ is redundant with $[s_{i}\mid s_{j}\mid s_{i}]$ the corresponding 3-cell. Taking $\partial_{3}^{F}c(\tau)$ we obtain:
    $$\tau= [s_{i}\mid s_{j}]+[s_{j}s_{i}\mid s_{i}]-s_{i}[s_{j}\mid s_{i}].$$
    Applying Lemma \ref{Lemma: irreducible tuples quotients to 0} and Lemma \ref{lemma: rewriting [w|v] for w not a generator} we have $\tau=[s_{i}\mid s_{j}]+s_{j}[s_{i}\mid s_{i}]$. 
\end{proof}
\begin{proposition}\label{prop: boundary of s,b}
    For each $1\leq i<j\leq n-1$,  
\[\partial_{3}^{Q}[s_{j}\mid s_{j}\mid \ramp{i}{j}]=(s_{j}+1)[s_{j}\mid \ramp{i}{j}]+\ramp{i}{j}[s_{j-1}\mid s_{j-1}]-[s_{j}\mid s_{j}].\]
\end{proposition}
\begin{proof}
    We evaluate the boundary:
    \begin{align*}
    \partial_{3}^{F}[s_{j}\mid s_{j}\mid \ramp{i}{j}]=&s_{j}[s_{j}\mid \ramp{i}{j}]-[s_{j}^{2}\mid \ramp{i}{j}]+[s_{j}\mid s_{j}\ramp{i}{j}]-[s_{j}\mid s_{j}]\\
    =&s_{j}[s_{j}\mid \ramp{i}{j}]+[s_{j}\mid \ramp{i}{j}s_{j-1}]-[s_{j}\mid s_{j}].
    \end{align*}
    Let $\tau=[s_{j}\mid \ramp{i}{j}s_{j-1}]$; then $\tau$ is redundant with the corresponding collapsible 3-cell $[s_{j}\mid \ramp{i}{j}\mid s_{j-1}]$. Evaluating $\partial c(\tau)$, we have:
    $$\tau= [s_{j}\mid \ramp{i}{j}]+[\ramp{i}{j}s_{j-1}\mid s_{j-1}]-s_{j}[\ramp{i}{j}\mid s_{j-1}].$$
    Applying Lemma \ref{Lemma: irreducible tuples quotients to 0} and Lemma \ref{lemma: rewriting [w|v] for w not a generator}: 
    $$q(\tau)=[s_{j}\mid \ramp{i}{j}]+\ramp{i}{j}[s_{j-1}\mid s_{j-1}]+\sum\limits_{k=0}^{j-i-1}\ramp{i}{i+k-1}q([s_{i+k}\mid \ramp{i+k+1}{j}]).$$
    But $\ramp{i+k}{j}$ is irreducible for all $0\leq k\leq j-i$, therefore $q([s_{i+k}\mid \ramp{i+k+1}{j}])=0$ by Lemma \ref{Lemma: irreducible tuples quotients to 0} and the proposition is complete. 
\end{proof}
\begin{proposition}\label{prop: boundary of c,s}
    For each $1\leq i<j-1\leq n-1$, we have
    \[\partial_{3}^{Q}[s_{i}\mid s_{j}\mid s_{j}]=(s_{i}-1)[s_{j}\mid s_{j}]-(s_{j}+1)[s_{i}\mid s_{j}].\]
\end{proposition}
\begin{proof}
We compute $\partial_{3}^{F}$ and remove degenerate cells after reducing each element of $\sg{n}$ to normal form:
$$\partial_{3}^{F}[s_{i}\mid s_{j}\mid s_{j}]=s_{i}[s_{j}\mid s_{j}]-[s_{j}s_{i}\mid s_{j}]-[s_{i}\mid s_{j}].$$
Applying Lemma \ref{lemma: rewriting [w|v] for w not a generator} to $[s_{j}s_{i}\mid s_{j}]$, we have:
$$q([s_{j}s_{i}\mid s_{j}])=s_{j}[s_{i}\mid s_{j}]+q([s_{j}\mid s_{j}s_{i}]).$$
Let $\tau=[s_{j}\mid s_{j}s_{i}]$, the corresponding collapsible 3-cell is $c(\tau)=[s_{j}\mid s_{j}\mid s_{i}]$. Therefore $\tau= s_{j}[s_{j}\mid s_{i}]+[s_{j}\mid s_{j}]$ and by Lemma \ref{Lemma: irreducible tuples quotients to 0} $q(\tau)=[s_{j}\mid s_{j}]$.
\end{proof}
\begin{proposition}\label{prop: boundary of c,c}
    For each $1\leq i<j-1<k-1\leq n-1$, we have
    \[\partial_{3}^{Q}[s_{i}\mid s_{j}\mid s_{k}]=(s_{i}-1)[s_{j}\mid s_{k}]-(s_{j}-1)[s_{i}\mid s_{k}]+(s_{k}-1)[s_{i}\mid s_{j}].\]
\end{proposition}
\begin{proof}
    We evaluate the boundary and put elements of $\sg{n}$ in normal form:
    $$\partial_{3}^{F}[s_{i}\mid s_{j}\mid s_{k}]=s_{i}[s_{j}\mid s_{k}]-[s_{j}s_{i}\mid s_{k}]+[s_{i}\mid s_{k}s_{j}]-[s_{i}\mid s_{j}].$$
    By Lemma \ref{lemma: rewriting [w|v] for w not a generator}, $q([s_{j}s_{i}\mid s_{k}])=s_{j}[s_{i}\mid s_{k}]+q([s_{j}\mid s_{k}s_{i}])$. Let $\tau=[s_{j}\mid s_{k}s_{i}]$; then $\tau$ is redundant with corresponding collapsible 3-cell $c(\tau)=[s_{j}\mid s_{k}\mid s_{i}]$. Evaluating the boundary we obtain $\tau=[s_{j}\mid s_{k}]+[s_{k}s_{j}\mid s_{i}]-s_{j}[s_{k}\mid s_{i}]$. Applying Lemma \ref{Lemma: irreducible tuples quotients to 0}, $q(\tau)=[s_{j}\mid s_{k}]$. Therefore:
    $$q([s_{j}s_{i}\mid s_{k}])=s_{j}[s_{i}\mid s_{k}]+[s_{j}\mid s_{k}].$$
    Let $\lambda=[s_{i}\mid s_{k}s_{j}]$; then $\lambda$ is redundant with corresponding collapsible 3-cell $c(\lambda)=[s_{i}\mid s_{k}\mid s_{j}]$. Evaluating the boundary, we have $\lambda =[s_{i}\mid s_{k}]+[s_{k}s_{i}\mid s_{j}]-s_{i}[s_{k}\mid s_{j}]$. Applying Lemma \ref{Lemma: irreducible tuples quotients to 0} and Lemma \ref{lemma: rewriting [w|v] for w not a generator} we obtain: 
    $$q(\lambda)=[s_{i}\mid s_{k}]+s_{k}[s_{i}\mid s_{j}].$$
\end{proof}

\begin{lemma}\label{lemma: quotient of braid relation}
    Let $\tau=[s_{i}\mid \ramp{k}{j}w]$ for some $i\leq j-1$. Then:
    $$q(\tau)=\begin{cases}
        \ramp{k}{j}q([s_{i}\mid w])+\sum\limits_{\ell=0}^{j-k}\ramp{k}{k+\ell-1}[s_{i}\mid s_{k+\ell}]\hspace{1em}&i<k-1\\
        0\hspace{1em}&i=k-1\\
        [s_{k}\mid s_{k}]\hspace{1em}&i=k\\
        [s_{i}\mid \ramp{k}{i}]+\ramp{k}{i}q([s_{i-1}\mid \ramp{i+1}{j}w])\hspace{1em}&k<i\leq j-1
    \end{cases}$$
\end{lemma}
\begin{remark}
    In this statement, the term $q([s_{i-1}\mid \ramp{i+1}{j}w])$, from the expression for $q(\tau)$ when $k<i\leq j-1$, is itself a term that the lemma applies to.
In fact, it can be rewritten using the expression from the case $i<k-1$, to get an expression in terms of essential cells.
\end{remark}
\begin{proof}[Proof of Lemma~\ref{lemma: quotient of braid relation}]
    Case: $i<k-1$. We induct on $j-k$. Suppose $j-k=1$, then $c(\tau)=[s_{i}\mid s_{j-1}\mid s_{j}w]$ and applying the rewriting rules:
    $$\partial c(\tau)=s_{i}[s_{j-1}\mid s_{j}w]+[s_{j-1}s_{i}\mid s_{j}w]+\tau-[s_{i}\mid s_{j-1}].$$
    Applying Lemma \ref{Lemma: irreducible tuples quotients to 0}, $q(\tau)=[s_{i}\mid s_{j-1}]+q([s_{j-1}s_{i}\mid s_{j}w])$. By Lemma \ref{lemma: rewriting [w|v] for w not a generator}, $q([s_{j-1}s_{i}\mid s_{j}w])=s_{j-1}[s_{i}\mid s_{j}w]$. Set $\tau'=[s_{i}\mid s_{j}w]$, then $c(\tau')=[s_{i}\mid s_{j}\mid w]$ and:
    $$\partial c(\tau')=s_{i}[s_{j}\mid w]-[s_{j}s_{i}\mid w]+\tau'-[s_{i}\mid s_{j}].$$
    By Lemma \ref{Lemma: irreducible tuples quotients to 0} and Lemma \ref{lemma: rewriting [w|v] for w not a generator}, $q(\tau')=s_{j}q([s_{i}\mid w])+[s_{i}\mid s_{j}]$. Thus:
    $$q(\tau)=[s_{i}\mid s_{j-1}]+s_{j-1}[s_{i}\mid s_{j}]+s_{j-1}s_{j}q([s_{i}\mid w]).$$

    Now, suppose $j-k>1$, then $c(\tau)=[s_{i}\mid s_{k}\mid \ramp{k+1}{j}w]$ and applying the rewriting rules:
    $$\partial c(\tau)=s_{i}[s_{k}\mid \ramp{k+1}{j}w]-[s_{k}s_{i}\mid \ramp{k+1}{j}w]+\tau-[s_{i}\mid s_{k}].$$
    By Lemma \ref{Lemma: irreducible tuples quotients to 0} and Lemma \ref{lemma: rewriting [w|v] for w not a generator}:
    $$q(\tau)=[s_{i}\mid s_{k}]+s_{k}q([s_{i}\mid \ramp{k+1}{j}w]).$$
    Applying the inductive hypothesis to $[s_{i}\mid \ramp{k+1}{j}]$ yields:
    \begin{align*}
        q(\tau)&=[s_{i}\mid s_{k}]+s_{k}\Bigg(\ramp{k+1}{j}q([s_{i}\mid w])\\
&\qquad+\sum\limits_{\ell=0}^{j-(k+1)}\ramp{k+1}{(k+1)+\ell-1}[s_{i}\mid s_{k+1+\ell}]\Bigg)\\
        &=[s_{i}\mid s_{k}]+\ramp{k}{j}q([s_{i}\mid w])+s_{k}[s_{i}\mid s_{k+1}]\\
&\qquad+s_{k}s_{k+1}[s_{i}\mid s_{k+2}]+\cdots+\ramp{k}{j}[s_{i}\mid s_{j}]\\
        &=\ramp{k}{j}q(s_{i}\mid w])+\sum\limits_{\ell=0}^{j-k}\ramp{k}{k+\ell-1}[s_{i}\mid s_{k+\ell}].
    \end{align*}

    Case: $i=k-1$. If $i=k-1$, then $s_{i}\ramp{k}{j}w=\ramp{i}{j}w$ is irreducible and by Lemma \ref{Lemma: irreducible tuples quotients to 0}, $q(\tau)=0$.

    Case: $i=k$. Suppose $i=k$, then $c(\tau)=[s_{k}\mid s_{k}\mid \ramp{k+1}{j}w]$ and applying the rewriting rules:
    $$\partial c(\tau)=s_{k}[s_{k}\mid \ramp{k+1}{j}w]-[1\mid \ramp{k+1}{j}w]+\tau-[s_{k}\mid s_{k}].$$
    Since $[1\mid \ramp{k+1}{j}w]$ is degenerate and by Lemma \ref{Lemma: irreducible tuples quotients to 0}, $q(\tau)=[s_{k}\mid s_{k}]$.

    Case: $k<i\leq j-1$. Then $c(\tau)=[s_{i}\mid \ramp{k}{i}\mid \ramp{i+1}{j}w]$ and applying the rewriting rules:
    $$\partial c(\tau)=s_{i}[\ramp{k}{i}\mid \ramp{i+1}{j}w]-[\ramp{k}{i}s_{i-1}\mid \ramp{i+1}{j}w]+\tau-[s_{i}\mid \ramp{k}{i}].$$
    By Lemma \ref{Lemma: irreducible tuples quotients to 0}, $q(\tau)=[s_{i}\mid \ramp{k}{i}]+q([\ramp{k}{i}s_{i-1}\mid \ramp{i+1}{j}w])$. Lemma \ref{lemma: rewriting [w|v] for w not a generator} completes the proof.
\end{proof}

\begin{proposition}\label{prop: boundary of c,b}
    Let $1\leq i<j-1$ and $1\leq k<j$.  Then $\partial_{3}^{Q}[s_{i}\mid s_{j}\mid \ramp{k}{j}]$ is:
    \begin{align*}
        i<k-1:\hspace{1cm}&(s_{i}-1)[s_{j}\mid \ramp{k}{j}]-[s_{i}\mid s_{j}]+\ramp{k}{j}[s_{i}\mid s_{j-1}]\\&+(1-s_{j})\sum\limits_{\ell=0}^{j-k}\ramp{k}{k+\ell-1}[s_{i}\mid s_{k+\ell}]\\
        i=k-1:\hspace{1cm}&s_{k-1}[s_{j}\mid \ramp{k}{j}]-[s_{j}\mid \ramp{k-1}{j}]-[s_{k-1}\mid s_{j}]\\
        i=k:\hspace{1cm}&s_{k}[s_{j}\mid \ramp{k}{j}]+(1-s_{j})[s_{k}\mid s_{k}]-[s_{j}\mid \ramp{k+1}{j}]-[s_{k}\mid s_{j}]\\
        i\geq k-1:\hspace{1cm}&(s_{i}-1)[s_{j}\mid \ramp{k}{j}]+(1-s_{j})[s_{i}\mid \ramp{k}{i}]+\ramp{k}{i}[s_{i-1}\mid s_{j-1}]\\
        &+(1-s_{j})\ramp{k}{i}\sum\limits_{\ell=0}^{j-i-1}\ramp{i+1}{i+\ell}[s_{i-1}\mid s_{i+1+\ell}].
    \end{align*}
\end{proposition}
\begin{proof}
    We have:    $$\partial_{3}^{F}[s_{i}\mid s_{j}\mid \ramp{k}{j}]=s_{i}[s_{j}\mid \ramp{k}{j}]-[s_{j}s_{i}\mid \ramp{k}{j}]+[s_{i}\mid \ramp{k}{j}s_{j-1}]-[s_{i}\mid s_{j}].$$
    Let $\tau=[s_{j}s_{i}\mid \ramp{k}{j}]$. 
By Lemma \ref{lemma: rewriting [w|v] for w not a generator}:
    $$q(\tau)=s_{j}q([s_{i}\mid \ramp{k}{j}])+q([s_{j}\mid s_{i}\ramp{k}{j}])$$ is determined by four cases.
    
    Case: $i<k-1$. We apply Lemma~\ref{lemma: quotient of braid relation} to $[s_{i}\mid \ramp{k}{j}]$, $[s_{j}\mid \ramp{k}{j}s_{i}]$, and $[s_{i}\mid \ramp{k}{j}s_{j-1}]$, and simplify.  This completes the case.

    Case: $i=k-1$. By Lemma~\ref{Lemma: irreducible tuples quotients to 0}, $q([s_{i}\mid \ramp{k}{j}])=0$ and $q(\tau)=[s_{j}\mid \ramp{k-1}{j}]$. Applying Lemma~\ref{lemma: quotient of braid relation} to $[s_{i}\mid \ramp{k}{j}s_{j-1}]$ completes the case.

    Case: $i=k$. By applying Lemma~\ref{lemma: quotient of braid relation} to $[s_{i}\mid \ramp{k}{j}]$, we have $q(\tau)=s_{j}[s_{k}\mid s_{k}]+[s_{j}\mid \ramp{k+1}{j}]$. 
Applying Lemma~\ref{lemma: quotient of braid relation} to $q([s_{i}\mid \ramp{k}{j}s_{j-1}])$ completes the case.

    Case: $i>k$. We apply Lemma~\ref{lemma: quotient of braid relation} to $[s_{i}\mid \ramp{k}{j}]$, $[s_{j}\mid \ramp{k}{j}s_{i-1}]$, and $[s_{i}\mid \ramp{k}{j}s_{j-1}]$, and simplify.
\end{proof}
\begin{proposition}\label{prop: boundary of b,s}
    Suppose $i<j$, then:
\[
\begin{split}
\partial_{3}^{Q}[s_{j}\mid \ramp{i}{j}\mid s_{j}]=
&s_{j}\ramp{i}{j-1}[s_{j}\mid s_{j}]-\ramp{i}{j-1}[s_{j}\mid s_{j-1}s_{j}]
 \\
&-\sum_{k=i-1}^{j-3}\ramp{i}{k}[s_{k+1}\mid s_j] -[s_{j}\mid \ramp{i}{j}] \\
&-\ramp{i}{j-2}[s_{j-1}\mid s_{j-1}] .
\end{split}
\]
\end{proposition}
\begin{proof}
    Since $i<j$, the faces of $[s_{j}\mid \ramp{i}{j}\mid s_{j}]$ have reductions by $\rewriting3$ and $\rewriting1$, and so it is essential with:
    $$\partial_{3}^{F}[s_{j}\mid \ramp{i}{j}\mid s_{j}]=s_{j}[\ramp{i}{j}\mid s_{j}]-[\ramp{i}{j}s_{j-1}\mid s_{j}]+[s_{j}\mid \ramp{i}{j-1}]-[s_{j}\mid \ramp{i}{j}].$$
    Since $s_{j}\ramp{i}{j-1}$ is irreducible, $q([s_{j}\mid\ramp{i}{j-1} ])=0$ by Lemma \ref{Lemma: irreducible tuples quotients to 0}. 
By Lemma~\ref{lemma: rewriting [w|v] for w not a generator}, $q([\ramp{i}{j}\mid s_{j}])=\ramp{i}{j-1}[s_{j}\mid s_{j}]$.
We also apply Lemma~\ref{lemma: rewriting [w|v] for w not a generator} to $[\ramp{i}{j}s_{j-1}\mid s_{j}]$; we get
\[
\begin{split}
q([\ramp{i}{j}s_{j-1}\mid s_{j}])=&\ramp{i}{j-1}q([s_{j}s_{j-1}\mid s_{j}]) \\
&+\sum_{k=i-1}^{j-3}\ramp{i}{k}q([s_{k+1}\mid s_j\ramp{k+2}{j-1}]) \\
&+\ramp{i}{j-2}q([s_{j-1}\mid s_{j-1}s_js_{j-1}]).
\end{split}
\]
First let $\tau=[s_{j}s_{j-1}\mid s_{j}]$; then $c(\tau)=[s_{j}\mid s_{j-1}\mid s_{j}]$ and:
    $$\partial c(\tau)=s_{j}[s_{j-1}\mid s_{j}]-\tau+[s_{j}\mid s_{j-1}s_{j}]-[s_{j}\mid s_{j-1}].$$
    By Lemma \ref{Lemma: irreducible tuples quotients to 0}, $q(\tau)=[s_{j}\mid s_{j-1}s_{j}]$.
Next, fix $k$ with $i-1\leq k\leq j-3$, and set $\tau=[s_{k+1}\mid s_j\ramp{k+2}{j-1}]$.
Then $c(\tau)=[s_{k+1}\mid s_j\mid \ramp{k+2}{j-1}]$, and
\[
\begin{split}
\partial c(\tau) =& s_{k+1}[s_j\mid\ramp{k+2}{j-1}]-[s_js_{k+1}\mid\ramp{k+2}{j-1}]\\
&+[s_{k+1}\mid s_j\ramp{k+2}{j-1}]-[s_{k+1}\mid s_j].
\end{split}
\]
By Lemma~\ref{lemma: rewriting [w|v] for w not a generator}, $q(\tau)=[s_{k+1}\mid s_j]$.
Finally, set $\tau=[s_{j-1}\mid s_{j-1}s_js_{j-1}]$.
Then $c(\tau)=[s_{j-1}\mid s_{j-1}\mid s_js_{j-1}]$ and
\[\partial c(\tau)=s_{j-1}[s_{j-1}\mid s_js_{j-1}]-[1\mid s_js_{j-1}]+[s_{j-1}\mid s_{j-1}s_js_{j-1}]-[s_{j-1}\mid s_{j-1}].\]
Again by Lemma~\ref{lemma: rewriting [w|v] for w not a generator}, we get $q(\tau)=[s_{j-1}\mid s_{j-1}]$.
This finishes the lemma.
\end{proof}
\begin{proposition}\label{prop: boundary of b,c}
    If $k<i<j-1$, then:
    \begin{align*}
    \partial_{3}^{Q}[s_{i}\mid \ramp{k}{i}\mid s_{j}]=&[s_{i}\mid s_{j}]+(s_{j}-1)[s_{i}\mid \ramp{k}{i}]-\ramp{k}{i}[s_{i-1}\mid s_{j}]\\
    &+(s_{i}-1)\sum\limits_{\ell=0}^{i-k}\ramp{k}{k+\ell-1}[s_{k+\ell}\mid s_{j}]
    \end{align*}
\end{proposition}
\begin{proof}
    The boundary is:
    $$\partial_{3}^{F}[s_{i}\mid \ramp{k}{i}\mid s_{j}]=s_{i}[\ramp{k}{i}\mid s_{j}]-[\ramp{k}{i}s_{i-1}\mid s_{j}]+[s_{i}\mid s_{j}\ramp{k}{i}]-[s_{i}\mid \ramp{k}{i}].$$
    Applying Lemma \ref{lemma: rewriting [w|v] for w not a generator}, we have
\[q([\ramp{k}{i}\mid s_{j}])=\sum\limits_{\ell=0}^{i-k}\ramp{k}{k+\ell-1}q([s_{k+\ell}\mid s_{j}\ramp{k+\ell+1}{i}]),\quad\text{ and }
\]
\[
\begin{split}
q([\ramp{k}{i}s_{i-1}\mid s_{j}])=&\ramp{k}{i}[s_{i-1}\mid s_{j}] \\
&+\sum\limits_{\ell=0}^{i-k}\ramp{k}{k+\ell-1}q([s_{k+\ell}\mid s_{j}\ramp{k+\ell+1}{i}]).
\end{split}
\]
    Let $\tau=[s_{k+\ell}\mid s_{j}\ramp{k+\ell+1}{i}]$; then $c(\tau)=[s_{k+\ell}\mid s_{j}\mid \ramp{k+\ell+1}{i}]$ and:
    $$\partial c(\tau)=s_{k+\ell}[s_{j}\mid \ramp{k+\ell+1}{i}]-[s_{j}s_{k+\ell}\mid \ramp{k+\ell+1}{i}]+\tau-[s_{k+\ell}\mid s_{j}].$$
    By Lemma \ref{Lemma: irreducible tuples quotients to 0}, $q(\tau)=[s_{k+\ell}\mid s_{j}]$. Now, let $\lambda=[s_{i}\mid s_{j}\ramp{k}{i}]$. Then $c(\lambda)=[s_{i}\mid s_{j}\mid \ramp{k}{i}]$ and:
    $$\partial c(\lambda)=s_{i}[s_{j}\mid \ramp{k}{i}]-[s_{j}s_{i}\mid \ramp{k}{i}]+\lambda-[s_{i}\mid s_{j}].$$
    By Lemma \ref{Lemma: irreducible tuples quotients to 0} and Lemma \ref{lemma: rewriting [w|v] for w not a generator}, $q(\lambda)=[s_{i}\mid s_{j}]+s_{j}[s_{i}\mid \ramp{k}{i}]+q([s_{j}\mid \ramp{k}{i}s_{i-1}])$. Since $s_{j}\ramp{k}{i}s_{i-1}$ is irreducible, by Lemma \ref{Lemma: irreducible tuples quotients to 0} $q([s_{j}\mid \ramp{k}{i}s_{i-1}])=0$. Simplifying $q\circ\partial_{3}^{F}[s_{i}\mid \ramp{k}{i}\mid s_{j}]$ completes the proof.
\end{proof}
We have described the boundaries of eight of the nine classes of essential $3$-cells.
To determine the boundary of the last kind of essential $3$-cell, we must make more careful applications of the rewriting system. 
Consider the reducible word $\ramp{i}{j}\ramp{k}{j}$.
Applying $\rewriting3$, we may reduce: $$\ramp{i}{j}\ramp{k}{j}\to \ramp{i}{j-1}\ramp{k}{j}s_{j-1}.$$
Applying $\rewriting3$ on the subword $s_{j-1}\ramp{k}{j-1}$ and applying $\rewriting2$ to the resulting word, we obtain:
$$\ramp{i}{j-1}\ramp{k}{j}s_{j-1}\to \ramp{i}{j-2}\ramp{k}{j}s_{j-2}s_{j-1}.$$
Continuing this, we may reduce $\ramp{i}{j}\ramp{k}{j}s_{j-1}$ to $\ramp{k}{j}\ramp{i-1}{j-1}$ if $i>k$. 

Now, if $i\leq k$ we may use the previous case to reduce $\ramp{i}{j}\ramp{k}{j}$ to 
\[\ramp{i}{k}\ramp{k}{j}\ramp{k}{j-1}.\]
 Applying $\rewriting1$ followed by $\rewriting2$ repeatedly, we reduce this to $\ramp{k+1}{j}\ramp{i}{j-1}$. 
\begin{lemma}\label{lemma: first face of bb}
    Let $1\leq i,k<j$ and $\tau=[\ramp{i}{j}\mid \ramp{k}{j}]$. 
If $i>k$, then $q(\tau)$ is:
\[
\begin{split}
        &\sum\limits_{\ell=0}^{j-i}\ramp{i}{i+\ell-1}[s_{i+\ell}\mid \ramp{k}{i+\ell}]\\
        &+\sum\limits_{\ell=0}^{j-i-1}\ramp{i}{i+\ell-1}\ramp{k}{i+\ell}\sum\limits_{t=1}^{j-i-\ell}\ramp{i+\ell+1}{i+\ell+t-1}[s_{i+\ell-1}\mid s_{i+\ell+t}].\\
\end{split}
\]
If $i\leq k$, then $q(\tau)$ is:
\[
\begin{split}
&\sum\limits_{\ell=k-i}^{j-i}\ramp{i}{i+\ell-1}[s_{i+\ell}\mid \ramp{k}{i+\ell}]\\&+\sum\limits_{\ell=0}^{k-i-1}\ramp{i}{i+\ell-1}\sum\limits_{t=1}^{j-k}\ramp{k+1}{k+t-1}[s_{i+\ell}\mid s_{k+t}]\\
        &+\sum\limits_{\ell=k-i+1}^{j-i-1}\ramp{i}{i+\ell-1}\ramp{k}{i+\ell}\sum\limits_{t=1}^{j-i-\ell}\ramp{i+\ell+1}{i+\ell+t-1}[s_{i+\ell-1}\mid s_{i+\ell+t}].
\end{split}
\]
\end{lemma}
\begin{proof}
    By Lemma \ref{lemma: rewriting [w|v] for w not a generator}:
    $$q(\tau)=\ramp{i}{j-1}[s_{j}\mid \ramp{k}{j}]+\sum\limits_{\ell=0}^{j-i-1}\ramp{i}{i+\ell-1}q([s_{i+\ell}\mid \ramp{i+\ell+1}{j}\ramp{k}{j}]).$$
    For each $\ell$ with $0\leq \ell \leq j-i-1$, let $\tau_{\ell}=[s_{i+\ell}\mid \ramp{i+\ell+1}{j}\ramp{k}{j}]$. The reduction of $\ramp{i+\ell+1}{j}\ramp{k}{j}$ yields two cases.
    
    Case: $i> k$. Then for each $\ell$, since $i+\ell+1>k$, we have $\tau_{\ell}=[s_{i+\ell}\mid \ramp{k}{j}\ramp{i+\ell}{j-1}]$. 
Therefore $c(\tau_{\ell})=[s_{i+\ell}\mid \ramp{k}{i+\ell}\mid \ramp{i+\ell+1}{j}\ramp{i+\ell}{j-1}]$, and:
    \begin{align*}
        \partial c(\tau_{\ell})=&s_{i+\ell}[\ramp{k}{i+\ell}\mid \ramp{i+\ell+1}{j}\ramp{i+\ell}{j-1}]\\
&-[\ramp{k}{i+\ell}s_{i+\ell-1}\mid \ramp{i+\ell+1}{j}\ramp{i+\ell}{j-1}]\\
        &+\tau_{\ell}-[s_{i+\ell}\mid \ramp{k}{i+\ell}].
    \end{align*}
    Therefore $q(\tau_{\ell})=[s_{i+\ell}\mid \ramp{k}{i+\ell}]+q(\tau_{\ell}')$ where 
\[\tau_{\ell}'=[\ramp{k}{i+\ell}s_{i+\ell-1}\mid \ramp{i+\ell+1}{j}\ramp{i+\ell}{j-1}].\]
 Applying Lemmas \ref{Lemma: irreducible tuples quotients to 0}, \ref{lemma: rewriting [w|v] for w not a generator} and~\ref{lemma: quotient of braid relation}:
    \begin{align*}
    q(\tau_{\ell}')=&\ramp{k}{i+\ell}q([s_{i+\ell-1}\mid \ramp{i+\ell+1}{j}s_{i+\ell}s_{j-1}]\\
    &+\sum\limits_{t=0}^{i+\ell-k}\ramp{k}{k+t-1}[s_{k+t}\mid \ramp{k+t+1}{j}\ramp{i+\ell-1}{j-1}]\\
    =&\ramp{k}{i+\ell}\sum\limits_{t=0}^{j-i-\ell-1}\ramp{i+\ell+1}{i+\ell+t}[s_{i+\ell-1}\mid s_{i+\ell+1+t}]\\
    =&\ramp{k}{i+\ell}\sum\limits_{t=1}^{j-i-\ell}\ramp{i+\ell+1}{i+\ell+t-1}[s_{i+\ell-1}\mid s_{i+\ell+t}].
    \end{align*}
    Therefore:
    \begin{align*}
        q(\tau)=&\ramp{i}{j-1}[s_{j}\mid \ramp{k}{j}]+\sum\limits_{\ell=0}^{j-i-1}\ramp{i}{i+\ell-1}\Bigg([s_{i+\ell}\mid \ramp{k}{i+\ell}]\\
        &+\ramp{k}{i+\ell}\sum\limits_{t=1}^{j-i-\ell}\ramp{i+\ell+1}{i+\ell+t-1}[s_{i+\ell-1}\mid s_{i+\ell+t}]\Bigg),\\
    \end{align*}
and this simplifies to the expression in the statement.

    Case: $i\leq k$. For those $\ell$ with $\ell<k-i$, we have $\tau_{\ell}=[s_{i+\ell}\mid \ramp{k+1}{j}\ramp{i+\ell+1}{j-1}]$. Applying Lemmas \ref{Lemma: irreducible tuples quotients to 0} and~\ref{lemma: quotient of braid relation}, we have:
    \begin{align*}
    q(\tau_{\ell})=&\sum\limits_{t=0}^{j-k-1}\ramp{k+1}{k+t}[s_{i+\ell}\mid s_{k+t+1}]
    =\sum\limits_{t=1}^{j-k}\ramp{k+1}{k+t-1}[s_{i+\ell}\mid s_{k+t}]
    \end{align*}
    For $\ell=k-i$, $\tau_{\ell}=[s_{k}\mid \ramp{k}{j}\ramp{i+\ell}{j-1}$ and $q(\tau_{k-i})=[s_{k}\mid s_{k}]$ by Lemma~\ref{lemma: quotient of braid relation}. 
For those $\ell$ with $\ell>k-i$, the same formula works as in the previous case.
Taking the sum from $\ell=0$ to $j-i-1$ yields the result. 
\end{proof}
\begin{lemma}\label{lemma: second face of bb}
    Suppose $1\leq i,k<j$ and $\tau=[\ramp{i}{j}s_{j-1}\mid \ramp{k}{j}]$. 
If $i> k$, then $q(\tau)$ equals
\[
\begin{split}
&\ramp{i}{j}[s_{j-1}\mid \ramp{k}{j-1}]+\ramp{k}{j-1}[s_{j-2}\mid s_{j}]\\
&+\sum\limits_{\ell=0}^{j-i}\ramp{i}{i+\ell-1}[s_{i+\ell}\mid \ramp{k}{i+\ell}]\\
       &+\sum\limits_{\ell=0}^{j-i}\ramp{i}{i+\ell-1}\ramp{k}{i+\ell}\sum\limits_{t=1}^{j-i-\ell}\ramp{i+\ell+1}{i+\ell+t-1}[s_{i+\ell-1}\mid s_{i+\ell+t}].\\
\end{split}
\]
If $i\leq k$, then $q(\tau)$ equals
\[
\begin{split}
&\ramp{i}{j}[s_{j-1}\mid \ramp{k}{j-1}]+\ramp{i}{j}\ramp{k}{j-1}[s_{j-2}\mid s_{j}]\\
       &+\sum\limits_{\ell=0}^{k-i-1}\ramp{i}{i+\ell-1}\sum\limits_{t=1}^{j-k}\ramp{k+1}{k+t-1}[s_{i+\ell}\mid s_{k+t}]\\
&+\sum\limits_{\ell=k-i}^{j-i}\ramp{i}{i+\ell-1}[s_{i+\ell}\mid \ramp{k}{i+\ell}]\\
       &+\sum\limits_{\ell=k-i+1}^{j-i}\ramp{i}{i+\ell-1}\ramp{k}{i+\ell}\sum\limits_{t=1}^{j-i-\ell}\ramp{i+\ell+1}{i+\ell+t-1}[s_{i+\ell-1}\mid s_{i+\ell+t}].
\end{split}
\]
\end{lemma}
\begin{proof}
    By Lemma \ref{lemma: rewriting [w|v] for w not a generator}:
\[
\begin{split}
q(\tau)=&\ramp{i}{j}q([s_{j-1}\mid \ramp{k}{j}])\\
&+\sum\limits_{\ell=0}^{j-i}\ramp{i}{i+\ell-1}q([s_{i+\ell}\mid \ramp{i+\ell+1}{j}s_{j-1}\ramp{k}{j}]).
\end{split}
\]
    Let $\tau'=[s_{j-1}\mid \ramp{k}{j}]$; then $c(\tau')=[s_{j-1}\mid \ramp{k}{j-1}\mid s_{j}]$ with:
    $$\partial c(\tau')=s_{j-1}[\ramp{k}{j-1}\mid s_{j}]-[\ramp{k}{j-1}s_{j-2}\mid s_{j}]+\tau'-[s_{j-1}\mid \ramp{k}{j-1}].$$
    Applying Lemma \ref{lemma: rewriting [w|v] for w not a generator}, $q(\tau')=[s_{j-1}\mid \ramp{k}{j-1}]+\ramp{k}{j-1}[s_{j-2}\mid s_{j}]$. 
Now, for each $\ell$ with $0\leq \ell\leq j-i$, let $\tau_{\ell}=[s_{i+\ell}\mid \ramp{i+\ell+1}{j}s_{j-1}\ramp{k}{j}]$. The reduction of $\ramp{i+\ell+1}{j}s_{j-1}\ramp{k}{j}$ yields two cases.
    
    Case: $i\geq k.$ For each $\ell>0$, $i+\ell+1>k$ and $\tau_{\ell}=[s_{i+\ell}\mid \ramp{k}{j}\ramp{i+\ell}{j-1}s_{j-2}]$. By Lemma~\ref{lemma: quotient of braid relation}:
\[
q(\tau_{\ell})=[s_{i+\ell}\mid \ramp{k}{i+\ell}]+\ramp{k}{i+\ell}q([s_{i+\ell-1}\mid \ramp{i+\ell+1}{j}\ramp{i+\ell}{j-1}s_{j-2}]).
\]
We apply Lemma~\ref{lemma: quotient of braid relation} again to this second term, apply Lemma~\ref{Lemma: irreducible tuples quotients to 0} and reindex the sum:
\[
\begin{split}
q([s_{i+\ell-1}&\mid \ramp{i+\ell+1}{j}\ramp{i+\ell}{j-1}s_{j-2}])\\
&=\ramp{i+\ell+1}{j}q([s_{i+\ell-1}|\ramp{i+\ell}{j-s}s_{j-2}])\\
&\qquad+\sum\limits_{t=0}^{j-i-\ell-1}\ramp{i+\ell+1}{i+\ell+t}[s_{i+\ell-1}\mid s_{i+\ell+t+1}]\\
&=0+\sum\limits_{t=1}^{j-i-\ell}\ramp{i+\ell+1}{i+\ell+t-1}[s_{i+\ell-1}\mid s_{i+\ell+t}].
\end{split}
\]
This proves the lemma in this case.

    Case: $i\leq k$. For those $\ell$ with $\ell<k-i-1$, we have 
\[\tau_{\ell}=[s_{i+\ell}\mid \ramp{k+1}{j} \ramp{i+\ell+1}{j-1}s_{j-2}].\]
 Note that $(k+1)-(i+\ell)\geq 2$; applying Lemma~\ref{lemma: quotient of braid relation} for $\ell<k-i-1$ yields:
\[
q(\tau_{\ell})=\sum\limits_{t=0}^{j-k-1}\ramp{k+1}{k+t}[s_{i+\ell}\mid s_{k+1+t}]
=\sum\limits_{t=1}^{j-k}\ramp{k+1}{k+t-1}[s_{i+\ell}\mid s_{k+t}]
\]
    Furthermore, applying Lemma \ref{lemma: quotient of braid relation} to the case $\ell=k-i$ yields $q(\tau_{k-i})=[s_{k}\mid s_{k}]$. 
We use the formula from the previous case for those $\ell$ with $\ell\geq k-i+1$, and we take the sums of $\tau_{\ell}$ from $\ell=0$ to $j-i$ to obtain the result.
\end{proof}
\begin{proposition}\label{prop: boundary of bb}
    Suppose $1\leq i,k<j$. 
Then $\partial_{3}^{Q}[s_{j}\mid \ramp{i}{j}\mid \ramp{k}{j}] $ determined by the following cases.
If $i>k$, it is
\[
\begin{split}
 &[s_{j}\mid \ramp{k}{j}]+\ramp{k}{j}[s_{j-1}\mid \ramp{i-1}{j-1}]-\ramp{i}{j}[s_{j-1}\mid \ramp{k}{j-1}]\\
        &-\ramp{k}{j-1}[s_{j-2}\mid s_{j}]-[s_{j}\mid \ramp{i}{j}] \\
&+(s_{j}-1)\Bigg(\sum\limits_{\ell=0}^{j-i}\ramp{i}{i+\ell-1}[s_{i+\ell}\mid \ramp{k}{i+\ell}]\\
        &+\sum\limits _{\ell=0}^{j-i-\ell}\ramp{i}{i+\ell-1}\ramp{k}{i+\ell}\sum\limits_{t=1}^{j-i-\ell}\ramp{i+\ell+1}{i+\ell+t-1}[s_{i+\ell-1}\mid s_{i+\ell+t}]\Bigg).\\
\end{split}
\]
If $i\leq k$, it is
\[
\begin{split}
        & [s_{j}\mid \ramp{k+1}{j}]+\ramp{k+1}{j}[s_{j-1}\mid \ramp{i}{j-1}]-\ramp{i}{j}[s_{j-1}\mid \ramp{k}{j-1}]\\
        &-\ramp{i}{j}\ramp{k}{j-1}[s_{j-2}\mid s_{j}]-[s_{j}\mid \ramp{i}{j}] \\
&+(s_{j}-1)\Bigg(\sum\limits_{\ell=k-i}^{j-i}\ramp{i}{i+\ell}[s_{i+\ell}\mid \ramp{k}{i+\ell}]\\
        &+\sum\limits_{\ell=0}^{k-i-1}\ramp{i}{i+\ell-1}\sum\limits_{t=1}^{j-k}\ramp{k+1}{k+t-1}[s_{i+\ell}\mid s_{k+t}]\\
        &+\sum\limits_{\ell=k-i+1}^{j-i}\ramp{i}{i+\ell-1}\ramp{k}{i+\ell}\sum\limits_{t=1}^{j-i-\ell}\ramp{i+\ell+1}{i+\ell+t-1}[s_{i+\ell-1}\mid s_{i+\ell+t}]\Bigg).
\end{split}
\]
\end{proposition}
\begin{proof}
    We have:
    \begin{align*}
        \partial_{3}^{F}[s_{j}\mid \ramp{i}{j}\mid \ramp{k}{j}]=&s_{j}[\ramp{i}{j}\mid \ramp{k}{j}]-[\ramp{i}{j}s_{j-1}\mid \ramp{k}{j}]\\
        &+[s_{j}\mid \ramp{i}{j}\ramp{k}{j}]-[s_{j}\mid \ramp{i}{j}].
    \end{align*}
    Let $\tau=[s_{j}\mid \ramp{i}{j}\ramp{k}{j}]$. The reduction of $\ramp{i}{j}\ramp{k}{j}$ yields two cases. 
    
    Case: $i>k$. Since $i>k$, $\tau=[s_{j}\mid \ramp{k}{j}\ramp{i-1}{j-1}]$ and $c(\tau)=[s_{j}\mid \ramp{k}{j}\mid \ramp{i-1}{j-1}]$ with:
    $$\partial_{3}^{F} c(\tau)=s_{j}[\ramp{k}{j}\mid \ramp{i-1}{j-1}]-[\ramp{k}{j}s_{j-1}\mid \ramp{i-1}{j-1}]+\tau-[s_{j}\mid \ramp{k}{j}].$$
    Therefore, by Lemmas \ref{Lemma: irreducible tuples quotients to 0} and \ref{lemma: rewriting [w|v] for w not a generator}, $q(\tau)=[s_{j}\mid \ramp{k}{j}]+\ramp{k}{j}[s_{j-1}\mid \ramp{i-1}{j-1}]$. Applying Lemmas \ref{lemma: first face of bb} and \ref{lemma: second face of bb} completes the case.

    Case: $i\leq k$. Since $i\leq k$, $\tau=[s_{j}\mid \ramp{k+1}{j}\ramp{i}{j-1}]$ and $c(\tau)=[s_{j}\mid \ramp{k+1}{j}\mid \ramp{i}{j-1}]$ with:
    $$\partial c(\tau)=s_{j}[\ramp{k+1}{j}\mid \ramp{i}{j-1}]-[\ramp{k+1}{j}s_{j-1}\mid \ramp{i}{j-1}]+\tau-[s_{j}\mid \ramp{k+1}{j}].$$
    Therefore, by Lemmas \ref{Lemma: irreducible tuples quotients to 0} and \ref{lemma: rewriting [w|v] for w not a generator}, $q(\tau)=[s_{j}\mid \ramp{k+1}{j}]+\ramp{k+1}{j}[s_{j-1}\mid \ramp{i}{j-1}]$. Applying Lemmas \ref{lemma: first face of bb} and \ref{lemma: second face of bb} completes the proof.
\end{proof}

\subsection{A chain map and the proof of the main theorem}
To prove Theorem~\ref{th:main}, we build a chain map from $Q_*$ to $P_*$, up to dimension $3$.
We define $\psi\co Q_*\to P_*$ as follows in dimensions $0$, $1$, and $2$.
(It is only necessary to define $\psi$ on essential cells since these freely generate $Q_*$.)
\begin{itemize}
\item We have $\psi_0([\hspace{.5em}])=*$.
\item For each $i$, we have $\psi_1([s_i])=e_i$.
\item For each $i$, we have $\psi_2([s_i|s_i])=c_i$.
\item For all $i$ and $j$ with $i\leq j-2$, we have $\psi_2([s_i|s_j])=-d_{ij}$.
\item For all $i$ and $j$ with $i<j$, we have 
\[\psi_2([s_j|\ramp{i}{j}])=
\sum_{\ell=i}^{j-2}\ramp{i}{\ell-1} d_{\ell j}
+\ramp{i}{j-2}b_{j-1}.\]
\end{itemize}

We verify in Lemma~\ref{le:lowdimchainmap} below that $\psi$ respects boundaries in dimensions $0$, $1$, and $2$.
To show that such a map exists in dimension $3$, it is enough to show that, for each essential cell $\tau\in Q_3$, the image $\psi_2(\partial\tau)$ in $P_2$ is a boundary.
The definition of $\psi\co Q_3\to P_3$ could then be made by sending $\tau$ to any chain that bounds $\psi(\partial\tau)$, for all essential cells $\tau$.
We show that the nine classes of essential cells in $Q_3$ satisfy this in a series of propositions.

\begin{lemma}\label{le:lowdimchainmap}
For each  essential cell $\tau$ in $Q_i$ for $i\in\{1,2\}$, we have 
$\partial^P_i \psi_i(\tau)=\psi_{i-1}(\partial^Q_i\tau)$.
\end{lemma}

\begin{proof}
First of all, for each $i$ from $1$ to $n-1$, we have 
$\partial^P_1 \psi_1([s_i])=\partial^P_1 e_i=(s_i-1)*,$
and $\psi_0(\partial^Q_1[s_i])=\psi_0((s_i-1)[\hspace{0.5em}])=(s_i-1)*$.

Now we consider the map on $2$-cells.
Applying Proposition \ref{prop: essential 1 and 2 cells} and the definition of $\psi_{1}$, for all $1\leq i\leq n-1$ we have $\psi_{1}\circ\partial_{2}^{Q}[s_{i}\mid s_{i}]=(s_{i}+1)e_{i}=\partial_{2}^{P}c_{i}$. 
Now, for all $1\leq i<j\leq n-1$ with $i\leq j-2$, applying Proposition \ref{prop: essential 1 and 2 cells} and the definition of $\psi_{1}$ we have $\psi_{1}\circ\partial_{2}^{Q}[s_{i}\mid s_{j}]=(s_{i}-1)e_{j}-(s_{j}-1)[s_{i}]=-\partial_{2}^{P}d_{ij}=\partial_{2}^{P}\circ\psi_{2}[s_{i}\mid s_{j}]$. 
It remains to show $\partial_{2}^{P}\circ \psi_{2}[s_{j}\mid \ramp{i}{j}]=\psi_{1}\circ\partial_{2}^{Q}[s_{j}\mid \ramp{i}{j}]$ for all $1\leq i<j\leq n-1$. Applying Proposition \ref{prop: essential 1 and 2 cells} and $\psi_{1}$:
$$\psi_{1}\circ\partial_{2}^{Q}[s_{j}\mid \ramp{i}{j}]=(s_{j}-1)\sum\limits_{\ell=i}^{j}\ramp{i}{\ell-1}e_{\ell}+e_{j}-\ramp{i}{j}e_{j-1}.$$
Note that $\sum\limits_{\ell=i}^{j-2}\ramp{i}{\ell-1}e_{j}-\sum\limits_{\ell=i}^{j-2}\ramp{i}{\ell}e_{j}=(1+s_{i}+\cdots +\ramp{i}{j-3})e_{j}-(s_{i}+s_{i}s_{i+1}+\cdots+\ramp{i}s_{j-2})e_{j}=e_{j}-\ramp{i}{j-2}e_{j}$. Therefore, for all $i<j$:
\begin{align*}
    \partial_{2}^{P}\circ\psi_{2}[s_{j}\mid \ramp{i}{j}]=&\sum\limits_{\ell=i}^{j-2}\ramp{i}{\ell-1}\partial_{2}^{P}d_{\ell j}+\ramp{i}{j-2}\partial_{2}^{P}b_{j-1}\\
    =&\sum\limits_{\ell=i}^{j-2}\ramp{i}{\ell-1}((s_{j}-1)e_{\ell}-(s_{\ell}-1)e_{j})\\
    &+\ramp{i}{j-2}((1-s_{j-1}+s_{j}s_{j-1})e_{j}-(1-s_{j}+s_{j-1}s_{j})e_{j-1})\\
\end{align*}
Grouping terms together, we get
\begin{align*}
    \partial_{2}^{P}\circ\psi_{2}[s_{j}\mid \ramp{i}{j}]=&(s_{j}-1)\sum\limits_{\ell=i}^{j-2}\ramp{i}{\ell-1}e_{\ell}+s_{j}-\sum\limits_{\ell=i}^{j-2}\ramp{i}{\ell}e_{j}+\sum\limits_{\ell=i}^{j-2}\ramp{i}{\ell-1}e_{j}\\
    &+s_{j}(\ramp{i}{j-1}e_{j}+\ramp{i}{j-2}e_{j-1})\\
&-(\ramp{i}{j-1}e_{j}+\ramp{i}{j-2}e_{j-1})\\
    &+\ramp{i}{j-2}e_{j}-\ramp{i}{j-2}e_{j-1}\\
    =&(s_{j}-1)\sum\limits_{\ell=i}^{j}\ramp{i}{\ell-1}e_{\ell}+e_{j}-\ramp{i}{j-2}e_{j-1}\\
    &=\psi_{1}\circ\partial_{2}^{Q}[s_{j}\mid \ramp{i}{j}].
\end{align*}
\end{proof}

\begin{proposition}\label{prop: image of s,s in P}
The chain $\psi_{2}\circ\partial_{3}^{Q}[s_{i}\mid s_{i}\mid s_{i}]$ is a boundary, 
    for each $i$ with $1\leq i\leq n-1$.
\end{proposition}
\begin{proof}
    Applying Proposition \ref{prop: boundary of s,s} $\psi_{2}\circ\partial_{3}^{Q}[s_{i}\mid s_{i}\mid s_{i}]=(s_{i}-1)\psi_{2}[s_{i}\mid s_{i}]=(s_{i}-1)c_{i}=\partial_{3}^{P} c_{i}^{3,1}$.
\end{proof}
\begin{proposition}\label{prop: image of s,c in P}
The chain $\psi_{2}\circ\partial_{3}^{Q}[s_{i}\mid s_{i}\mid s_{j}]$ is a boundary, for all $i$ and $j$ with $1\leq i<j\leq n-1$ and $i\leq j-2$
\end{proposition}
\begin{proof}
    Since $i<j$, applying Proposition \ref{prop: boundary of s,c} we obtain:
    \begin{align*}
        \psi_{2}\circ\partial_{3}^{Q}[s_{i}\mid s_{i}\mid s_{j}]=&(s_{i}+1)\psi_{2}[s_{i}\mid s_{j}]+(s_{j}-1)\psi_{2}[s_{i}\mid s_{i}]\\
        =&-(s_{i}+1)d_{i j}+(s_{j}-1)c_{i}\\
        =&\partial_3^{P}c_{i,j}^{3,2}.
    \end{align*}
\end{proof}

\begin{proposition}\label{prop: image of s,b in P}
    For all $1\leq i<j\leq n-1$, $\psi_{2}\circ\partial_{3}^{Q}[s_{j}\mid s_{j}\mid \ramp{i}{j}]$ is a boundary.
\end{proposition}
\begin{proof}
Let $A\in P_3$ be the chain
\[A=\ramp{i}{j-2}c_{j-1}^{3,5}+\sum\limits_{\ell=i}^{j-2}\ramp{i}{\ell-1}c_{j,\ell}^{3,2}.\]
Then by Proposition \ref{prop: boundary of s,b} and since $i<j$:
\[
\begin{split}
        \partial_{3}^{P}(A)=&\ramp{i}{j-2}((s_{j}+1)b_{j-1}-c_{j}+s_{j-1}s_{j}c_{j-1})\\
        &+\sum\limits_{\ell=i}^{j-2}\ramp{i}{\ell-1}((s_{\ell-1})c_{j}+(s_{j}+1)d_{\ell j})\\
        =&(s_{j}+1)\ramp{i}{j-2}b_{j-1}-\ramp{i}{j-2}c_{j}+\ramp{i}{j}c_{j-1}\\
        &+\sum\limits_{\ell=i}^{j-2}\ramp{i}{\ell}c_{j}-\sum\limits_{\ell=i}^{j-2}\ramp{i}{\ell-1}c_{j}+(s_{j}+1)\sum\limits_{\ell=i}^{j-2}\ramp{i}{\ell-1}d_{\ell j}.\\
\end{split}
\]
Using that $\sum\limits_{\ell=i}^{j-2}\ramp{i}{\ell}c_{j}-\sum\limits_{\ell=i}^{j-2}\ramp{i}{\ell-1}c_{j}=\ramp{i}{j-2}c_{j}-c_{j}$, this becomes:
\[
\begin{split}
        \partial_{3}^{P}(A)=&(s_{j}+1)(\ramp{i}{j-2}b_{j-1}+\sum\limits_{\ell=i}^{j-2}\ramp{i}{\ell-1}d_{\ell j})+\ramp{i}{j}c_{j-1}-c_{j}\\
        =&(s_{j}+1)\psi_{2}[s_{j}\mid \ramp{i}{j}]+\ramp{i}{j}\psi_{2}[s_{j-1}\mid s_{j-1}]-\psi_{2}[s_{j}\mid s_{j}]\\
        =&\psi_{2}\circ\partial_{3}^{Q}[s_{j}\mid s_{j}\mid \ramp{i}{j}].
\end{split}
\]
\end{proof}

\begin{proposition}\label{prop: image of c,s in P}
    Suppose $1\leq i<j\leq n-1$ with $i\leq j-2$. Then $\psi_{2}\circ\partial_{3}^{Q}[s_{i}\mid s_{j}\mid s_{j}]$ is a boundary.
\end{proposition}
\begin{proof}
    Since $i<j$ and by Proposition \ref{prop: boundary of c,s}: $\psi_{2}\circ\partial_{3}^{Q}[s_{i}\mid s_{j}\mid s_{j}]=(s_{i}-1)\psi_{2}[s_{j}\mid s_{j}]-(s_{j}+1)\psi_{2}[s_{i}\mid s_{i}]=(s_{i}-1)c_{j}+(s_{j}+1)d_{ij}=\partial_{3}^{P}c_{ji}^{3,2}$
\end{proof}
\begin{proposition}\label{prop: image of c,c in P}
    Suppose $1\leq i<j<k\leq n-1$ with $i\leq j-2$ and $j\leq k-2$. Then $\psi_{2}\circ\partial_{3}^{Q}[s_{i}\mid s_{j}\mid s_{k}]$ is a boundary.
\end{proposition}
\begin{proof}
    By Proposition \ref{prop: boundary of c,c}:
    \begin{align*}
        \psi_{2}\circ\partial_{3}^{Q}[s_{i}\mid s_{j}\mid s_{k}]=&(s_{i}-1)\psi_{2}[s_{j}\mid s_{k}]-(s_{j}-1)\psi_{2}[s_{i}\mid s_{k}]+(s_{k}-1)\psi_{2}[s_{i}\mid s_{j}]\\
        =&-(s_{i}-1)d_{jk}+(s_{j}-1)d_{ik}-(s_{k}-1)d_{ij}\\
        =&-\partial_{3}^{P}c_{ijk}^{3,3}.
    \end{align*}
\end{proof}

\begin{proposition}\label{pr:boundcellcbcase1}
Let $i$, $j$, and $k$ be integers with $1\leq i<k-1$ and $k<j<n$.
Let $\sigma=[s_i|s_j|\ramp{k}{j}]$.
Then the chain $\psi^P_3(\partial \sigma)$ is a boundary in $P_2$.
\end{proposition}

\begin{proof}
By Proposition~\ref{prop: boundary of c,b},
\[
\begin{split}
\psi_2(\partial^Q_3 \sigma) &= (s_i-1)\bigg(\sum_{\ell=k}^{j-2}\ramp{k}{\ell-1}d_{\ell j}+\ramp{k}{j-2}b_{j-1})\bigg)+d_{ij}\\
&\quad-\ramp{k}{j}d_{i,j-1}-(1-s_j)\sum_{\ell=0}^{j-k}\ramp{k}{k+\ell-1}d_{i,k+\ell}.
\end{split}
\]
We recognize this as the boundary of 
\[\sum_{\ell=0}^{j-k-2}\ramp{k}{k+\ell-1}c^{33}_{i,k+\ell,j}+\ramp{k}{j-2}c^{34}_{j-1,i}.\]
\end{proof}

\begin{proposition}\label{pr:boundcellcbcase23}
Let $i$, $j$, and $k$ be integers with $1\leq i<j-1$ and $1\leq k<j<n$.
Suppose that $i=k$ or $i=k-1$.
Let $\sigma=[s_i|s_j|\ramp{k}{j}]$.
Then the chain $\psi_2(\partial^Q_3 \sigma)$ is a boundary in $P_2$.
\end{proposition}

\begin{proof}
First assume that $i=k$.
Then by Proposition~\ref{prop: boundary of c,b},
\[
\begin{split}
\psi_2(\partial^Q_3 \sigma) &= s_k\Bigg(\sum_{\ell=k}^{j-2}\ramp{k}{\ell-1}d_{\ell j}+\ramp{k}{j-2}b_{j-1}\Bigg)+(1-s_j)c_k\\
&\quad-\sum_{\ell=k+1}^{j-2}\ramp{k+2}{\ell-1}d_{\ell j}-\ramp{k+1}{j-2}b_{j-1}.
\end{split}
\]
This simplifies to
\[\psi_2(\partial^Q_3 \sigma) = (s_k+1)d_{kj}+(1-s_j)c_k=-\partial c^{32}_{kj}.\]
Next assume that $i=k-1$.
Again by Proposition~\ref{prop: boundary of c,b},
\[
\begin{split}
\psi_2(\partial^Q_3 \sigma) &= s_{k-1}\Bigg(\sum_{\ell=k}^{j-2}\ramp{k}{\ell-1}d_{\ell j}+\ramp{k}{j-2}b_{j-1}\Bigg)\\
&\quad-\sum_{\ell=k-1}^{j-2}\ramp{k-1}{\ell-1}d_{\ell j}-\ramp{k-1}{j-2}b_{j-1}+d_{k-1,j}.
\end{split}
\]
This simplifies to $\psi_2(\partial^Q_3 \sigma)=0$, so it is a boundary.
\end{proof}

\begin{proposition}\label{pr:boundcellcbcase4}
Let $i$, $j$, and $k$ be integers with $1\leq i< j-1$, $1\leq k<j<n$, and $k <i$.
Let $\sigma=[s_i|s_j|\ramp{k}{j}]$.
Then the chain $\psi_2(\partial^Q_3 \sigma)$ is a boundary in $P_2$.
\end{proposition}

\begin{proof}
By Proposition~\ref{prop: boundary of c,b},
\[
\begin{split}
\psi_2(\partial^Q_3 \sigma) =& (s_i-1)\Bigg(\sum_{\ell=k}^{j-2}\ramp{k}{\ell-1}d_{\ell j}+\ramp{k}{j-2}b_{j-1}\Bigg)\\
&+(1-s_j)\Bigg(\sum_{\ell=k}^{i-2}\ramp{k}{\ell-1}d_{\ell i}\\
&+\ramp{k}{i-2}b_{i-1}
-\sum_{\ell=i+1}^j\ramp{k}{\ell-1} d_{i-1,\ell}\Bigg)+d_{i j}-\ramp{k}{j}d_{i-1,j-1}.
\end{split}
\]
We can recognize this as the boundary of
\[
\begin{split}
&\sum_{\ell=k}^{i-2}\ramp{k}{\ell-1}c^{33}_{\ell,i,j}
-\ramp{k}{i-2}c^{34}_{i-1,j} \\
&+ \sum_{\ell=i+1}^{j-2} \ramp{k}{\ell-1}c^{33}_{i-1,\ell,j}
+\ramp{k}{j-2}c^{34}_{j-1,i-1}.
\end{split}
\]
\end{proof}

\begin{proposition}\label{pr:boundcellbbcase1}
Let $i,j,k$ be integers with $1\leq k<i<j<n$.
Let $\sigma=[s_j|\ramp{i}{j}|\ramp{k}{j}]\in Q_3$.
Then the chain $\psi_2(\partial^Q_3 \sigma)$ is a boundary in $P_2$.
\end{proposition}

\begin{proof}
By using the formulas for $\partial^Q_3 \sigma$ in terms of essential $2$-simplices in $Q_2$, and using the definition of $\psi_2$, we can describe the structure of $\psi_2(\partial^Q_3 \sigma)$.
It is helpful to visualize $\psi_2(\partial^Q_3 \sigma)$ as a tiling of the surface of a cuboid by squares and hexagons.
In this description, the front face represents a relation conjugating $\ramp{i}{j}$ by $\ramp{k}{j}^{-1}$, and this face comes from $s_j[\ramp{i}{j}|\ramp{k}{j}]$ in $\partial \sigma$.
The back face contains the basepoint and is an identical copy of the front face.
These faces consist of a diagonal band of copies of $b_i$-cells, with the rest being filled in by $d_{ij}$-cells.
The contribution of these two faces to $\psi_2(\partial^Q_3 \sigma)$ is
\[\begin{split}
(s_j-1)\Bigg(&\sum_{\ell=0}^{j-i}\ramp{i}{i-1+\ell}\ramp{k}{i-2+\ell}b_{i-1+\ell} \\
&+\sum_{\ell=0}^{j-i}\sum_{m=0}^{i-k-2}\ramp{i}{i-1+\ell}\ramp{k}{k-1+m}d_{k+m,i+\ell} \\
&+\sum_{\ell=1}^{j-i}\sum_{m=i-k-1}^{i-k+\ell-2}\ramp{i}{i-1+\ell}\ramp{k}{k-1+m}d_{k+m,i+\ell} \\
&-\sum_{\ell=0}^{j-i-1}\sum_{m=i-k+\ell+1}^{j-k}\ramp{i}{i-1+\ell}\ramp{k}{k-1+m}d_{i+\ell-1,k+m}\Bigg).
\end{split}\]
The remaining faces are simpler.
The left face is:
\[\sum_{\ell=0}^{j-k-2}\ramp{k}{k+\ell-1}d_{k+\ell,j}+\ramp{k}{j-2}b_{j-1},\]
and the right face is:
\[
\begin{split}
-\ramp{i}{j}\Bigg(&\sum_{\ell=0}^{j-k-3}\ramp{k}{k+\ell-1}d_{k+\ell,j-1}\\
&+\ramp{k}{j-3}(b_{j-2}-s_{j-2}s_{j-1}d_{j-2,j})\Bigg).
\end{split}
\]
The bottom face is
\[-\sum_{\ell=0}^{j-i-2}\ramp{i}{i+\ell-1}d_{i+\ell,j}-\ramp{i}{j-2} b_{j-1},\]
and the top face is
\[\ramp{k}{j}\Bigg(\sum_{\ell=0}^{j-i-2}\ramp{i-1}{i+\ell-2}d_{i+\ell-1,j-1}+\ramp{i-1}{j-3}b_{j-2}\Bigg).\]
This is illustrated in Figure~\ref{fi:cuboid1}.
To see that this description is correct, compare to the expression for $\partial^Q_3 \sigma$ in Proposition~\ref{prop: boundary of bb}.
There are seven terms in this expression; under $\psi$, the first term becomes the left face, the second term becomes the top face, the third and fourth terms become the right face, the fifth term becomes the bottom face, and the last two terms become the front and back faces.)
\begin{figure}[h!]
\begin{tabular}{cc}
\includegraphics[scale=0.5]{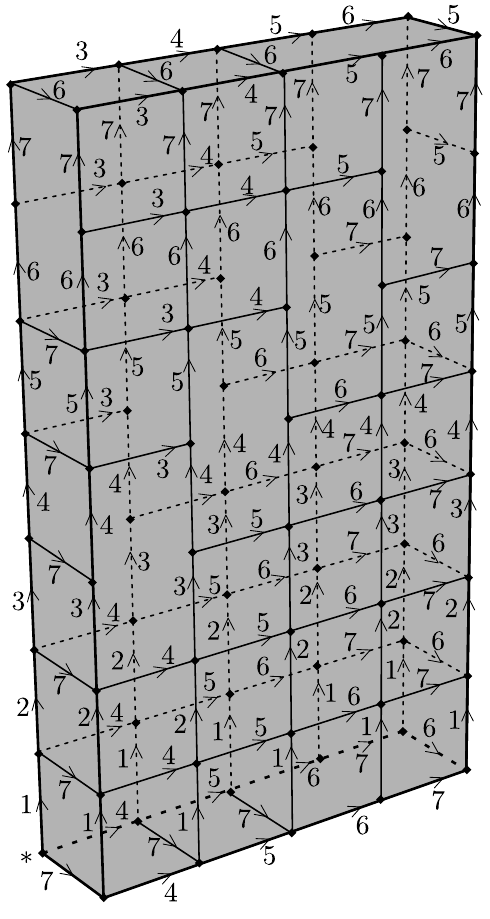} &
\includegraphics[scale=0.5]{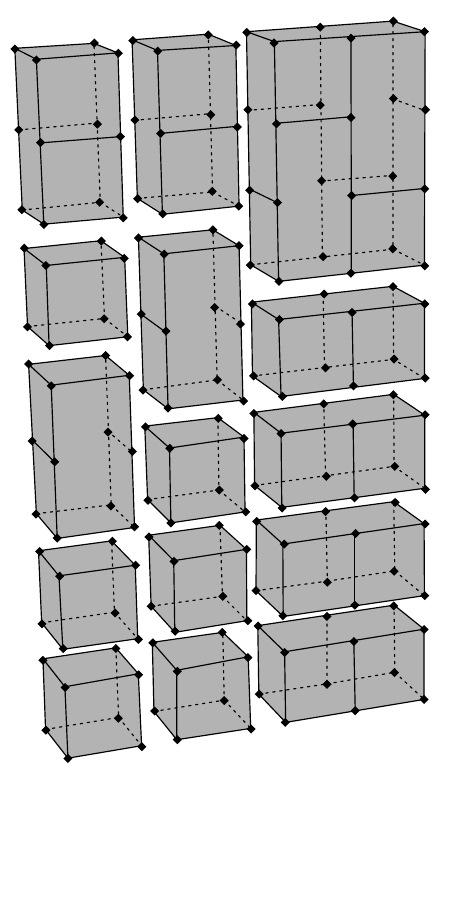}
\end{tabular}
\caption{An example of a filling of $\psi_2(\partial^Q_3 \sigma)$, as in Proposition~\ref{pr:boundcellbbcase1}, for $\sigma=[s_j|\ramp{i}{j}|\ramp{k}{j}])$.  
Here $k=1$, $i=4$, and $j=7$.  
Each edge $e_i$ is labeled with ``$i$".
The left figure is the visualization of $\psi_2(\partial^Q_3 \sigma)$ as a cuboid tiled by squares and hexagons.
The right figure is the same cuboid, exploded into $3$-cells of types $c^{33}_{ijk}$, $c^{34}_{ij}$, and $c^{37}_i$.}\label{fi:cuboid1}
\end{figure}

To prove the proposition, we must exhibit a sum of $3$-cells that bounds this cycle.
First of all, the upper-right corner of the cuboid can be filled with a cell of type $c^{37}_i$.
Specifically, the cell
\begin{equation}\label{eq:urcorner1}
\ramp{i}{j-2} \ramp{k}{j-3} c^{37}_{j-2}
\end{equation}
bounds the eleven $2$-cells closest to the upper-right corner (if $i=j-1$ or $k=j-2$, then it bounds more).

Next we bound the remainder of the right side using several $c^{34}_{ij}$-cells.
We use
\begin{equation}\label{eq:rside1}
\sum_{\ell=0}^{j-k-3}\ramp{i}{j-2}\ramp{k}{k+\ell-1} c^{34}_{j-1,k+\ell}.
\end{equation}
This bounds the rightmost two $d_{ij}$-cells in each of the bottom $j-k-2$ rows of the front and back faces, and the rightmost $b_i$-cell on the bottom, and the remaining cells on the right face.

Next we bound the remainder of the top side.
We use
\begin{equation}\label{eq:tside1}
-\sum_{\ell=0}^{j-i-2}\ramp{k}{j-2} \ramp{i-1}{i-2+\ell} c^{34}_{j-1,i-1+\ell}.
\end{equation}
This bounds the uppermost two $d_{ij}$-cells in each of the leftmost $j-i-1$ columns of the front and back faces, and the top $b_i$-cell on the left face, and the remaining cells on the bottom.
Further, these two stacks of $c^{34}_{ij}$-cells have two faces that cancel with two of the remaining faces from the $c^{37}_i$-cell in the corner.

Now there is a rectangle of $c^{33}_{ijk}$-cells that we can fill at the bottom of the cuboid.
We use
\begin{equation}\label{eq:bside1}
\sum_{\ell=0}^{j-i-2}\sum_{m=0}^{i-k-2}\ramp{i}{i-1+\ell}\ramp{k}{k-1+m}c^{33}_{k+m,i+\ell,j}.
\end{equation}
This bounds the bottom $i-k-1$ rows of the front, back, and left faces, and the rest of the bottom face.  
It meshes with the $c^{34}_{ij}$-cells on the right.

Finally, we fill in the remainder of the cuboid, which is the center-left section.
We bound the diagonal bands of $b_i$-cells with $c^{34}_{ij}$-cells, and fill in the triangular regions of $d_{ij}$-cells with $c^{33}_{ijk}$-cells.
This amounts to
\begin{equation}\label{eq:lside1}
\begin{split}
&\sum_{\ell=0}^{j-i-2}\ramp{i}{i-1+\ell}\ramp{k}{i-2+\ell}c^{34}_{i-1+\ell,j} \\
&+\sum_{\ell=1}^{j-i-2}\sum_{m=i-k-1}^{i-k+\ell-2}\ramp{i}{i-1+\ell}\ramp{k}{k-1+m}c^{33}_{k+m,i+\ell,j} \\
&-\sum_{\ell=0}^{j-i-3}\sum_{m=i-k+\ell+1}^{j-k-2}\ramp{i}{i-1+\ell}\ramp{k}{k-1+m}c^{33}_{i+\ell-1,k+m,j}.
\end{split}
\end{equation}
So $\psi_2(\partial^Q_3 \sigma)$ is equal to the boundary of the chain given by the sum of~\eqref{eq:urcorner1}, \eqref{eq:rside1}, \eqref{eq:tside1}, \eqref{eq:bside1}, and~\eqref{eq:lside1}.
\end{proof}

\begin{proposition}\label{pr:boundcellbbcase2}
Let $i,j,k$ be integers with $1\leq i\leq k<j<n$ and $j-k\geq 2$.
Let $\sigma=[s_j|\ramp{i}{j}|\ramp{k}{j}]\in Q_3$.
Then $\psi_2(\partial^Q_3 \sigma)$ is a boundary in $P_2$.
\end{proposition}

\begin{proof}
This is similar to the previous proposition, with one complication.
When we reduce $\ramp{i}{j}\ramp{k}{j}$ to normal form, this involves applying a squaring relation $s_k^2=1$.
This means that when we visualize $\psi_2(\partial^Q_3 \sigma)$ as a tiling of a cuboid, in addition to having squares and hexagons in the tiling, we also have two bigons.

The front and back faces of the cuboid each have a rectangular region of $(k-i)\times (j-k)$ squares on the left, a diagonal band of hexagonal regions going up to the upper right, a single bigon at the lower left of the diagonal band, and two triangular regions of squares above and below the band.
This is
\[\begin{split}
(s_j-1)\Bigg(&\ramp{i}{k-1}c_k \\
&-\sum_{\ell=0}^{k-i-1}\sum_{m=0}^{j-k-1}\ramp{i}{i+\ell-1}\ramp{k+1}{k+m}d_{i+\ell,k+m+1} \\
&+\sum_{\ell=0}^{j-k-1} \ramp{i}{k+\ell}\ramp{k}{k-1+\ell} b_{k+l} \\
&+\sum_{\ell=0}^{j-k-2} \sum_{m=0}^\ell\ramp{i}{k+\ell+1} \ramp{k}{k+m-1} d_{k+m,k+\ell+2} \\
&-\sum_{\ell=0}^{j-k-2} \sum_{m=\ell}^{j-k-2} \ramp{i}{k+\ell} \ramp{k}{k+1+m} d_{k+l,k+2+m}
\Bigg).
\end{split}\]

Because of the cancellation represented by the bigons, the left and top sides have one fewer tile each than the right and bottom sides, respectively.
The left face is
\[\sum_{\ell=0}^{j-k-3}\ramp{k+1}{k+\ell}d_{k+\ell+1,j}+\ramp{k+1}{j-2}b_{j-1},\]
and the right face is
\[-\ramp{i}{j}\Bigg(\sum_{\ell=0}^{j-k-3}\ramp{k}{k-1+\ell}d_{k+\ell,j-1}+\ramp{k}{j-3}(b_{j-2}-s_{j-2}s_{j-1}d_{j-2,j})\Bigg).\]
The bottom face is
\[-\sum_{\ell=0}^{j-i-2}\ramp{i}{i+\ell-1}d_{i+\ell,j}-\ramp{i}{j-2} b_{j-1},\]
and the top face is
\[\ramp{k+1}{j}\Bigg(\sum_{\ell=0}^{j-i-3}\ramp{i}{i+\ell-1}d_{i+\ell,j-1}+\ramp{i}{j-3}b_{j-2}\Bigg).\]

Like with the previous proposition, we can justify this by looking at the expression for $\partial\sigma$ in Proposition~\ref{prop: boundary of bb}.
That expression has eight terms.
The first term is the left face, the second and fourth terms are the right face, the third term is the top face, the fifth term is the bottom face, and the last three terms are the front and back faces.

Bounding this cycle works similarly to the previous proposition.
First of all, we fill the upper-right corner with a $c^{37}_i$-cell:
\begin{equation}\label{eq:urcorner2}
\ramp{i}{j-2} \ramp{k}{j-3} c^{37}_{j-2}.
\end{equation}

Again, we bound the remainder of the right side using several $c^{34}_{ij}$-cells:
\begin{equation}\label{eq:rside2}
\sum_{\ell=0}^{j-k-3}\ramp{i}{j-2}\ramp{k}{k+\ell-1} c^{34}_{j-1,k+\ell}.
\end{equation}

Similarly, we bound the remainder of the top side:
\begin{equation}\label{eq:tside2}
-\sum_{\ell=0}^{j-i-3}\ramp{k+1}{j-2} \ramp{i}{i-1+\ell} c^{34}_{j-1,i+\ell}.
\end{equation}

Like in the previous proposition, there is a rectangle that we can fill with $c^{33}_{ijk}$-cells, but this time it is on the left of the cuboid.
We use
\begin{equation}\label{eq:lside2}
\sum_{\ell=0}^{k-i-1}\sum_{m=1}^{j-k-2}\ramp{i}{i-1+\ell}\ramp{k+1}{k-1+m}c^{33}_{i+\ell,k+m,j}.
\end{equation}

Finally, we fill in the remaining region with a diagonal band of $c^{34}_{ij}$-cells and two triangular regions of $c^{33}_{ijk}$-cells.
The main difference between this and the previous case is that we also include a $c^{32}_{ij}$-cell to bound the two bigon regions:
\begin{equation}\label{eq:lastregion2}
\begin{split}
&\ramp{i}{k-1} c^{32}_{kj} 
+\sum_{\ell=0}^{j-k-3}\ramp{i}{k+\ell}\ramp{k}{k-1+\ell}c^{34}_{k+\ell,j} \\
&+\sum_{\ell=k-i+2}^{j-i-2}\sum_{m=0}^{\ell-k+i-2}\ramp{i}{i-1+\ell}\ramp{k}{k-1+m}c^{33}_{k+m,i+\ell,j} \\
&-\sum_{\ell=k-i}^{j-i-4}\sum_{m=i+\ell-k+2}^{j-k-2}\ramp{i}{i+\ell}\ramp{k}{k-1+m}c^{33}_{i+\ell,k+m,j}.
\end{split}
\end{equation}
Then $\psi_2(\partial^Q_3\sigma)$ is the boundary of the sum of the chains~\eqref{eq:urcorner2}, \eqref{eq:lside2}, \eqref{eq:tside2}, \eqref{eq:rside2}, and~\eqref{eq:lastregion2}.
\end{proof}

\begin{proposition}\label{pr:boundcellbbcase3}
Let $i,j$ be integers with $1\leq i\leq j-1$ and $j<n$.
Let $\sigma=[s_j|\ramp{i}{j}|\ramp{j-1}{j}]\in Q_3$.
Then the chain $\psi_2(\partial^Q_3 \sigma)$ is a boundary in $P_2$.
\end{proposition}

\begin{proof}
In this special case, $\psi_2(\partial^Q_3 \sigma)$ is simpler than in the two preceding propositions.
By Proposition~\ref{prop: boundary of bb},
\[
\begin{split}
\psi_2(\partial^Q_3 \sigma) &= c_j-\sum_{k=0}^{j-i-2}\ramp{i}{i+k-1}(1+s_j)d_{i+k,j} \\
&\quad +s_j\ramp{i}{j-2}(s_{j-1}b_{j-1} - s_{j}b_{j-1} + s_{j-1}c_{j-1}) \\
&\quad -\ramp{i}{j}(c_j+s_{j-1}c_{j-1}).
\end{split}
\]
This is the boundary of
\[-\sum_{k=0}^{j-i-2} \ramp{i}{i+k-1} c^{32}_{j,i+k}-\ramp{i}{j-2}c^{31}_{j}-\ramp{i}{j-1} c^{36}_{j-1}.\]
\end{proof}

\begin{proposition}\label{pr:boundcellbc}
Let $i, j, k$ be integers with $1\leq k < i< j-1$ and $j<n$.
Let $\sigma=[s_i|\ramp{k}{i}|s_j]\in Q_3$.
Then the chain $\psi_2(\partial^Q_3 \sigma)$ is a boundary in $P_2$.
\end{proposition}

\begin{proof}
By Proposition~\ref{prop: boundary of b,c},
\[
\begin{split}
\psi(\partial \sigma) &= -d_{ij}+(s_j-1)(\sum_{\ell=k}^{i-2}\ramp{k}{\ell-1}d_{\ell i}+ \ramp{k}{i-2}b_{i-1}) \\
&\quad +\ramp{k}{i} d_{i-1,j} +(1-s_i)\sum_{\ell=0}^{i-k}\ramp{k}{k+\ell-1}d_{k+\ell,j}.
\end{split}
\]
We visualize this as a cuboid tiled by squares and hexagons.
The first term is the bottom, the second term is the left and right sides, the third term is the top, and the last term is the front and back sides.
We then recognize this as the boundary of 
\[
\sum_{\ell=0}^{i-k-2}\ramp{k}{k+\ell-1} c^{33}_{k+\ell,i,j}+\ramp{k}{i-2}c^{34}_{i-1,j}.\]
\end{proof}

\begin{proposition}\label{pr:boundcellbs}
Let $i, j$ be integers with $1\leq i< j<n$.
Let $\sigma=[s_j|\ramp{i}{j}|s_j]\in Q_3$.
Then the chain $\psi_2(\partial^Q_3 \sigma)$ is a boundary in $P_2$.
\end{proposition}

\begin{proof}
By Proposition~\ref{prop: boundary of b,s},
\[\psi_2(\partial^Q_3 \sigma) = s_j\ramp{i}{j-1} c_j -\ramp{i}{j-2}((1+s_{j-1})b_{j-1}+c_{j-1}).\]
This is the boundary of
\[\ramp{i}{j-2}c^{3,6}_{j-1}-\ramp{i}{j-1}c^{3,5}_{j-1}.\]
(Note: at a glance, $\psi(\partial\sigma)$ appears to be the boundary of a single $c^{3,5}$-cell, but the configuration of edges is not quite right; the $c^{3,6}$-cell is needed to correct this.)
\end{proof}

Now we can finish the proof of Theorem~\ref{th:main}.
\begin{proof}[Proof of Theorem~\ref{th:main}.]
By Theorem~\ref{th:rwscomplete}, our rewriting system is complete.
So by Theorem~\ref{Theorem: Quotient Complex}, the chain complex $Q_*$ is a free resolution for $\Z$ as a $\Z \sg{n}$-module.
To show the theorem, we build chain maps $\psi\co Q_*\to P_*$ and $\phi\co P_*\to Q_*$ in dimensions $0$ through $3$, and show that $\psi_2$ is a left-inverse to $\phi_2$.

Now we claim that there is a chain map $\psi\co Q_*\to P_*$ in dimensions $0$ through $3$.
The map $\psi$ that we have already defined in dimensions $0$ through $2$ (at the beginning of the current section) will work; we just need to explain how it can be extended to dimension $3$.
It is a chain map in dimensions $0$ through $2$ by Lemma~\ref{le:lowdimchainmap}.
For each basis element $\sigma\in Q_3$, we need a chain $\psi_3(\sigma)\in P_3$ with $\partial^P_3\psi_3(\sigma)=\psi_2(\partial^Q_3 \sigma)$.
So it is enough to show that each $\psi_2(\partial^Q_3\sigma)$ is a boundary in $P_2$.
. 
We enumerated the nine kinds of basis elements for $Q_3$ in Remark~\ref{remark:Q3gens}.
Each kind of basis element has a boundary that maps to a boundary in $P_2$ under $\psi_2$; we showed this
\begin{itemize}
\item for $[s_i\mid s_i\mid s_i]$ in Proposition~\ref{prop: image of s,s in P},
\item for $[s_i\mid s_i\mid s_j]$, with $i<j-1$, in Proposition~\ref{prop: image of s,c in P},
\item for $[s_i\mid s_i\mid \ramp{j}{i}]$, with $j<i$, in Proposition~\ref{prop: image of s,b in P},
\item for $[s_i\mid s_j\mid s_j]$, with $i<j-1$, in Proposition~\ref{prop: image of c,s in P},
\item for $[s_i\mid s_j\mid s_k]$, with $i<j-1$ and $j<k-1$, in Proposition~\ref{prop: image of c,c in P},
\item for $[s_i\mid s_j\mid \ramp{k}{j}]$, with  $i<j-1$ and $k<j$, in Propositions~\ref{pr:boundcellcbcase1}, \ref{pr:boundcellcbcase23}, and~\ref{pr:boundcellcbcase4}.
\item for $[s_i\mid \ramp{j}{i}\mid s_i]$, with $j<i$, in Proposition~\ref{pr:boundcellbs},
\item for $[s_i\mid \ramp{j}{i}\mid s_k]$, with $j<i$ and $i<k-1$, in Proposition~\ref{pr:boundcellbc},
\item for $[s_i\mid \ramp{j}{i}\mid \ramp{k}{i}]$, with $j<i$ and $k<i$, in Proposition~\ref{pr:boundcellbbcase1}, \ref{pr:boundcellbbcase2}, and~\ref{pr:boundcellbbcase3}.
\end{itemize}
So for each $\Z\sg{n}$-basis element, we pick a bounding chain and define the image of this basis element under $\psi_3$ to be this chain.
We extend by linearity to finish the definition of $\psi$.

We also build a chain map $\phi\co P_*\to Q_*$.
It is easy to verify that the following formulas define a chain map:
\begin{itemize}
\item $\phi_0(*)=[\hspace{0.5em}]$,
\item $\phi_1(e_i)=[s_i]$,
\item $\phi_2(c_i)=[s_i\mid s_i]$, $\phi_2(b_i)=[s_{i+1}\mid s_is_{i+1}]$, and 
$\phi_2(d_{ij})=-[s_i\mid s_j]$, for $i<j-1$.
\end{itemize}
(It is also possible to define $\phi_3$, but we don't need this.)
It also follows from our formulas that $\psi_2\circ \phi_2=\mathrm{id}$ on $P_2$.

Now we use our maps to show that $P_*$ is the truncation of a free resolution.
Since the $2$-skeleton of $P_*$ is a Cayley complex for a presentation for $\sg{n}$, it is enough to show that $H_2(P_*)=0$.
Let $[c]\in H_2(P_*)$ be a class with cycle representative $c\in P_2$.
Then $\phi_2(c)\in Q_2$, and since $\phi$ is chain map, $\partial\phi_2(c)=\phi_1(\partial c)=0$.
Since $Q_*$ is a free resolution of $\Z$, we have $H_2(Q_*)=0$, so there is a chain $f\in Q_3$ with $\partial f=\phi_2(c)$.
Then $\psi_3(f)\in P_3$.
As we painstakingly showed, $\psi$ is a chain map, so $\partial\psi_3(f)=\psi_2(\partial f)=\psi_2(\phi_2(c))$.
But this composition is the identity, so $\partial \psi_3(f)=c$.
Since $[c]$ was arbitrary, this shows that $H_2(P_*)=0$.
This means that $P_*$ is a free resolution of $\Z$ as a $\Z \sg{n}$-module up to dimension $3$, which is what we needed to show.
\end{proof}

\section{Preliminaries for the applications}
\subsection{Actions on sets and coinduced modules}
In this section, we fix a commutative ring $R$.
Some of the modules that we consider below are derived from group actions by the free $R$-module functor.
We compute their cohomology by identifying them as coinduced modules and using Shapiro's lemma, which we recall below.
Coinduced modules can be defined more generally, but for simplicity we only coinduce up from untwisted modules.
So suppose $G$ is a group, and $H\leq G$, and $M$ is an $R$-module considered as a trivial $RH$-module.
Then the \emph{coinduced module} of $M$ from $H$ to $G$ is
\[\mathrm{Coind}_H^G(M)=\hom_{RH}(RG,M).\]
Here $\hom_{RH}(RG,M)$ consists of the $R$-module homomorphisms $f\colon RG\to M$ such that, for all $h\in H$ and $a\in RG$, we have $f(ah)=f(a)$.
Then $\mathrm{Coind}_H^G(M)$ is an $RG$-module by, for $f\in \mathrm{Coind}_H^G(M)$, $g\in G$, and $a\in RG$, $(gf)(a)=f(g^{-1}a)$.

The main application of coinduced modules is Shapiro's lemma.
A reference is Brown~\cite{Brown}.
The next theorem is part of Theorem~III.6.2 of Brown, and we sketch a proof of it in the proof of Proposition~\ref{pr:explicitShapiro} below.
\begin{theorem}[Shapiro's lemma for cohomology]
For $G$, $H$, and $M$ as above, there is an isomorphism
\[H^*(G;\mathrm{Coind}_H^G(M))\cong H^*(H;M).\]
\end{theorem}

We need to use this isomorphism explicitly on cocycles.
\begin{proposition}\label{pr:explicitShapiro}
The isomorphism in Shapiro's lemma is the composition
\[H^*(G;\mathrm{Coind}_H^G(M))\to H^*(H;\mathrm{Coind}_H^G(M)) \to H^*(H;M),\]
where the left map is induced by the inclusion $H\to G$ and the right map is induced by evaluation $\mathrm{Coind}_H^G(M)\to M$, $f\mapsto f(1)$.
\end{proposition}

\begin{proof}
    First we sketch the proof of Shapiro's lemma for cohomology.
    Fix a projective resolution $F_*$ for $\Z$ as a $\Z G$-module.
    Then $F_*$ is also a projective resolution for $\Z$ as a $\Z H$-module.
    For any $i$,
    \[
    \begin{split}
        H^i(G;\mathrm{Coind}_H^G(M))&= H^i(\hom_{\Z G}(F_*,\hom_{\Z H}(\Z G,M))) \\
        &\cong H^i(\hom_{\Z H}(F_*\otimes_{\Z G} \Z G,M)) \\
        &\cong H^i(\hom_{\Z H}(F_*,M)) \\
        &= H^i(H;M).
    \end{split}
    \]
    If $\alpha\co F_i\to \mathrm{Coind}_H^G(M)$ is a cocycle, the hom-tensor adjunction sends $\alpha$ to the map 
    \[F_i\otimes_{\Z G}\Z G\to M,\quad c\otimes g\mapsto \alpha_c(g).\]
    Since $F_i\to F_i\otimes_{\Z G}\Z G$ by $c\mapsto c\otimes 1$ is an isomorphism, the next isomorphism sends $\alpha$ to 
    \[F_i\to M, \quad c\mapsto \alpha_c(1).\]
    But this cocycle also represents the image of $[\alpha]$ under the composite map in the statement.
\end{proof}

Coinduced modules do arise naturally.
Let $G$ be a group acting transitively on a finite set $S$.  Let $F$ be the free $R$-module with basis $S$, considered as an $RG$-module by extending the action on basis elements linearly.
It turns out that  $F$ is a coinduced module, as we explain below.  

A typical element of $F$ has the form $\sum_{s\in S}r_s s$ with coefficients $r_s\in R$.
For  any $t\in S$, we have an  $R$-linear projection map $\pi_t\co F\to R$ by $\pi_t(\sum_{s\in S}r_s s)=r_t$.
Pick a basepoint $e\in S$ and let $H\leq G$ be the stabilizer of $e$.

\begin{proposition}\label{pr:freeascoinduced}
We have 
$F \cong \mathrm{Coind}_H^G(R),$
where $R$ is the trivial $H$-module.
This isomorphism is $\psi\co F\to \mathrm{Coind}_H^G(R)$, given by $\psi(v)=\psi_v\co RG\to R$, for $v\in F$, with 
    \[\psi_v(g)=\pi_{e}(g^{-1}v).\] 
\end{proposition}
\begin{proof}
    First we explain how the $\pi_t$ functions transform.
    For any $t\in S$, $v\in F$ and $g\in G$, we have $\pi_t(gv)=\pi_{g^{-1}t}(v)$.
    To see this, write $v=\sum_{s\in S}r_s s$; then
    \[\pi_t(gv)=\pi_t(\sum_{s\in S}r_s gs)=\pi_t(\sum_{s\in S}r_{g^{-1}s} s)=r_{g^{-1}t}=\pi_{g^{-1}t}(v).\]
    
    Now fix $v\in F$.
    It is straightforward to verify that $\psi_v$ is $R$-linear.  
    To see that $\psi_v$ is $H$-equivariant, let $h\in H$.  
    Then for $g\in G$
    \[\phi_v(gh)
    =\pi_{e}(h^{-1}g^{-1}v)
    =\pi_{he}(g^{-1}v)
    =\pi_{e}(g^{-1}v)
    =\phi_v(g).
    \]
    So $\psi_v$ is a well defined element of $\mathrm{Coind}_H^G(R)$.

    Now let $v$ vary.  Again, it is easy to see that $\psi$ is $R$-linear.
    To see that $\psi$ is $G$-equivariant, let $g\in G$ and $v\in F$.
    Let $k\in G$.  Then
    \[\psi_{gv}(k)=\pi_e(k^{-1}gv)=\psi_v(g^{-1}k)=(g\psi_v)(k).\]
    This shows that $\psi_{gv}=g\psi_v$, meaning that $\psi$ is $G$-equivariant.
    So $\psi$ is a well defined $RG$-linear map.

    Now suppose $v\in F$ is in the kernel of $\psi$.
    Write $v=\sum_{s\in S} r_s s$.
    Since the action of $G$ on $S$ is transitive, for each $t\in S$, there is $g\in G$ with $ge=t$.
    Then $0=\psi_v(g)=\pi_e(g^{-1}v)=\pi_{ge}(v)=r_t$.
    So $v=0$, and the kernel is trivial, so $\psi$ is injective.

    Finally, suppose $f\in \mathrm{Coind}_H^G(R)$.  Since $S$ is finite, for each $s\in S$, we can choose an element $g_s\in G$ with $g_se=s$, and set $v=\sum_{s\in S} f(g_s)s$.
    Fix $g\in G$ and set $t=ge$.
    Then $\psi_v(g)=\pi_e(g^{-1}v)=\pi_{ge}(v)=f(g_{t})$.
    But $g_{t}e=ge$, so $g_{t}^{-1}g\in H$, so $f(g_{t})=f(g_{t}g_{t}^{-1}g)=f(g)$.
    Since $g$ was arbitrary, $\psi_v=f$, and $\psi$ is surjective.
    So $\psi$ is an isomorphism.
\end{proof}

\subsection{Untwisted second homology and cohomology}

Now we recall some classical facts.
For the untwisted module $\Z$, $H_1(\sg{n};\Z)=\Z/2\Z$ for $n\geq 2$.
(This follows from the presentation for $\sg{n}$ using the interpretation of $H_1$ as the abelianization.)
For $n\geq 4$, we have $H_2(\sg{n};\Z)=\Z/2\Z$ (and this is trivial for $n=2$ and $n=3$).
This is due to Schur~\cite{Schur}, and predates the definition of homology.

We need explicit representatives of the low-dimensional homology classes for $\sg{n}$ using $P_*\otimes_{\Z \sg{n}}\Z$ as the chain complex.
For $H_1(\sg{n};\Z)$, it is easy to see that $e_1$ represents the nontrivial class (for example, by reasoning about the abelianization).
For $H_2(\sg{n};\Z)$, finding a representative takes a little more work; in doing so, we recover Schur's computation of $H^2(\sg{n};\Z)$.
\begin{lemma}\label{le:H2Sn}
Let $n\geq 4$.  Then $H_2(\sg{n};\Z)=\Z/2\Z$.
Using $P_*\otimes_{\Z \sg{n}} \Z$ to compute $H_2(\sg{n};\Z)$, the nontrivial homology class in $H_2(\sg{n};\Z)$ is $[d_{13}]$.
\end{lemma}
\begin{proof}
After tensoring with $\Z$ to take coinvariants, the boundary on $P_2$ becomes $c_i\mapsto 2e_i$, $b_i\mapsto e_{i+1}-e_i$, and $d_{ij}\mapsto 0$.
From this it is straightforward to show that the $2$-cycles are generated by all the $d_{ij}\otimes 1$ elements, and for $i$ from $1$ to $n-2$, the elements $2b_i +c_i-c_{i+1}$.
The boundary of $c_i^{3,5}$ is precisely this second kind of generator.
The boundary of $c_{ij}^{3,2}$ shows that each $[d_{ij}]$ has order at most $2$ in $H_2(\sg{n};\Z)$.
The boundary of $c_{ij}^{3,4}\otimes 1$ shows that for all $i$ and $j$ for which it makes sense, we have that $[d_{ij}]=[d_{i+1,j}]$ and $[d_{ji}]=[d_{j,i+1}]$.
The other $3$-chain generators map to $0$ after tensoring with $\Z$. 
 The lemma follows.
\end{proof}

For any abelian group $A$ and positive integer $k$, we use $A[k]$ to denote the $k$-torsion of $A$:
\[A[k]=\{a\in A\,|\,ka=0\}.\]
The following is well known:
\begin{proposition}\label{pr:Snuntwisted}
    Let $A$ be an abelian group.
    For all $n\geq 2$, we have
    \[H^1(\sg{n};A)=A[2].\]
    We have
    \[H^2(\sg{n};A)= 
    \left\{
    \begin{array}{cc}
    A/2A & \text{if $n\in\{2,3\}$} \\
    A[2]\oplus A/2A &\text{if $n\geq 4$}
    \end{array}
    \right.
    \]
\end{proposition}
\begin{proof}
We prove this as a corollary of Proposition~\ref{pr:untwistedcohomreps} below.
\end{proof}

Because we will need to do explicit computations later, we define cocycles representing these cohomology classes.
Since $P_1$ and $P_2$ are free $\Z \sg{n}$-modules, to define our cocycles it suffices to define them on the bases given above.
\begin{definition}
    Let $A$ be an abelian group, considered as a $\Z \sg{n}$-module with the trivial action. 
    For $r\in A[2]$,
    define $\kappa^0_r\co P_1\to A$ by $\kappa^0_r(e_i)=r$ for all $i$ from $1$ to $n-1$.

    For $r\in A$, define $\alpha^0_r\co P_2\to A$ by $\alpha^0_r(c_i)=r$, $\alpha^0_r(d_{ij})=0$, and $\alpha^0_r(b_i)=0$, for all choices of $i$ and $j$ that make sense.

    Now assume $n\geq 4$.
    For $r\in A[2]$, define $\beta^0_r\co P_2\to A$ by
    $\beta^0_r(c_i)=0$, $\beta^0_r(d_{ij})=r$, and $\beta^0_r(b_i)=0$, for all choices of $i$ and $j$ that make sense.
\end{definition}
The reason for the superscript $0$ is to fit with notation that we use for permutation modules.

\begin{proposition}\label{pr:untwistedcohomreps}
    For all choices of $r$ as above, the functions $\kappa^0_r$, $\alpha^0_r$, and $\beta^0_r$ are cocycles.
    
    The map $A[2]\to H^1(\sg{n};A)$ by $r\mapsto [\kappa^0_r]$ is an isomorphism.

    If $n\geq 4$, then the map $A[2]\oplus A/2A\to H^2(\sg{n};A)$ by $(r,[s])\mapsto [\beta^0_r+\alpha^0_s]$ is an isomorphism.
    Similarly, if $n$ is $2$ or $3$, then $A/2A\to H^2(\sg{n};A)$ by $[s]\mapsto [\alpha^0_s]$ is an isomorphism.
\end{proposition}

\begin{proof}
    It is quick to verify that these functions are cocycles by evaluating their coboundaries on generators, using the generators and boundaries for $P_*$ given in Definition~\ref{de:Snres} above.
    We leave this as an exercise.

    To verify the statements about isomorphisms, we turn to the universal coefficients theorem.  Brown~\cite{Brown}, Section~I.0, is a reference.  For $i=1,2$, we have exact sequences
    \[0\to \mathrm{Ext}^1_\Z(H_{i-1}(\sg{n}),A)\to H^i(\sg{n};A)\to \hom_\Z(H_i(\sg{n}),A)\to 0.\]
    (Suppressed coefficients are $\Z$, untwisted.)
    If $i=1$, then since $H_0(\sg{n})=\Z$ is free abelian, the exact sequence says that evaluation defines an isomorphism from $H^1(\sg{n};A)$ to $\hom_\Z(H_1(\sg{n}),A)$.
    Since $H_1(\sg{n})=\Z/2\Z$ is generated by $[e_1]$, and $\kappa_a^0(e_1)=a$, we see that the composition $A[2]\to H^1(\sg{n};A)\to \hom_\Z(H_1(\sg{n}),A)$ is an isomorphism.
    This proves the statement about $H^1(\sg{n};A)$.

    For $i=2$ and $n\geq 4$, consider the following diagram with exact rows:
    \begin{equation}\label{eq:uctdiagram}
    \xymatrix{
    0 \ar[r] & A/2A \ar[r] \ar[d] & A[2]\oplus A/2A \ar[r] \ar[d] & A[2] \ar[r] \ar[d] & 0 \\
    0 \ar[r] & \mathrm{Ext}^1_\Z(H_{1}(\sg{n}),A) \ar[r] & H^2(\sg{n};A) \ar[r] & \hom_\Z(H_2(\sg{n}),A) \ar[r] & 0.
    }
    \end{equation}
    The middle vertical map is the map in the statement, and the maps in the first row are the direct sum coordinate maps.
    The right vertical map sends $r$ to $[\beta^0_r]$, which is an isomorphism since $H_2(\sg{n};\Z)$ is generated by $[d_{13}]$ by Lemma~\ref{le:H2Sn}.
    So the right square commutes.
    We have a chain homotopy equivalence of projective resolutions by
    \begin{equation}\label{eq:extres}
    \xymatrix{
    0 \ar[r] \ar[d] & \Z \ar[r]^2 \ar[d]^{c_1} & \Z \ar[r] \ar[d]^{e_1} & \Z/2\Z \ar@{=}[d] \\
    Z_2(P_*\otimes \Z)  \ar[r] & P_2\otimes \Z \ar[r]^\partial & Z_1(P_*\otimes \Z) \ar[r] & H_1(\sg{n};\Z).
    }
    \end{equation}
    (We are tensoring over $\Z \sg{n}$ here.)
    For a cocycle $\gamma\colon P_2\to A$, we can consider $\gamma$ as a map  
    \[P_2\otimes \Z\to A \text{ by } c\otimes 1\mapsto \gamma(c),\]
    since the action on $A$ is trivial.
    If we further assume that $[\gamma]$ maps to the trivial element of $\hom_\Z(H_2(\sg{n}),Z)$, then it restricts to $0$ on $Z_2(P_*\otimes \Z)$ and defines an element of $\mathrm{Ext}^1_\Z(H_{1}(\sg{n}),A)$.
    This association is an isomorphism, and its inverse
    \[\mathrm{Ext}^1_\Z(H_{1}(\sg{n}),A)\stackrel{\cong}{\longrightarrow} \ker\big(H^2(\sg{n};A)\to \hom_\Z(H_2(\sg{n});A)\big)\]
    is the map in the universal coefficients theorem.
    Let $s\in A$; then since $\alpha_s^0(d_{13})=0$, we have $[\alpha_s^0]\in \mathrm{Ext}^1_\Z(H_{1}(\sg{n}),A)$.
    The vertical homotopy equivalence in diagram~\eqref{eq:extres} induces the isomorphism $\mathrm{Ext}^1_\Z(H_{1}(\sg{n}),A)\to A/2A$, and it takes $[\alpha_s^0]$ to $[s]$.
    The inverse of this map is the left vertical map in diagram~\eqref{eq:uctdiagram}, and it follows from these definitions that the left square commutes.
    So the middle vertical map is an isomorphism by the five lemma.
    If $n$ is $2$ or $3$, the conclusion follows by the same argument, but with $A[2]$ replaced by $0$.
\end{proof}

As an application, we study the low-dimensional cohomology of $\sg{n}$ in the permutation modules of the $2$-row Young diagrams.
Recall that a \emph{Young diagram} is a partition of $n$ given as a left-justified configuration of boxes, and a \emph{Young tableau} is a bijective filling of these boxes by the integers $1$ through $n$.
A \emph{tabloid} is an equivalence class of tableaux under reorderings of rows.
So a tabloid for the partition $(n-k,k)$ has the same data as a $k$-subset of $\{1,2,\dotsc,n\}$.

Let $n\geq 4$, and let $k$ be an integer with $1\leq k\leq n/2$, and consider the action of $\sg{n}$ on the free module with basis given by the $(n-k,k)$-tabloids.
To apply Shapiro's lemma, we need to consider the stabilizer of a basis element.
This is the same as a point stabilizer in the natural action of $\sg{n}$ on the $k$-subsets of $\{1,2,\dotsc,n\}$.
The stabilizer of $\{n-k+1,\dotsc,n\}$ is the subgroup generated by $s_1,\dotsc,s_{n-k-1}$ and $s_{n-k+1},\dotsc, s_{n-1}$, and it is isomorphic to $\sg{n-k}\times \sg{k}$.
We use the notation $\sg{n-k,k}$ for this stabilizer subgroup.
Let $\iota_1\co \sg{n-k}\to\sg{n-k,k}$ and $\iota_2\co \sg{k}\to\sg{n-k,k}$ be the product inclusion maps.

\begin{proposition}\label{pr:stabilizerhomology}
Let $n\geq 2$, with $1\leq k\leq n/2$.
    Let $A$ be an abelian group, considered as an untwisted $\Z\sg{n-k,k}$-module.  Then
\[
H_1(\sg{n-k,k};A)=
\left\{
\begin{array}{cl}
0 & \text{if $(n,k)=(2,1)$,} \\
A/2A & \text{if $k=1$, $n-k\geq 2$,} \\
(A/2A)^2 & \text{otherwise.}
\end{array}
\right.
\]
Generators are given by $\iota_{1*}(e_1)\otimes a$ and $\iota_{2*}(e_1)\otimes a$ for $a\in A[2]$.
\[
H_2(\sg{n-k,k};A)=
\left\{
\begin{array}{cl}
0 & \text{if $(n,k)=(2,1)$,} \\
A[2] & \text{if $k=1$, $n-k\in \{2,3\}$,} \\
A[2]^2\oplus (A/2A) & \text{if $k,n-k\in \{2,3\}$,} \\
A[2]\oplus (A/2A) & \text{if $k=1$ and $n-k\geq 4$,} \\
A[2]^2\oplus (A/2A)^2 & \text{if $k\in\{2,3\}$ and $n-k\geq 4$,} \\
A[2]^2\oplus (A/2A)^3 & \text{otherwise.}
\end{array}
\right.
\]
Generators are given by $\iota_{1*}(c_1)\otimes a$ and $\iota_{2*}(c_1)\otimes a$ for $a\in A[2]$, and $\iota_{1*}(d_{13})\otimes a$, $\iota_{2*}(d_{13})\otimes a$, and $\iota_{1*}(e_1)\otimes \iota_{2*}\otimes a$ for $a\in A$.
This last generator is interpreted via the homology cross product.
The generator $\iota_{1*}(c_1)\otimes a$ can be nontrivial if $n-k\geq 2$, 
$\iota_{2*}(c_1)\otimes a$ can be nontrivial if $k\geq 2$, 
$\iota_{1*}(d_{13})\otimes a$ can be nontrivial if $n-k\geq 4$, 
$\iota_{2*}(d_{13})\otimes a$ can be nontrivial if $k\geq 4$, and
$\iota_{1*}(e_1)\otimes \iota_{2*}(e_1)\otimes a$ can be nontrivial if both $k, n-k\geq 2$.
\end{proposition}

\begin{proof}
It should be clear that $\sg{n-k,k}$ has a presentation gotten by taking the disjoint union of the presentations of $\sg{n-k}$ and $\sg{k}$, and adding the relations that all generators from one factor commute with all generators from the other.
Taking the abelianization of this presentation, we get the description of $H_1(\sg{n-k,k};A)$ above for $A=\Z$.
Since $H_0(\sg{n-k,k};\Z)=\Z$ has no torsion, the universal coefficients theorem then implies that 
\[H_1(\sg{n-k,k};A)=H_1(\sg{n-k,k};\Z)\otimes A,\]
which fits the formula above.

For the second homology, we use the K\"unneth formula, with coefficients in $\Z$:
\[
\begin{split}
0\to \bigoplus_{i+j=2}H_i(\sg{n-k})\otimes H_j(\sg{k})&
\to H_2(\sg{n-k,k}) \\
&\to \bigoplus_{i+j=1}\mathrm{Tor}_1(H_i(\sg{n-k}), H_j(\sg{k}))
\to 0.
\end{split}
\]
Since $H_0(\sg{n-k})$ and $H_0(\sg{k})$ have no torsion, there is no contribution from the cokernel of the sequence.
The contribution from the kernel of the sequence is via the homology cross product.
If both of $k$ and $n-k$ are at least $2$, then there is a factor of $\Z/2\Z$ that comes from $H_1(\sg{n-k})\otimes H_1(\sg{k})$.
There is a factor of $\Z/2\Z$ that comes from $H_2(\sg{n-k})$ if $n-k\geq 4$, and if $k\geq 4$ then there is another factor of $\Z/2\Z$ from $H_2(\sg{k})$.
This agrees with the description given above if $A=\Z$.

For general $A$, we apply the universal coefficients theorem.
The overall statement follows.
\end{proof}

To describe the low-dimensional cohomology, we need a way of describing the cocycles.
There are product projections $\pi_1\co \sg{n-k,k}\to \sg{n-k}$ and $\pi_2\co \sg{n-k,k}\to \sg{k}$, but these do not detect all the low-dimensional cohomology.
Let $G=(\Z/2\Z)^2$.
We also consider the map $\pi_3\co \sg{n-k,k}\to G$, given by $\pi_3(\sigma)=(\mathrm{sign}(\pi_1(\sigma)),\mathrm{sign}(\pi_2(\sigma)))$.
(Here $\mathrm{sign}(\sigma)$ is $0$ if $\sigma$ is even and $1$ if $\sigma$ is odd).
Let $Q_*$ be free resolution of $\Z$ as a $\Z G$-module given by the cellular chains of the universal cover of $\R P^\infty\times \R P^\infty$.
This makes sense because $\R P^\infty\times \R P^\infty$ is a $K(G,1)$-space; we mean for it to have the standard cellular structure where each $\R P^\infty$-factor has a single cell in every dimension, and where cells in the product correspond to pairs of cells in the factors.
Let $e_1$ and $e_2$ denote the $1$-cells and let $d_{12}$ denote the $2$-cell that bounds the commutator square of $e_1$ and $e_2$.
Fix an abelian group $A$.
For $r\in A[2]$, let $\hat\beta^0_r\co Q_2\to A$ be the cochain with $\hat\beta^0_r(d_{12})=r$, and which evaluates to $0$ on the other two $2$-cells.
This is a cocycle because it is trivial on the boundaries of all four generators for $Q_3$.

\begin{proposition}\label{pr:cohomstabilizer} 
Let $n\geq 2$, with $1\leq k\leq n/2$.
    Let $A$ be an abelian group, considered as an untwisted $\Z\sg{n-k,k}$-module.  Then
    \[
    H^1(\sg{n-k,k};A) = 
    \left\{
    \begin{array}{cl}
    0 & \text{if $(n,k)=(2,1)$,}\\
    A[2] & \text{if $k=1$,} \\
    A[2]^2 & \text{otherwise.}
    \end{array}
    \right.
    \]
    The map $A[2]^2\to H^1(\sg{n-k,k};A)$ by $(r,s)\mapsto [\pi_{1}^*\kappa^0_r+\pi_{2}^*\kappa^0_s]$ is an isomorphism if $k\geq 2$, and there is a similar isomorphism by $s\mapsto \pi_{1}^*\kappa^0_s$ if $k=1$.
    \[
   H^2(\sg{n-k,k};A) = 
    \left\{
    \begin{array}{cl}
0 & \text{if $(n,k)=(2,1)$,} \\
A/2A & \text{if $k=1$, $n-k\in \{2,3\}$,} \\
(A/2A)^2\oplus A[2] & \text{if $k,n-k\in \{2,3\}$,} \\
(A/2A)\oplus A[2] & \text{if $k=1$ and $n-k\geq 4$,} \\
(A/2A)^2\oplus A[2]^2 & \text{if $k\in\{2,3\}$ and $n-k\geq 4$,} \\
(A/2A)^2\oplus A[2]^3 & \text{otherwise.}
    \end{array}
    \right.
    \]
    The map $(A/2A)^2\oplus A[2]^3\to H^2(\sg{n-2,2};A)$ by 
    \[([q],[r],s,t,u)\mapsto [\pi_{1}^*\alpha^0_q+\pi_{2}^*\alpha^0_r+\pi_{1}^*\beta^0_s+\pi_{2}^*\beta^0_s+\pi_3^*\hat\beta^0_u]\]
    is an isomorphism if both $n-k,k\geq 4$.
    There are similar isomorphisms for smaller $n$; we need the term $\pi_1^*\alpha^0_q$ if $n-k\geq 2$, we need $\pi_2^*\alpha^0_r$ if $k\geq 2$, we need $\pi_1^*\beta^0_s$ if $n-k\geq 4$, and we need the $\pi_3^*\hat\beta^0_u$ if both $n-k,k\geq 2$.
\end{proposition}

\begin{proof}
The descriptions of the cohomology groups follow easily using Proposition~\ref{pr:stabilizerhomology}, using the universal coefficients theorem.

The generators follow from an argument similar to that in Proposition~\ref{pr:untwistedcohomreps}.
    For the argument to work, we need to check the evaluations of the given cocycles on homology chains.
Momentarily assume that both $n-k,k\geq 4$.
    We have 
    \[\pi_1^*\kappa_r^0(\iota_{1*}e_i)=r,\quad \pi_1^*\kappa_r^0(\iota_{2*}e_i)=0,\]
    \[\pi_2^*\kappa_r^0(\iota_{1*}e_i)=0,\quad \pi_2^*\kappa_r^0(\iota_{2*}e_i)=r,\]
    \[\pi_1^*\alpha^0_r(\iota_{1*}c_i)=r,\quad \pi_1^*\alpha^0_r(\iota_{1*}b_i)=0,\quad \pi_1^*\alpha^0_r(\iota_{1*}d_{ij})=0,\]
    \[\pi_1^*\alpha^0_r(\iota_{2*}c_i)=0,\quad \pi_1^*\alpha^0_r(\iota_{2*}b_i)=0,\quad \pi_1^*\alpha^0_r(\iota_{2*}d_{ij})=0,\]
    \[\pi_1^*\alpha^0_r(\iota_{1*}e_1\otimes \iota_{2*}e_1)=0,\]
    \[\pi_2^*\alpha^0_r(\iota_{1*}c_1)=0,\quad \pi_2^*\alpha^0_r(\iota_{1*}b_i)=0,\quad \pi_2^*\alpha^0_r(\iota_{1*}d_{ij})=0,\]
    \[\pi_2^*\alpha^0_r(\iota_{2*}c_1)=r,\quad \pi_2^*\alpha^0_r(\iota_{2*}b_i)=0,\quad \pi_2^*\alpha^0_r(\iota_{2*}d_{ij})=0,\]
    \[\pi_2^*\alpha^0_r(\iota_{1*}e_i\otimes \iota_{2*}e_i)=0,\]
    \[\pi_1^*\beta^0_r(\iota_{1*}c_i)=0,\quad \pi_1^*\beta^0_r(\iota_{1*}b_i)=0,\quad \pi_1^*\beta^0_r(\iota_{1*}d_{ij})=r,\]
    \[\pi_1^*\beta^0_r(\iota_{2*}c_i)=0,\quad \pi_1^*\beta^0_r(\iota_{2*}b_i)=0,\quad \pi_1^*\beta^0_r(\iota_{2*}d_{ij})=0,\]
    \[\pi_1^*\beta^0_r(\iota_{1*}e_i\otimes \iota_{2*}e_1)=0,\]
    \[\pi_2^*\beta^0_r(\iota_{1*}c_i)=0,\quad \pi_2^*\beta^0_r(\iota_{1*}b_i)=0,\quad \pi_2^*\beta^0_r(\iota_{1*}d_{ij})=0,\]
    \[\pi_2^*\beta^0_r(\iota_{2*}c_i)=0,\quad \pi_2^*\beta^0_r(\iota_{2*}b_i)=0,\quad \pi_2^*\beta^0_r(\iota_{2*}d_{ij})=r,\]
    \[\pi_1^*\beta^0_r(\iota_{1*}e_i\otimes \iota_{2*}e_1)=0,\]
    \[\pi_3^*\hat\beta^0_r(\iota_{1*}c_i)=0,\quad \pi_3^*\hat\beta^0_r(\iota_{1*}b_i)=0,\quad \pi_3^*\hat\beta^0_r(\iota_{1*}d_{ij})=0,\]
    \[\pi_3^*\hat\beta^0_r(\iota_{2*}c_i)=0,\quad \pi_3^*\hat\beta^0_r(\iota_{2*}b_i)=0,\quad \pi_3^*\hat\beta^0_r(\iota_{2*}d_{ij})=0,\]
    \[\pi_3^*\hat\beta^0_r(\iota_{1*}e_i\otimes \iota_{2*}e_1)=r.\]
    Momentarily assume that $n\geq 6$.
    From these evaluations, we see that $(r,s)\mapsto \pi_1^*\kappa_r^0+\pi_2^*\kappa_s^0$ is an isomorphism  $A[2]^2\to \hom(H_1(\sg{n-2,2}),A)$, which is therefore $H^1(\sg{n-2,2},A)$ by universal coefficients.
    We see that $(s,t)\mapsto \pi_1^*\beta^0_s+\pi_3^*\hat\beta^0_t$ is an isomorphism $A[2]^2\to \hom(H_2(\sg{n-2,2}),A)$, and $([q],[r])\mapsto \pi_1^*\alpha_q^0+\pi_2^*\alpha_r^0$ is an isomorphism $(A/2A)^2\to\mathrm{Ext}^1_\Z(H_1(\sg{n-2,2}),A)$.
    These assemble into an isomorphism as in the statement, by universal coefficients and the five-lemma.
    In the other cases, similar arguments work, modified as indicated in the statement.
\end{proof}

\section{Applications}

\subsection{Twisted cohomology in permutation modules}
We adopt the convention that, if the group is suppressed in the cohomology notation, but only the module is  given, then it is cohomology of $\sg{n}$ in that module.

We fix a commutative ring $R$.
Fix a positive integer $n$ and an integer $k$ with $1\leq k\leq n/2$.
Let $\underline{n}=\{1,2,3,\dotsc,n\}$, and let $\binom{\underline{n}}{k}$ denote the set of $k$-subsets of $\underline{n}$.
\begin{definition}\label{de:permmod}
The module $M^k=M^{(n-k,k)}_R$, the \emph{permutation module} for the partition $(n-k,k)$,
is the free $R$-module on the basis $\{v_S\,|\,S\in\binom{\underline{n}}{k}\}$.
We give $M^k$ the structure of an $R\sg{n}$-module by using the natural action of $\sg{n}$ on $\binom{\underline{n}}{k}$, and extending linearly.
\end{definition}
(More generally, given a young diagram, the tabloids are equivalence classes of tableaux under permutation of rows, and the permutation module is the free $R$-module on this set, but the definition in terms of $k$-subsets suffices for two-row diagrams.)

Again, we need to work explicitly, so we define cocycles that will represent our cohomology classes.
As above, we use $R[2]$ as notation for the $2$-torsion in $R$.
\begin{definition}
   All sums in this definition are over $k$-subsets $S\subseteq \underline{n}$, in other words, over elements $S\in \binom{\underline{n}}{k}$.  
Additional conditions on $S$ are added in some of the sums.

   For $r\in R[2]$, define $\kappa^k_r\co P_1\to M^k$ by 
$\kappa^k_r(e_i)=r\sum_{S}v_S$.

   For $r\in R[2]$, if $k\geq 2$, define $\hat \kappa^k_r\co P_1\to M^k$ by 
\[\hat \kappa^k_r(e_i)=r\sum_{i,i+1\in S}v_S.\]

    For $r\in R$, define $\alpha^k_r\co P_2\to M^k$ by $\alpha^k_r(c_i)=r\sum_{S}v_S$, $\alpha^k_r(d_{ij})=0$, and $\alpha^k_r(b_i)=0$.

   For $r\in R$, if $k\geq 2$, define $\hat \alpha^k_r\co P_2\to M^k$ by 
\[\hat \alpha^k_r(c_i)=r\sum_{i,i+1\in S}v_S,\]
with $\hat\alpha^k_r(d_{ij})=0$, and $\hat\alpha^k_r(b_i)=0$.
    
    For $r\in R[2]$, define $\beta^k_r\co P_2\to M^k$ by $\beta^k_r(c_i)=0$, $\beta^k_r(d_{ij})=r\sum_{S}v_S$, and $\beta^k_r(b_i)=0$.

    For $r\in R[2]$, if $k\geq 2$, define $\hat \beta^k_r\co P_2\to M^k$ by $\hat \beta^k_r(c_i)=0$, 
\[\hat \beta^k_r(d_{ij})=r\sum\big\{v_S\,\big|\,\abs{\{i,i+1,j,j+1\}\cap S}=2\big\},\] 
and $\hat \beta^k_r(b_i)=0$.

    For $r\in R[2]$, if $k\geq 4$, define $\tilde \beta^k_r\co P_2\to M^k$ by $\tilde \beta^k_r(c_i)=0$, 
\[\tilde \beta^k_r(d_{ij})=r\sum\{v_S\,|\,i,i+1,j,j+1\in S\},\] 
and $\tilde \beta^k_r(b_i)=0$.
\end{definition}

\begin{lemma}
    For all choices of $r$ as above, the functions $\kappa^k_r$, $\hat\kappa^k_r$, $\alpha^k_r$, $\hat \alpha^k_r$, $\beta^k_r$, $\hat \beta^k_r$, and $\tilde\beta^k_r$ are cocycles.
\end{lemma}

\begin{proof}
There is an $R \sg{n}$-module homomorphism $\iota\co R\to M^k$ by $r\mapsto r\sum_{S\in \binom{\underline{n}}{k}}v_S$.
    Then for all choices of $r$, we have $\kappa^k_r=\iota^*\kappa^0_r$, $\alpha^k_r=\iota^*\alpha^0_r$, and $\kappa^k_r=\iota^*\kappa^0_r$.
    This immediately implies that these three functions are cocycles.

For the other functions, we must evaluate them on boundaries.
We use a temporary notation to make this easier; for $i,j\in\underline{n}$, let $\chi(i,j)=\sum_{i,j\in S}v_S\in M^k$, so that $\hat\kappa^k_r(e_i)=r\chi(i,i+1)$.
Notice that for $\sigma\in \sg{n}$, we have $\sigma\chi(i,j)=\chi(\sigma(i),\sigma(j))$.
For $r\in R[2]$, we have
\[\hat\kappa^k_r(\partial c_i)=\hat\kappa^k_r((1+s_i)e_i)=r(1+s_i)\chi(i,i+1)=2r\chi(i,i+1)=0,\]
\[\hat\kappa^k_r(\partial d_{ij})=\hat\kappa^k_r((1-s_i)e_j -(1-s_j)e_i)=r((1-s_i)\chi(j,j+1)-(1-s_j)\chi(i,i+1))=0,\]
(since in this case $\abs{i-j}\geq 2$ so that $s_i$ fixes $j$ and $j+1$ and $s_j$ fixes $i$ and $i+1$) and,
\[\begin{split}
\hat\kappa^k_r(\partial b_i)&=\hat\kappa^k_r((1-s_i+s_{i+1}s_i)e_{i+1}-(1-s_{i+1}+s_is_{i+1})e_i)\\
&= r((1-s_i+s_{i+1}s_i)\chi(i+1,i+2)-(1-s_{i+1}+s_is_{i+1})\chi(i,i+1)) \\
&= r(\chi(i+1,i+2)-\chi(i,i+2)+\chi(i,i+1) \\
&\quad -\chi(i,i+1)+\chi(i,i+2)-\chi(i+1,i+2))=0
\end{split}
\]

Next we consider $\hat\alpha^k_r$, for $r\in R$.
Again, we use $\chi(i,j)$ to denote the sum of $v_S$ over $k$-subsets $S$ with $i,j\in S$, so that  $\hat\alpha^k_r(c_i)=r\chi(i,i+1)$.
Since $\hat\alpha^k_r$ is $0$ on cells of the form $d_{ij}$ and $b_i$, it is $0$ on boundaries of the form $\partial c^{33}_{ijk}$, $\partial c^{34}_{ij}$, and $\partial c^{37}_i$.
For $c^{32}_{ij}$, there is a multiple of $d_{ij}$ or $d_{ji}$ that maps to $0$, so that
\[\hat\alpha^k_r(\partial c^{32}_i)=\hat\alpha^k_r((s_j-1)c_i)=r(s_j-1)\chi(i,i+1)=0,\]
since in this case $\abs{i-j}\geq 2$ and $s_j$ fixes $i$ and $i+1$.
For the other boundaries,
\[\hat\alpha^k_r(\partial c^{31}_i)=
\hat\alpha^k_r((s_i-1)c_i)=r(s_i-1)\chi(i,i+1)=0,\]
\[\hat\alpha^k_r(\partial c^{35}_i)=\hat\alpha^k_r((s_{i+1}+1)b_i-c_{i+1}+s_is_{i+1}c_i)
=r(s_is_{i+1}\chi(i,i+1)-\chi(i+1,i+2))=0,\]
\[\begin{split}
\hat\alpha^k_r(\partial c^{36}_i)&=\hat\alpha^k_r((s_is_{i+1}-1)b_i+(s_{i+1}-1)(s_ic_{i+1}+c_i))\\
&=r(s_{i+1}s_i\chi(i+1,i+2)+s_{i+1}\chi(i,i+1)-s_i\chi(i+1,i+2)-\chi(i,i+1))\\
&=0.
\end{split}\]

Now we look at $\hat\beta^k_r$, for $r\in R[2]$.
We establish a temporary notation, and set
\[\zeta(i,j,p,q)=\sum\{v_S\,|\, \text{$S\in \binom{\underline{n}}{k}$ and $\abs{\{i,j,p,q\}\cap S}=2$}\},\]
so $\zeta(i,j,p,q)$ is the sum of $v_S$ over all choices of $S$ that contain exactly two of $i$, $j$, $p$, and $q$.
Then $\hat\beta^k_r(d_{ij})=r\zeta(i,i+1,j,j+1)$.
Also, for $\sigma\in \sg{n}$, we have $\sigma\zeta(i,j,p,q)=\zeta(\sigma(i),\sigma(j),\sigma(p),\sigma(q))$.
We evaluate $\hat\beta^k_r$ on boundaries; since it is zero on cells of the form $c_i$ and $b_i$, it is zero on $\partial c^{31}_i$, $\partial c^{35}_i$, and $\partial c^{36}_i$.
For $c^{32}_{ij}$, we assume $i<j$; the computation is almost the same if $j<i$:
\[\hat\beta^k_r(\partial c^{32}_{ij})=\hat\beta^k_r((s_j-1)c_i-(s_i+1)d_{ij})
=-r(s_i+1)\zeta(i,i+1,j,j+1)=0,\]
since $s_i$ fixes $\{i,i+1,j,j+1\}$ setwise and $2r=0$.
For $c^{33}_{ijl}$, we know that $j\geq i+2$ and $l\geq j+2$.  So
\[
\begin{split}
\hat\beta^k_r(\partial c^{33}_{ijl})&=
r((s_i-1)\zeta(j,j+1,l,l+1)-(s_j-1)\zeta(i,i+1,l,l+1)\\
&\quad
+(s_l-1)\zeta(i,i+1,j,j+1))\\
&=0.
\end{split}
\]
For $c^{34}_{ij}$, we again assume that $i<j$ (really, that $i+1<j$) since the case for the reverse inequality is almost the same:
\[
\begin{split}
\hat\beta^k_r(\partial c^{34}_{ij})
&=r(-(1-s_i+s_{i+1}s_i)\zeta(i+1,i+2,j,j+1) \\
&\quad +(1-s_{i+1}+s_is_{i+1})\zeta(i,i+1,j,j+1))\\
&=r(-\zeta(i+1,i+2,j,j+1)+\zeta(i,i+2,j,j+1)-\zeta(i,i+1,j,j+1) \\
&\quad +\zeta(i,i+1,j,j+1)-\zeta(i,i+2,j,j+1)+\zeta(i+1,i+2,j,j+1))\\
&=0.
\end{split}
\]
Finally, we explain what happens with $c^{37}_i$.
We can ignore the multiples of $b_i$ and $b_{i+1}$, since these map to $0$.
Then $\hat\beta^k_r$ maps $\partial c^{37}_i$ to 
\[r(1-s_{i+1}+s_is_{i+1}-s_{i+2}s_is_{i+1}+s_{i+2}s_{i+1}+s_{i+1}s_{i+2}s_is_{i+1})\zeta(i,i+1,i+2,i+3).\]
Since each of the six permutations here fixes $\{i,i+1,i+2,i+3\}$ setwise, we have
\[\hat\beta^k_r(\partial c^{37}_i)=2r\zeta(i,i+1,i+2,i+3)=0,\]
since $2r=0$.

Finally, we need to argue that $\tilde\beta^k_r$ is a cocycle.
If we change the meaning of $\zeta(i,j,p,q)$ to be the sum of $v_S$ with $i,j,p,q\in S$, then exactly the same computations that we just did to show that $\hat\beta^k_r$ is a cocycle will show the same thing for $\tilde \beta^k_r$.
\end{proof}

As above, we assume that $n\geq 2k$ in the following statement.
\begin{proposition}\label{pr:Mkcohomology}
We have $H^*(\sg{n};M^k)=H^*(\sg{n,n-k},R)$, and $H^1(\sg{n};M^k)$ and $H^2(\sg{n};M^k)$ are generated by the classes of the cocycles just given.
In particular, if $k\geq 2$, then 
\[R[2]^2\to H^1(\sg{n};M^k), \quad (r,s)\mapsto [\kappa^k_r+\hat\kappa^k_s]\]
is an isomorphism, and if $k\geq 4$, then
\[(R/2R)^2\oplus R[2]^3\to H^2(\sg{n};M^k),\quad ([r],[s],t,u,v)\mapsto
[\alpha^k_r+\hat\alpha^k_s+\beta^k_t+\hat\beta^k_u+\tilde\beta^k_v]\]
is an isomorphism.

There are similar isomorphisms for small values of $k$ and $n-k$; specifically:
\begin{itemize}
    \item     If $n\geq 3$, then $R[2]\to H^1(\sg{n};M^1)$ by $r\mapsto [\kappa^1_r]$ is an isomorphism.
    \item   For $n=2$, $H^1(\sg{2};M^1)=0$.
    \item If $n\geq 5$, then $R[2]\oplus R/2R\to H^2(\sg{n};M^1)$ by $(r,[s])\mapsto [\alpha^1_s+\beta^1_r]$ is an isomorphism.
    \item
    If $n\in\{3,4\}$, then $R/2R\to H^2(\sg{n};M^1)$ by $[s]\mapsto[\alpha^1_s]$ is an isomorphism.
    \item 
    If $n=2$, then $H^2(\sg{2};M^1)=0$.
    \item
    If $n\in \{2,3\}$, then $R[2]\to H^1(\sg{n};M^2)$ by $r\mapsto [\kappa^2_r]$ is an isomorphism.
    \item     
    If $n\geq6$, then $R[2]^2\oplus (R/2R)^2\to H^2(\sg{n};M^2)$ by $(q, r,[s],[t])\mapsto [\alpha^2_s+\hat\alpha^2_t+\beta^2_q+\hat\beta^2_r]$ is an isomorphism.
    \item
    If $n\in\{4,5\}$, then $R[2]\oplus (R/2R)^2\to H^2(\sg{n};M^2)$ by $( r,[s],[t])\mapsto [\alpha^2_s+\hat\alpha^2_t+\hat\beta^2_r]$ is an isomorphism.
\end{itemize}
\end{proposition}

\begin{proof}
Let $T=\{n-k+1,n-k+2,\dotsc,n\}$.
Since this is a $k$-subset, $v_T$ is a basis element for $M^k$, and $\sg{n-k,k}$ is the stabilizer of $v_T$ in $\sg{n}$.  
Let $\pi_T\co M^k\to R$ be the coordinate projection for $v_T$.
    By Proposition~\ref{pr:freeascoinduced}, we have an isomorphism $\psi\co M^k\to \mathrm{Coind}_{\sg{n-k,k}}^{\sg{n}}(R)$ by $\psi_v(g)=\pi_T(g^{-1}v)$.
    So we apply Shapiro's lemma, and by Proposition~\ref{pr:explicitShapiro}, the isomorphism on cohomology is induced by inclusion $\sg{n-k,k}\to \sg{n}$ and evaluation $\mathrm{Coind}_{\sg{n-k,k}}^{\sg{n}}(R)\to R$, $f\mapsto f(1)$.
    Composing these, we get the isomorphism $H^*(\sg{n};M^k)\cong H^*(\sg{n-k,k};R)$ by restricting to $\sg{n-k,k}$ and mapping $M^k\to R$ by $v\mapsto \pi_T(v)$.
Let $\eta$ denote this map on cochains.

Now we look at $H^1(\sg{n};M^k)$;
we use the notation from Proposition~\ref{pr:cohomstabilizer}.
Fix $r\in R[2]$.
Since $\kappa^k_r$ outputs $r(\sum_S v_S)$ on all $1$-cell inputs, we have that $\pi_T\circ \kappa^k_r$ outputs $r$ on all such inputs.
So we can recognize $[\eta\kappa^k_r]=[\pi_1^*\kappa^0_r+\pi_2^*\kappa^0_r]$.
On the other hand, $\pi_T\circ \hat\kappa^k_r$ is more interesting.
It outputs $r$ on inputs of $e_{n-k+1}$ through $e_{n-1}$, but outputs $0$ on other $1$-cells. 
(This is because this is the range of $e_i$ such that $i,i+1\in T$.)
So $[\eta\kappa^k_r]=[\pi_2^*\kappa^0_r]$.
This implies that the map in the statement is an isomorphism, since the isomorphism in Proposition~\ref{pr:cohomstabilizer} factors through it.

We do a similar analysis of $H^2(\sg{n};M^k)$.
Fix $r\in R$.
First, $[\eta\alpha^k_r]=[\pi_1^*\alpha^0_r+\pi_2^*\alpha^0_r]$, since $\pi_T\circ\alpha^k_r$ outputs $r$ for all $c_i$-cells, and $0$ for other $2$-cells.
Also, $[\eta\hat\alpha^k_r]=[\pi_2^*\alpha^0_r]$, since $\pi_T\circ \alpha^k_r$ only outputs $r$ for $c_{n-k+1}$ through $c_{n-1}$, and outputs $0$ for other $2$-cells.
(Again, this is because this is the range of $c_i$ such that $i,i+1\in T$.)
Now fix $r\in R[2]$.
We have $[\eta\beta^k_r]=[\pi_1^*\beta^0_r+\pi_2^*\beta^0_r+\pi_3^*\hat\beta^0_r]$, since $\pi_T\circ \beta^k_r$ outputs $r$ for all $d_{ij}$-cells, and $0$ for all other $2$-cells.
We have $[\eta\hat\beta^k_r]=[\pi_3^*\hat\beta^0_r]$, since $\pi_T\circ \hat\beta^k_r$ returns $r$ only for those $d_{ij}$-cells with $1\leq i\leq n-k-1$ and $n-k+1\leq j\leq n-1$, and it returns $0$ for all other $2$-cells.
(These are the choices of $d_{ij}$ with exactly two of $i$,$i+1$, $j$, and $j+1$ in $T$.)
Finally, we have 
$[\eta\tilde\beta^k_r]=[\pi_2^*\beta^k_r]$, since $\pi_T\circ\tilde\beta^k_r$ only returns $r$ for those $d_{ij}$-cells with $n-k+1\leq i<j\leq n-1$, and returns $0$ for other $2$-cells.
(These are the choices  of $d_{ij}$ with $i$, $i+1$, $j$, $j+1\in T$.)
Again, this implies that the map in the statement is an isomorphism, since the isomorphism in Proposition~\ref{pr:cohomstabilizer} factors through it.

The special cases for small values of $k$ and $n-k$ follow by a similar argument, using the special cases of the cohomology of $\sg{n-k,k}$ in Proposition~\ref{pr:cohomstabilizer}.
\end{proof}

\subsection{Twisted homology in permutation modules}
Since the index of $\sg{n-k,k}$ in $\sg{n}$ is finite, $M^k$ is isomorphic to both $\mathrm{Coind}_{\sg{n-k,k}}^{\sg{n}}(R)$ and $\mathrm{Ind}_{\sg{n-k,k}}^{\sg{n}}(R)$.
However, we would like to give explicit generators for low-dimensional homology similar to what we did for cohomology in Proposition~\ref{pr:Mkcohomology}, so we give an explicit isomorphism and explain how to exploit Shapiro's lemma for homology explicitly.
Recall that if $G$ is a group, $H\leq G$, and $M$ is a $\Z H$-module, then $\mathrm{Ind}_H^G(M)$ is $\Z G\otimes_{\Z H} M$, where we are considering $\Z G$ as a left $\Z G$-module and a right $\Z H$-module (by the inclusion of $H$ into $G$).

\begin{lemma}\label{le:inducedasfree}
Let $G$ be a group acting transitively on a set $S$.
Let $p\in S$ and let $H\leq G$ be the stabilizer of $p$.
Let $R$ be a commutative ring, and let $M$ be the free $R$-module with basis $S$, given the structure of an $R G$-module by extending the action on $S$ linearly.
Then the map $\phi\co \mathrm{Ind}_H^G(R)\to M$ by $g\otimes r\mapsto rg\cdot p$ is an isomorphism of $R G$-modules.
\end{lemma}

\begin{proof}
The universal property of tensor products implies that $\phi$ is a well defined homomorphism of $R G$-modules.
As an $R$-module, $\mathrm{Ind}_H^G(R)$ is a free $R$-module with basis given by the elements $g\otimes 1$, where $g$ ranges over a set of left coset representatives of $H$ in $G$.
Since $G$ acts transitively on $S$ and $H$ is the stabilizer of $p$, $\phi$ sends this basis bijectively to $S\subset M$.
This implies that $\phi$ is an isomorphism of $R$-modules, and therefore an isomorphism of $R G$-modules.
\end{proof}

\begin{proposition}\label{pr:explicithomShapiro}
Let $H\leq G$.
Let $F_*$ and $Q_*$ be projective resolutions of $\Z$ as $\Z H$- and $\Z G$-modules, respectively.
Let $f\co F_*\to Q_*$ be an augmentation-preserving $\Z H$-chain module map.
Let $M$ be a $\Z H$-module.
Then the map
\[F_*\otimes_H M \to Q_*\otimes_G \Z G\otimes_H M, \quad c\otimes m\mapsto f(c)\otimes 1\otimes m\]
induces the isomorphism from Shapiro's lemma: $H_*(H;M)\cong H_*(G;\mathrm{Ind}_H^G(M))$.
\end{proposition}

\begin{proof}
For clarity, we note that $Q_*$ is also a chain complex of $\Z H$-modules since $H$ is a subgroup of $G$, and a map such as $f$ will exist and will be unique up to homotopy by standard arguments.
The proposition is then immediate from the proof of Shapiro's lemma in a standard reference, for example Theorem~III.6.2 of Brown~\cite{Brown}.
\end{proof}

\begin{proposition}\label{pr:Mkhomology}
Let $n$ and $k$ be positive integers with $1\leq k\leq n/2$.
Then $H_*(\sg{n};M^k)\cong H_*(\sg{n-k,k};R)$.

Let $T=\{n-k+1,n-k+2,\dotsc,n\}$, so that $v_T$ is the basis element in $M^k$ with stabilizer $\sg{n-k,k}$.
If $k\geq 2$, then
\[(R/2R)^2\to H_1(\sg{n};M^k),\quad ([r],[s])\mapsto [e_1\otimes rv_T + e_{n-k+1}\otimes sv_T]\]
is an isomorphism.
If $k\geq 4$, then
\[R[2]^2\oplus (R/2R)^3 \to H_2(\sg{n};M^k),\]
\[
\begin{split}
&(q,r,[s],[t],[u])\mapsto \\
&\quad [c_1\otimes qv_T + c_{n-k+1}\otimes rv_T+d_{13}\otimes sv_T + d_{1,n-k+1}\otimes tv_T+d_{n-k+1,n-k+3}\otimes tv_T]
\end{split}
\]
is an isomorphism.
There are similar isomorphisms for smaller values of $k$ and $n$.
\end{proposition}

\begin{proof}
By Lemma~\ref{le:inducedasfree} and Proposition~\ref{pr:explicithomShapiro}, we have isomorphisms
\[H_*(\sg{n-k,k},R)\to H_*(\sg{n};\mathrm{Ind}_{\sg{n-k,k}}^{\sg{n}}(R))\to H_*(\sg{n};M^k),\]
since $M^k$ is free on a basis that is a transitive $\sg{n}$-set with $\sg{n-k,k}$ being the stabilizer of $v_T$.
To see what happens to cocycles, we take our resolution $F_*$ for $\sg{n-k,k}$ to be $P^{n-k}_*\otimes P^k_*$, where $P^m_*$ is a completion of the partial resolution for $\sg{m}$ from Definition~\ref{de:Snres}.
We take our resolution $Q_*$ for $\sg{n}$ to be $P^n_*$.
Our chain map $f\co F_*\to Q_*$ is an extension of
\begin{itemize}
\item $*\otimes *\mapsto *$,
\item $e_i\otimes *\mapsto e_i$, $*\otimes e_i\mapsto e_{n-k+i}$,
\item $b_i\otimes *\mapsto b_i$, $*\otimes b_i\mapsto b_{n-k+i}$,
\item $c_i\otimes *\mapsto c_i$, $*\otimes c_i\mapsto c_{n-k+i}$,
\item $d_{ij}\otimes *\mapsto d_{ij}$, $*\otimes d_{ij}\mapsto d_{n-k+i,n-k+j}$,
\item $e_i\otimes e_j\mapsto -d_{i,n-k+j}$.
\end{itemize}
Such a map $f$ exists because the map we defined in dimensions $0$, $1$, and $2$
is an augmentation-preserving $\sg{n-k,k}$-chain map, and it can be extended to higher dimensions by standard methods since these modules are projective.
To see that this is an $\sg{n-k,k}$-chain map, we check that it respects boundaries on a few examples (the other cases are similar).
Let $\iota_1\co \sg{n-k}\to\sg{n-k,k}$ and $\iota_2\co \sg{k}\to\sg{n-k,k}$ be the product inclusion maps.
Then 
$\partial (e_i\otimes *)=(\iota_1(s_i)-1)*\otimes *$, which maps to $(s_i-1)*$ under $f$, which is $\partial e_i$.
Likewise,
$\partial (*\otimes e_i)=(\iota_2(s_i)-1)*\otimes *$, which maps to $(s_{n-k+i}-1)*$ under $f$, which is $\partial e_{n-k+i}$.
The interesting $2$-cell is $e_i\otimes e_j$; under usual conventions for tensor products of chain complexes, 
\[\partial (e_i\otimes e_j)=(\iota_1(s_i)-1)*\otimes e_j-(\iota_2(s_j)-1)e_i\otimes *.\]
Then
\[f(\partial (e_i\otimes e_j))=(s_i-1)e_{n-k+j}-(s_{n-k+j}-1)e_i.\]
Since this is  $\partial(-d_{i,n-k+j})$, this choice for $f(e_i\otimes e_j)$ works.
We find our generators for $H_i(\sg{n};M^k)$  ($i=1,2$) by taking our generators for $H_i(\sg{n-k,k};R)$ from Proposition~\ref{pr:stabilizerhomology} and sending them through the composition of maps above.
Since $v_T$ is our basepoint, our map $\mathrm{Ind}_{\sg{n-k,k}}^{\sg{n}}(R)\to M^k$ sends $g\otimes 1$ to $gv_T$ for $g\in \sg{n}$.
Then our generators $\iota_{1*}e_1$, $\iota_{1*}c_1$, and $\iota_{1*}d_{13}$ map to $e_1\otimes v_T$, $c_1\otimes v_T$, and $d_{13}\otimes v_T$, respectively.
Our generators in the image of $\iota_2$ are in the second factor, and get shifted, so that $\iota_{1*}e_1$, $\iota_{1*}c_1$, and $\iota_{1*}d_{13}$ map to $e_{n-k+1}\otimes v_T$, $c_{n-k+1}\otimes v_T$, and $d_{n-k+1,n-k+3}\otimes v_T$, respectively.
Finally, we have that $\iota_{1*}e_1\otimes \iota_{2*}e_1$ maps to
$-d_{1,n-k+1}\otimes v_T$.
\end{proof}

\subsection{Untwisted third homology}\label{ss:untwisted3rd}

Our starting point is the following lemma, which is adapted from Lemma II.5.1 of Brown~\cite{Brown}.
\begin{lemma}\label{le:upstairscycles}
Let $G$ be a group and let $n$ be a positive integer.
Let $F_n\to \dotsm\to F_0\to \Z\to 0$ be an exact sequence of $\Z G$-modules where each $F_i$ is projective.
Then $H_n(G)=\coker(H_n(F_*)\to H_n((F_*)_G))$.
\end{lemma}

\begin{proof}
The partial resolution $F_*$ can be completed to a full resolution $F'_*$ by finding a projective module $F'_{n+1}$ with a surjection $F'_{n+1}\to \ker(F_n\to F_{n-1})$, and likewise for higher $F'_k$.
Then $H_n(G)$ is the homology of the coinvariants of $F'_*$.
So $H_n(G)=\ker((F_{n})_G\to (F_{n-1})_G)/\im((F'_{n+1})_G\to (F_n)_G)$.
But $\ker((F_n)_G\to (F_{n-1})_G)=H_n((F_*)_G)$, 
and $\im((F'_{n+1})_G\to (F_n)_G)$ is the image of $F'_{n+1}\to F_n$ composed with $F_n\to (F_n)_G$.
We chose $F'_{n+1}$ so that the image of $F'_{n+1}\to F_n$ is $\ker(F_n\to F_{n-1})=H_n(F_*)$.
This proves the lemma. 
\end{proof}

A more intuitive interpretation of Lemma~\ref{le:upstairscycles} is that, if $X$ is an $n$-dimensional approximation to a $K(G,1)$-complex and $\widetilde X$ is its universal cover, then $H_n(G)$ is the cokernel of the covering map $H_n(\widetilde X)\to H_n(X)$.
The idea is that we can recover the boundaries that would be filled by the missing $(n+1)$-skeleton by finding the $n$-cycles that actually come from the universal cover.
These are the cycles that would be filled by cells in building the $(n+1)$-skeleton in the next step of the approximation to the $K(G,1)$-complex.

Our plan is to use Lemma~\ref{le:upstairscycles} to get an upper bound to $H_3(\sg{n})$; in other words, to find a generating set for $H_3(\sg{n})$ and enough relations until we get a conjecturally correct picture of $H_3(\sg{n})$.
Then we will show that the cycles are nontrivial by other methods in order to prove Theorem~\ref{th:untwistedthird}.

Now, fix $n\geq 6$.
Let $X$ be the finite cellular $3$-complex described by Definition~\ref{de:Snres}.
This means that $X$ has a single $0$-cell $*$, $n-1$ $1$-cells $e_1,\dotsc, e_{n-1}$, and $2$- and $3$-cells according to the several families in that definition.
The cells are glued according to the boundaries in that definition, so that the complex $P_*$ we define there is the complex of cellular chains for $X$.
Let $\widetilde X$ be the universal cover of $X$.
By Lemma~\ref{le:upstairscycles}, $H_3(\sg{n})=\coker(H_3(\widetilde X)\to H_3(X))$.
Our generators for $P_*$ denote cells in $\widetilde X$, where the formulas for their boundaries correctly account for the $\sg{n}$-action on $\widetilde X$ by deck transformations.
We use barred versions of these notations for our cells in $X$, or equivalently, for their coinvariance classes.

\begin{lemma}\label{le:Xhomology}
$H_3(X)$ is generated by the following classes, for all choices of $i$, $j$, and $k$ that make sense, with $i<j<k$:
\[\bar c^{31}_i, \bar c^{32}_{ij}+\bar c^{32}_{ji}, \bar c^{33}_{ijk}, 2\bar c^{34}_{ij}-\bar c^{32}_{i+1,j}+\bar c^{32}_{ij}, 2\bar c^{34}_{ji}-\bar c^{32}_{i,j+1}+\bar c^{32}_{ij}, \bar c^{36}_i, \text{ and } \bar c^{37}_i+\bar c^{32}_{i,i+2}.\]
\end{lemma}

\begin{proof}
The cellular $3$-chains of $X$ are the same as the coinvariants $(P_3)_{\sg{n}}$, so we compute the boundaries of the generators for $P_3$, ignoring the  $\sg{n}$-action.

We record the results here.
Since $c^{32}_{ij}$ and $c^{34}_{ij}$ have different boundaries depending on whether $i<j$, we assume $i<j$ and record both orders.
\begin{itemize}
\item $\partial \bar c^{31}_i=0$.
\item $\partial \bar c^{32}_{ij}=-2 \bar d_{ij}$.
\item $\partial \bar c^{32}_{ji}=2\bar d_{ij}$.
\item $\partial \bar c^{33}_{ijk}=0$.
\item $\partial \bar c^{34}_{ij}=\bar d_{ij}-\bar d_{i+1,j}$.
\item $\partial \bar c^{34}_{ji}=\bar d_{i,j+1}-\bar d_{ij}$.
\item $\partial \bar c^{35}_i=2\bar b_i+\bar c_i - \bar c_{i+1}$.
\item $\partial \bar c^{36}_i=0$.
\item $\partial \bar c^{37}_i=2\bar d_{i,i+2}$.
\end{itemize}
Now assume that we have a chain $x=\sum a_i\bar c_i$ for some $a_i\in \Z$ and $\bar c_i$ among the generators for $(P_3)_{\sg{n}}$, and assume $\partial x=0$.
We will modify $x$ using the generators given above until we are left with the zero chain.
First of all, if any of the cells in $x$ have the form $\bar c^{31}_i$, $\bar c^{33}_{ijk}$, or $\bar c^{36}_i$, then we simply remove these from the sum; since they have zero boundary, this does not change our assumption that $x$ has zero boundary.
Next, if any cell of the form $\bar c^{34}_{ij}$, $\bar c^{34}_{ji}$, or $\bar c^{37}_{i}$ appears, we add an appropriate multiple of $2\bar c^{34}_{ij}-\bar c^{32}_{i+1,j}+\bar c^{32}_{ij}$, $2\bar c^{34}_{ji}-\bar c^{32}_{i,j+1}+\bar c^{32}_{ji}$, or $\bar c^{37}_i+\bar c^{32}_{i,i+2}$.
So we assume that no cells of these forms appear in $x$.
Next we note that no cell of the form $\bar c^{35}_i$ can appear in $x$ with a nonzero coefficient, since if it did, then a cell of the form $\bar b_i$ would appear in $\partial x$ with nonzero coefficient.
From this, we deduce that all the cells in $x$ have the form $\bar c^{32}_{ij}$  or $\bar c^{32}_{ji}$.
Since $\partial x=0$, we deduce that for each pair $\{i,j\}$, we have the same coefficient on $\bar c^{32}_{ij}$ as $\bar c^{32}_{ji}$.
So we subtract off the appropriate multiple of $\bar c^{32}_{ij}+\bar c^{32}_{ji}$, for each pair, and reduce $x$ to zero.
\end{proof}

Our next goal is to give a recipe for building several useful chains in $P_*$.
The idea is to pick a consecutive cycle $\alpha=(i,i+1,\dotsc,j)$ and a $k$-cell $c_I$ in $P_k$ for  $k\in\{0,1,2\}$, and build a $(k+1)$-chain $\Pi(\alpha,c_I)$ in $P_{k+1}$ that represents a prism bounding $c_I$ and $\alpha c_{\alpha(I)}$:
\[\partial \Pi(\alpha, c_I)=\alpha c_{\alpha(I)}-c_I-\Pi(\alpha,\partial c_I).\]
However, for this to work, we must pick $c$ so that this conjugate is also a single cell.
So what we are about to define is a kind of ``partial prism operator" on $P_*$.

\begin{definition}
Let $i<j$ be from $\{1,2,\dotsc, n\}$ and let $\alpha=(i,i+1,\dotsc, j)$.
Pick a cell in $P_*$ of dimension $0$, $1$, or $2$ and write it as $\sigma c_I$, where $\sigma\in \sg{n}$ and $c$ is one of the cells in our basis.
For the following construction to be defined, we demand that all $1$-cells on $c$ are translates of $1$-cells $e_k$, where $k\neq i$ and $k\neq j$.
Define $\Pi(\alpha,\sigma c)$ as follows:
\begin{itemize}
\item If $\sigma\neq (1)$, then $\Pi(\alpha,\sigma c)=\sigma \Pi(\alpha,c)$.
\item If $c=*$, then 
\[\Pi(\alpha, *)= e_i + s_i e_{i+1}+\dotsc+ s_i s_{i+1}\dotsm s_{j-2} e_{j-1}.\]
\item If $c=e_k$, we define $\Pi$ in two steps.
\begin{itemize}
\item If $\alpha=s_i$, then $k\neq i$ and $k\neq i+1$.
We define 
\[\Pi(s_i,e_k)=
\left\{\begin{array}{cl} - d_{ik} & \text{if $i<k$} \\
 d_{ki} & \text{if $k<i$.}\end{array}\right.
\]
\item We define
\[\Pi(s_i s_{i+1}, e_{i+1})=-b_i.\]
\end{itemize}
In general, if $k<i$ or $k>j$, then
\[\Pi((i,i+1,\dotsc, j),e_k) =\sum_{l=i}^{j-1} s_is_{i+1}\dotsc s_{l-1}\Pi(s_l,e_k).\]
If $i<k<j$, then
\[
\begin{split}
\Pi((i,i+1,\dotsc, j),e_k)&= \Pi((i,i+1,\dotsc, k-1),e_k) \\
&\quad +(i,i+1,\dotsc, k-1)\cdot \Pi(s_{k-1}s_k,e_k)\\
&\quad +(i,i+1,\dotsc, k+1)\cdot\Pi((k+1,k+2,\dotsc, j),e_{k-1}).
\end{split}
\]
\item 
If $c=c_k$, then:
\begin{itemize}
\item if $\alpha=s_i$ and $\abs{i-k}\geq 2$, then
\[\Pi(s_i,c_k)=c^{32}_{ki};\]
\item if $\alpha=s_{k-1}s_k$, then
\[\Pi(s_{k-1}s_k,c_k)=c^{35}_{k-1}.\]
\end{itemize}
Like with $e_k$, if $k<i$ or $k>j$, then
\[\Pi((i,i+1,\dotsc, j),c_k) =\sum_{l=i}^{j-1} s_is_{i+1}\dotsc s_{l-1}\Pi(s_l,c_k).\]
If $i<k<j$, then
\[
\begin{split}
\Pi((i,i+1,\dotsc, j),c_k)&= \Pi((i,i+1,\dotsc, k-1),c_k) \\
&\quad+(i,i+1,\dotsc, k-1)\cdot \Pi(s_{k-1}s_k,c_k) \\
&\quad +(i,i+1,\dotsc, k+1)\cdot\Pi((k+1,k+2,\dotsc, j),c_{k-1}).
\end{split}
\]
\item
If $c=d_{jk}$, then:
\begin{itemize}
\item
if $\alpha=s_i$ and $\abs{i-j}\geq 2$ and $\abs{i-k}\geq2$, then
\[\Pi(s_i, d_{jk})=
\left\{\begin{array}{cl}
c^{33}_{ijk} & \text{if $i<j$} \\
-c^{33}_{jik} & \text{if $j<i<k$} \\
c^{33}_{jki} & \text{if $k<i$;}
\end{array}
\right.
\]
\item
if $\alpha=s_{j-1}s_j$, then
\[\Pi(s_{j-1}s_j,d_{jk})=c^{34}_{j-1,k};\]
\item
if $\alpha=s_{k-1}s_k$, then 
\[\Pi(s_{k-1}s_k,d_{jk})=-c^{34}_{k-1,j}.\]
\end{itemize}
Generally, if neither $k$ nor $l$ is in $\{i,i+1,\dotsc, j\}$, then
\[\Pi((i,i+1,\dotsc,j),d_{kl})=\sum_{m=i}^{j-1}s_is_{i+1}\dotsm s_{m-1}\Pi(s_m,d_{kl}),\]
and otherwise
\[
\begin{split}
\Pi((i,i+1,\dotsc,j),d_{kl})&=\Pi((i,i+1,\dotsc,k-1),d_{kl})\\
&\quad+(i,i+1,\dotsc,k-1)\cdot\Pi(s_{k-1}s_k,d_{kl})\\
&\quad+(i,i+1,\dotsc,k+1)\cdot\Pi((k+1,k+2,\dotsc,l-1),d_{k-1,1})\\
&\quad+(i,i+1,\dotsc,l-1)\cdot\Pi(s_{l-1,l},d_{k-1,l})\\
&\quad+(i,i+1,\dotsc,l+1)\cdot\Pi((l+1,l+2,\dotsc,j),d_{k-1,l-1}).
\end{split}
\]
\item
If $c=b_k$, then
\begin{itemize}
\item if $\alpha=s_i$ with $i+1<k$ or $i>k+2$, then
\[\Pi(s_i,b_k)=c^{34}_{ik};\]
\item if $\alpha=s_{k-1}s_ks_{k+1}$, then
\[\Pi(s_{k-1}s_ks_{k+1},b_k)=c^{37}_{k-1}.\]
\end{itemize}
Generally, if $k>j$ or $k+2<i$ then
\[\Pi((i,i+1,\dotsc,j),b_k)=\sum_{l=i}^{j-1}s_is_{i+1}\dotsm s_{l-1}\Pi(s_l,b_k),\]
and if $k-1\geq i$ and $k+1\leq j$, then
\[
\begin{split}
\Pi((i,i+1,\dotsc, j), b_k)&=\Pi((i,i+1,\dotsc, k-1),b_k)\\
&\quad+(i,i+1,\dotsc, k-1)\cdot\Pi(s_{k-1}s_ks_{k+1},b_k)\\
&\quad+(i,i+1,\dotsc, k+2)\cdot\Pi((k+2,k+3,\dotsc ,j),b_k).
\end{split}
\]
\end{itemize}
\end{definition}

The point of the definition is the following prism operator property, and its corollary about homology classes.
\begin{proposition}
Let $\alpha$ be a consecutive cycle $(i,i+1,\dotsc,j)$ and let $c_I\in P_m$ be a basis cell, for $m\in \{0,1,2\}$, and where $I$ is the list of indices in the definition of the cell.
Suppose that $c_I$ does not contain any $1$-cells that are translates of $e_i$ or $e_j$.
Then $\Pi(\alpha,c_I)\in P_{m+1}$ is defined, and
\[\partial \Pi(\alpha,c_I) = \alpha c_{\alpha\cdot I} -c_I -\Pi(\alpha,\partial c_I).\]
\end{proposition}

\begin{proof}
This is by construction in the definition of $\Pi$.
It is straightforward to verify this by copying out the definitions of the boundaries in the cases of 
\begin{itemize}
\item $\Pi(s_i,c_k)$, 
\item $\Pi(s_{k-1}s_k,c_k)$, 
\item $\Pi(s_i, d_{jk})$, 
\item $\Pi(s_{j-1}s_j,d_{jk})$, 
\item $\Pi(s_{k-1}s_k,d_{jk})$, 
\item $\Pi(s_i,b_k)=c^{34}_{ik}$, and 
\item $\Pi(s_{k-1}s_ks_{k+1},b_k)$.
\end{itemize}
Then it follows in general since the general definitions are built as concatenations of these special cases.
\end{proof}

\begin{corollary}
For each of our generators for $H_3(X)$ in Lemma~\ref{le:Xhomology}, the homology class that the generator represents in $H_3(\sg{n})$ does not depend on the choice of subscripts.
In particular, $H_3(\sg{n})$ is generated by
\[\bar c^{31}_1, \bar c^{32}_{13}+\bar c^{32}_{31}, \bar c^{33}_{135}, 2\bar c^{34}_{14}-\bar c^{32}_{24}+\bar c^{32}_{14}, 2\bar c^{34}_{31}-\bar c^{32}_{14}+\bar c^{32}_{13}, \bar c^{36}_1, \text{ and } \bar c^{37}_1+\bar c^{32}_{13}.\]
\end{corollary}

\begin{proof}
We explain the argument for $\bar c^{33}_{ijk}$; the argument for the other generators is similar.
First let $\alpha=(1,2,\dotsc,i)$.
Then 
\[\partial \Pi(\alpha,c^{33}_{ijk})= \alpha\cdot c^{33}_{i-1,j,k}-c^{33}_{ijk}-\Pi(\alpha,\partial c^{33}_{ijk}).\]
Since $\bar c^{33}_{ijk}$ is a cycle in $(P_*)_G$, the cells in $\bar c^{33}_{ijk}$ can be split into pairs of the same cell appearing twice with opposite signs, but possibly with one of the pair being translated.
These pairs cancel when taking coinvariants.
Since $\Pi$ is equivariant in its second input, the prisms in $\Pi(\alpha,\partial c^{33}_{ijk})$" also pair off and cancel when taking coinvariants.
So we deduce that $\bar c^{33}_{ijk}$ and $c^{33}_{i-1,j,k}$ represent the same homology class.
Repeating this, we get that $\bar c^{33}_{ijk}\sim \bar c^{33}_{1jk}$.
Next we repeat this with $\alpha=(3,4,\dotsc,j)$, and deduce that $\bar c^{33}_{1jk}\sim \bar c^{33}_{13k}$.
Finally, we repeat this with $\alpha=(5,6,\dotsc, k)$, and deduce that $\bar c^{33}_{13k}\sim \bar c^{33}_{135}$.
So $\bar c^{33}_{ijk}\sim \bar c^{33}_{135}$.
\end{proof}

\begin{proposition}
The cycles $2\bar c^{34}_{14}-\bar c^{32}_{24}+\bar c^{32}_{14}$ and $2\bar c^{34}_{31}-\bar c^{32}_{14}+\bar c^{32}_{13}$ are null-homologous in $H_3(\sg{n})$.
\end{proposition}

\begin{proof}
Notice that
\[(s_2+1) c^{34}_{14}- c^{32}_{24}+ s_1s_2 c^{32}_{14}=\Pi(s_4,\partial c^{35}_1)\]
and
\[(s_4+1) c^{34}_{31}- c^{32}_{14}+ s_3s_4 c^{32}_{13}=\Pi(s_1,\partial c^{35}_3).\]
In $\widetilde X$, we have a cycle $s_4c^{35}_1-c^{35}_1-\Pi(s_4,\partial c^{35}_1)$  
(this would be $\partial \Pi(s_4, c^{35}_1)$ if $\widetilde X$ had $4$-cells).
So the image of this cycle is trivial in $H_3(\sg{n})$.
But since $s_4c^{35}_1$ and $c^{35}_1$ represent the same thing in $H_3(\sg{n})$, we have $[\Pi(s_4,\partial c^{35}_1)]=0$ in $H_3(\sg{n})$.
By the same argument, $[\Pi(s_1,\partial c^{35}_3)]=0$ in $H_3(\sg{n})$.
\end{proof}

\begin{proposition}
In $H_3(\sg{n})$, we have
\[2[\bar c^{37}_1+\bar c^{32}_{13}]=[ \bar c^{32}_{13}+\bar c^{32}_{31}].\]
\end{proposition}

\begin{proof}
First, we verify that the following is a cycle in $\widetilde X$:
\[\partial( (1+s_3)c^{37}_1-c^{32}_{31}+s_2s_1s_3s_2c^{32}_{13} + (1-s_1)c^{35}_2 +(s_2s_3-s_1s_2s_3)c^{35}_1)=0.\]
Since cycles in $\widetilde X$ produce trivial elements of $H_3(\sg{n})$, this implies that
\[[2\bar c^{37}_1-\bar c^{32}_{31}+\bar c^{32}_{13}]=0\]
in $H_3(\sg{n})$.
Rearranging this, we get the statement.
\end{proof}

\begin{proposition}
In $H_3(\sg{n})$, the elements 
\[[\bar c^{31}_1], [\bar c^{33}_{135}], \text{ and }[ \bar c^{32}_{13}+\bar c^{32}_{31}]\]
have order dividing $2$, and the element $[\bar c^{36}_1]$ has order dividing $3$.
\end{proposition}

\begin{proof}
First we notice the following cycle in $\widetilde X$:
\[\partial (s_1+1)c^{31}_1=0.\]
In $H_3(\sg{n})$, this becomes $2[\bar c^{31}_1]=0$.
Next we notice:
\[\partial ((s_1+1)c^{33}_{135}+(s_5-1)c^{32}_{13}+(1-s_3)c^{32}_{15})=0.\]
This becomes $2[\bar c^{33}_{135}]=0$.
Similarly,
\[\partial ((1+s_1s_3)(c^{32}_{13}+c^{32}_{31})+(s_3-1)c^{31}_1+(s_1-1)c^{31}_3)=0,\]
which becomes $2[\bar c^{32}_{13}+\bar c^{32}_{31}]=0$.
Finally, we check that
\[\partial ( (1+s_1s_2+s_2s_1)c^{36}_1 +(-1+s_1s_2s_1 +s_2)c^{31}_1+(1-s_1-s_1s_2s_1)c^{31}_2)=0.\]
Since $[c^{31}_1]=[c^{31}_2]$ in $H_3(\sg{n})$, this becomes $3[c^{36}_1]=0$ in $H_3(\sg{n})$.
\end{proof}

We summarize what we have just proven in the following corollary.
\begin{corollary}\label{co:H3upperbound}
The group $H_3(\sg{n})$ is a quotient of $(\Z/2\Z)^2\oplus (\Z/3\Z)\oplus (\Z/4\Z)$.
It is generated by $[\bar c^{31}_1]$ and $[\bar c^{33}_{135}]$, which have order dividing $2$, $[\bar c^{36}_1]$, which has order dividing $3$, and  $[\bar c^{37}_1+\bar c^{32}_{13}]$, which has order dividing $4$.
\end{corollary}

Next we want to show that these classes are actually nontrivial.
We can quote results to do most of the work.
We derive the following statement from the exposition in Adem--Milgram~\cite{AMCFG}; it is based on results of Nakaoka~\cite{Nakaoka} and Mann~\cite{Mann}.
\begin{theorem}\label{th:cyclesmodp}
Let $n\geq 6$.
Then $H_3(\sg{n};\mathbb{F}_2)=\mathbb{F}_2^4$, and $\dim(H_3(\sg{n};\mathbb{F}_3))\geq 1$.
\end{theorem}

\begin{remark}
It also follows from Adem--Milgram that $H_3(\sg{n};\mathbb{F}_3)=\mathbb{F}_3$ exactly, and $H_3(\sg{n};\mathbb{F}_p)=0$ for $p\geq 5$, but we do not need this since we are only using the result to show that certain cycles represent nontrivial classes.
\end{remark}

\begin{proof}[Proof of Theorem~\ref{th:cyclesmodp}]
Let $p$ be a prime.
The approach in Adem--Milgram is to show that the cohomology of $\sg{n}$ in $\mathbb{F}_p$ is detected by the abelian subgroups $V$ of $\sg{n}$.
The cohomology of an abelian group in $\mathbb{F}_p$ consists of a polynomial algebra if $p=2$, and the tensor product of a polynomial algebra and an exterior algebra if $p\geq 3$.
The image of the restriction map $H^*(\sg{n};\mathbb{F}_p)\to H^*(V;\mathbb{F}_p)$ consists of highly symmetric elements; for the polynomial subalgebra this is the classic Dickson algebra.
The point is that Adem and Milgram give a procedure for finding the degrees of generators for $H^*(\sg{n};\mathbb{F}_p)$.
These then give generators in the same degree for $H_*(\sg{\infty};\mathbb{F}_p)$, which has the structure of an algebra.
There is a second degree that can be used to determine the image of $H_*(\sg{n};\mathbb{F}_p)$ in $H_*(\sg{\infty};\mathbb{F}_p)$.

For $p=2$, define an \emph{admissible sequence} $I=(i_1,\dotsc, i_m)$ to be a nondecreasing sequence of positive integers.
An admissible sequence $I$ of length $m$ determines a generator $x_I$ in $H_*(\sg{\infty};\mathbb{F}_2)$ of degree $d(I)=i_1+2i_2+4i_3+\dotsm+2^{m-1}i_m$; $H_*(\sg{\infty};\mathbb{F}_2)$ is the $\mathbb{F}_2$-polynomial algebra on all such generators.
The generator $x_I$ has bidegree $b(x_I)=2^m$; the image of $H_*(\sg{n};\mathbb{F}_2)$  in $H_*(\sg{\infty};\mathbb{F}_2)$ is generated by monomials with bidegree less than or equal to $n$.
(This is part of the discussion on page~176 of Adem--Milgram~\cite{AMCFG}.)
The admissible sequences of degree $3$ or less are $(1)$, $(2)$, $(1,1)$, and $(3)$.
From this, we see that a basis for $H_3(\sg{\infty};\mathbb{F}_2)$ is $x_1^3$, $x_1x_2$, $x_{1,1}$, and $x_3$.
Since all of these have bidegree less than or equal to $6$, we find that $H_3(\sg{n};\mathbb{F}_2)$ has the same basis (since $n\geq 6$).

For odd primes, the description of generators is a little more involved.
For an elementary abelian group $(\Z/p\Z)^i$ with $p$ odd, we have a description of the cohomology algebra as
\[H^*((\Z/p\Z)^i;\mathbb{F}_p)=\mathbb{F}_p[b_1,\dotsc,b_i]\otimes \Lambda(e_1,\dotsc, e_i).\]
Here the $b_j$ classes are $2$-dimensional, the $e_j$ classes are $1$-dimensional, and $\Lambda$ denotes the exterior algebra over $\mathbb{F}_p$ on these generators.
(See, for example, Corollary II.4.3 in Adem--Milgram~\cite{AMCFG}.)
For each positive integer $i$, Adem--Milgram define several polynomials in the $b$- and $e$-classes.
Most of these polynomials involve another parameter $j$, an integer with $0\leq j<i$.
The names of these polynomials are $d_{p^i-p^j}$, $L_i$, and $M_{j,i}$.
The statement, explained in the discussion on page~196 of Adem--Milgram~\cite{AMCFG}, is that 
\[H_*(\sg{\infty};\mathbb{F}_p)=\bigotimes_{i=1}^\infty \Lambda(\mathcal{N}_i),\]
where $\mathcal{N}_i$ is the dual group to an ideal $\mathcal{I}_i$ in a subalgebra of $H^*((\Z/p\Z)^i;\mathbb{F}_p)$.
Here 
\[\mathcal{I}_i=\mathbb{F}_p[d_{p^i-p^{i-1}},\dotsc,d_{p^i-1}](d_{p^i-1},M_{0,i}L_i^{p-2},\dotsc,M_{j,i}L_i^{p-2},\dotsc,M_{0,i}\dotsm M_{i-1,i}L_i^\gamma).\]
(Here $\gamma$ is the minimal power that makes this polynomial $\mathbb{F}_p^*$-invariant, and it depends on $i$ and $p$.)
We don't need to understand this in general, but only for $i=1$.
If $i=1$, then $L_1=b_1$ and $M_{0,1}=e_1$ (this is from the definitions on page~187 of Adem--Milgram~\cite{AMCFG}).
The polynomial $M_{0,1}L_1^{p-2}$ is in $\mathcal{I}_1$, and so it has a dual element in $\mathcal{N}_1$ and determines a class in $H_*(\sg{\infty};\mathbb{F}_p)$.
We have $\dim(M_{0,1}L_1^{p-2})=2p-3$, so that this class is in $H_{2p-3}(\sg{\infty};\mathbb{F}_p)$.
The bidegree of this class is $p$, so this class is also in $H_{2p-3}(\sg{n};\mathbb{F}_p)$ for $n\geq p$.
In particular, if $p=3$, we see that $H_3(\sg{n};\mathbb{F}_3)$ is nontrivial for $n\geq 3$.
\end{proof}

\begin{corollary}\label{co:gensnontriv}
For $n\geq 6$,
 the classes $[\bar c^{31}_1]$, $[\bar c^{33}_{135}]$, $[\bar c^{36}_1]$, and  $[\bar c^{37}_1+\bar c^{32}_{13}]$ are nontrivial in $H_3(\sg{n})$.
\end{corollary}
\begin{proof}
We derive this from Theorem~\ref{th:cyclesmodp} using the universal coefficients theorem.
This states:
\[0\to H_3(\sg{n})\otimes_\Z \mathbb{F}_p \to H_3(\sg{n}; \mathbb{F}_p)\to \mathrm{Tor}^\Z_1(H_2(\sg{n}),\mathbb{F}_p)\to 0.\]
We know $H_2(\sg{n})=\Z/2\Z$, so for $p=2$, the $\mathrm{Tor}$-term is $\mathbb{F}_2$.
This implies that $H_3(\sg{n})\otimes_\Z\mathbb{F}_2$ is $\mathbb{F}_2^3$.
If any of $[\bar c^{31}_1]$, $[\bar c^{33}_{135}]$, or $[\bar c^{37}_1+\bar c^{32}_{13}]$ were trivial, this would imply that the dimension of this module is less than three, a contradiction.

For $p=3$, the $\mathrm{Tor}$-term is trivial, which implies that $H_3(\sg{n})\otimes_\Z\mathbb{F}_3$ is nontrivial.
This implies that $[\bar c^{36}_1]$ is nontrivial.
\end{proof}

This leaves unresolved the question of whether $[\bar c^{37}_1+\bar c^{32}_{13}]$ has order $2$ or order $4$.
We answer this by showing that $H_3(\sg{n})$ must have an element of order $4$, using a transfer map and an embedded copy of the dihedral group $D_8$ in $\sg{n}$.
We recall that $D_8$ is the group of order $8$:
\[\langle r,s\,|\,r^4=s^2=(rs)^2=1\rangle.\]
If $n\geq 4$, then this embeds in $\sg{n}$ by $r\mapsto (1234)$ and $s\mapsto (24)$.
It is a consequence of a theorem of Handel~\cite{Handel}, that $H_3(D_8)$ contains an element of order $4$.
We will show this explicitly, by building a cycle and a cocycle that we will relate to $\sg{n}$.

As in section~\ref{section: resolutions}, we let $\bar F_*(G)$ denote the normalized bar resolution of a group $G$.
Our usual way for representing cycles for $H_3(D_8;\Z)$ is using $\bar F_*(D_8)\otimes_{D_8}\Z$ as a chain complex.
This is the same as the coinvariants $(\bar F_*(D_8))_{D_8}$.
Let $c\in \bar F_3(D_8)$ denote the $3$-chain
\[c=[r|r|r]+[r|r^2|r]+[r|r^3|r].\]
A quick computation, using that $r^4=1$ and that we are in the normalized chain complex, shows that
\[\partial c=(r-1)([r|r]+[r^2|r]+[r^3|r]).\]
So the image $\bar c\in (\bar F_*(D_8))_{D_8}$ is a cycle.
In fact, we will show that $[\bar c]\in H_3(D_8)$ has order $4$, and so does its image in $H_3(\sg{n})$.

Likewise, our usual way for representing cocycles for $H^3(D_8;\Z/4\Z)$ is using $\hom_{D_8}(\bar F_*(D_8),\Z/4\Z)$ as a cochain complex.
We define the following cochain $\chi\co \bar F_3(D_8)\to \Z/4\Z$ by
\[
\chi([r^{a_1}s^{b_1}|r^{a_2}s^{b_2}|r^{a_3}s^{b_3}])=
\left\{
\begin{array}{cl}
(-1)^{b_1+b_2}a_1 & \text{ if $a_2+(-1)^{b_2}a_3\notin\{0,1,2,3\},$} \\
0 & \text{otherwise,}
\end{array}
\right.
\]
and extend by $\Z G$-linearity.
Here we think of the $a_i$ as integers and demand that $a_i\in\{0,1,2,3\}$ for $i=1,2,3$.
Since this cochain is $0$ if any of $r^{a_i}s^{b_i}$ is trivial, it follows that this is a normalized cochain, and is well defined.
Intuitively, this cochain checks if there is a carry digit modulo $4$ in evaluating the power of $r$ in the normal form of $r^{a_2}s^{b_2}r^{a_3}s^{b_3}$, and outputs a signed multiple of $a_1$ if there is a carry digit.

\begin{lemma}
The cochain $\chi$ is a cocycle.
\end{lemma}
\begin{proof}
We consider evaluating $\chi$ on a generic boundary.
Let $r^{a_i}s^{b_i}\in D_8$ be in normal form for $i=1,2,3,4$, meaning that $0\leq a_i<4$ and $0\leq b_i<2$.
Define the following simplices:
\[ \sigma_0=r^{a_1}s^{b_1}[r^{a_2}s^{b_2}|r^{a_3}s^{b_3}|r^{a_4}s^{b_4}],\quad
\sigma_1=[r^{a_1+(-1)^{b_1}a_2}s^{b_1+b_2}|r^{a_3}s^{b_3}|r^{a_4}s^{b_4}],\]
\[\sigma_2=[r^{a_1}s^{b_1}|r^{a_2+(-1)^{b_2}a_3}s^{b_2+b_3}|r^{a_4}s^{b_4}],\quad
\sigma_3=[r^{a_1}s^{b_1}|r^{a_2}s^{b_2}|r^{a_3+(-1)^{b_3}a_4}s^{b_3+b_4}],\]
\[\sigma_4=[r^{a_1}s^{b_1}|r^{a_2}s^{b_2}|r^{a_3}s^{b_3}].\]

Then $\partial[r^{a_1}s^{b_1}|r^{a_2}s^{b_2}|r^{a_3}s^{b_3}|r^{a_4}s^{b_4}]=\sigma_0-\sigma_1+\sigma_2-\sigma_3+\sigma_4$.
Consider evaluating $\chi$ on this boundary.
Please note, it is important to reduce the exponents on $r$ modulo $4$ before plugging into the definition of $\chi$.
Say that a particular simplex \emph{triggers} $\chi$ if it produces a carry digit in the product of its second and third terms; in other words, if we use the first piece in the piecewise definition of $\chi$.
Note that the $\sigma_0$ triggers $\chi$ if and only if the $\sigma_1$ triggers $\chi$, since their second and third entries are the same.
This means that these two simplices contribute cancelling copies of $a_2$:
if they both trigger, they both contribute $\pm (-1)^{b_2+b_3}a_2$, and if neither one triggers, they both contribute $0a_2$.
So there is no $a_2$-term in $\chi(\sigma_0-\sigma_1)$; since the remaining simplices cannot output a multiple of $a_2$ under $\chi$, we deduce that there is no $a_2$-term in the output of this boundary under $\chi$.

For counting the $a_1$-term, notice that any of $\sigma_1$ through $\sigma_4$, if they trigger, will contribute $\pm a_1$.
Verifying that the output is trivial, then, comes down to tabulating when the simplices trigger $\chi$, in cases depending on the $b_i$.
Since there are $4$ choices for each of $a_2$, $a_3$, and $a_4$, and $2$ choices for each of $b_2$ and $b_3$, there are $256$ cases to check.
These checks appear in Table~\ref{ta:cocyclechecks}.
This data was produced using the Wolfram Mathematica command:
\begin{verbatim}
Table[
  {a2,b2,a3,b3,a4, 
  If[Quotient[a3+((-1)^b3)a4,4]!=0,(-1)^(1+b2+b3),0], 
  If[Quotient[Mod[a2+((-1)^b2)a3,4]+((-1)^(b2+b3))a4,4]!=0,
    (-1)^(b2 + b3),0], 
  If[Quotient[a2+((-1)^b2)Mod[a3+((-1)^b3)a4,4],4]!=0,
    (-1)^(1+b2),0], 
  If[Quotient[a2+((-1)^b2)a3,4]!=0,(-1)^b2,0]}, 
{a2,0,3},{b2,0,1},{a3,0,3},{b3,0,1},{a4,0,3}]
\end{verbatim}
\end{proof}

\begin{table}
\begingroup
\setlength{\tabcolsep}{0.5pt} % Default value: 6pt
\tiny
\begin{tabular}{|ccccc|cccc|}
\hline
$a_2$ & $b_2$ & $a_3$ & $b_3$ & $a_4$ &  &  &  &  \\
\hline
 0 & 0 & 0 & 0 & 0 & 0 & 0 & 0 & 0 \\
 0 & 0 & 0 & 0 & 1 & 0 & 0 & 0 & 0 \\
 0 & 0 & 0 & 0 & 2 & 0 & 0 & 0 & 0 \\
 0 & 0 & 0 & 0 & 3 & 0 & 0 & 0 & 0 \\
 0 & 0 & 0 & 1 & 0 & 0 & 0 & 0 & 0 \\
 0 & 0 & 0 & 1 & 1 & 1 & -1 & 0 & 0 \\
 0 & 0 & 0 & 1 & 2 & 1 & -1 & 0 & 0 \\
 0 & 0 & 0 & 1 & 3 & 1 & -1 & 0 & 0 \\
 0 & 0 & 1 & 0 & 0 & 0 & 0 & 0 & 0 \\
 0 & 0 & 1 & 0 & 1 & 0 & 0 & 0 & 0 \\
 0 & 0 & 1 & 0 & 2 & 0 & 0 & 0 & 0 \\
 0 & 0 & 1 & 0 & 3 & -1 & 1 & 0 & 0 \\
 0 & 0 & 1 & 1 & 0 & 0 & 0 & 0 & 0 \\
 0 & 0 & 1 & 1 & 1 & 0 & 0 & 0 & 0 \\
 0 & 0 & 1 & 1 & 2 & 1 & -1 & 0 & 0 \\
 0 & 0 & 1 & 1 & 3 & 1 & -1 & 0 & 0 \\
 0 & 0 & 2 & 0 & 0 & 0 & 0 & 0 & 0 \\
 0 & 0 & 2 & 0 & 1 & 0 & 0 & 0 & 0 \\
 0 & 0 & 2 & 0 & 2 & -1 & 1 & 0 & 0 \\
 0 & 0 & 2 & 0 & 3 & -1 & 1 & 0 & 0 \\
 0 & 0 & 2 & 1 & 0 & 0 & 0 & 0 & 0 \\
 0 & 0 & 2 & 1 & 1 & 0 & 0 & 0 & 0 \\
 0 & 0 & 2 & 1 & 2 & 0 & 0 & 0 & 0 \\
 0 & 0 & 2 & 1 & 3 & 1 & -1 & 0 & 0 \\
 0 & 0 & 3 & 0 & 0 & 0 & 0 & 0 & 0 \\
 0 & 0 & 3 & 0 & 1 & -1 & 1 & 0 & 0 \\
 0 & 0 & 3 & 0 & 2 & -1 & 1 & 0 & 0 \\
 0 & 0 & 3 & 0 & 3 & -1 & 1 & 0 & 0 \\
 0 & 0 & 3 & 1 & 0 & 0 & 0 & 0 & 0 \\
 0 & 0 & 3 & 1 & 1 & 0 & 0 & 0 & 0 \\
 0 & 0 & 3 & 1 & 2 & 0 & 0 & 0 & 0 \\
 0 & 0 & 3 & 1 & 3 & 0 & 0 & 0 & 0 \\
 0 & 1 & 0 & 0 & 0 & 0 & 0 & 0 & 0 \\
 0 & 1 & 0 & 0 & 1 & 0 & -1 & 1 & 0 \\
 0 & 1 & 0 & 0 & 2 & 0 & -1 & 1 & 0 \\
 0 & 1 & 0 & 0 & 3 & 0 & -1 & 1 & 0 \\
 0 & 1 & 0 & 1 & 0 & 0 & 0 & 0 & 0 \\
 0 & 1 & 0 & 1 & 1 & -1 & 0 & 1 & 0 \\
 0 & 1 & 0 & 1 & 2 & -1 & 0 & 1 & 0 \\
 0 & 1 & 0 & 1 & 3 & -1 & 0 & 1 & 0 \\
 0 & 1 & 1 & 0 & 0 & 0 & 0 & 1 & -1 \\
 0 & 1 & 1 & 0 & 1 & 0 & 0 & 1 & -1 \\
 0 & 1 & 1 & 0 & 2 & 0 & 0 & 1 & -1 \\
 0 & 1 & 1 & 0 & 3 & 1 & 0 & 0 & -1 \\
 0 & 1 & 1 & 1 & 0 & 0 & 0 & 1 & -1 \\
 0 & 1 & 1 & 1 & 1 & 0 & 1 & 0 & -1 \\
 0 & 1 & 1 & 1 & 2 & -1 & 1 & 1 & -1 \\
 0 & 1 & 1 & 1 & 3 & -1 & 1 & 1 & -1 \\
 0 & 1 & 2 & 0 & 0 & 0 & 0 & 1 & -1 \\
 0 & 1 & 2 & 0 & 1 & 0 & 0 & 1 & -1 \\
 0 & 1 & 2 & 0 & 2 & 1 & 0 & 0 & -1 \\
 0 & 1 & 2 & 0 & 3 & 1 & -1 & 1 & -1 \\
 0 & 1 & 2 & 1 & 0 & 0 & 0 & 1 & -1 \\
 0 & 1 & 2 & 1 & 1 & 0 & 0 & 1 & -1 \\
 0 & 1 & 2 & 1 & 2 & 0 & 1 & 0 & -1 \\
 0 & 1 & 2 & 1 & 3 & -1 & 1 & 1 & -1 \\
 0 & 1 & 3 & 0 & 0 & 0 & 0 & 1 & -1 \\
 0 & 1 & 3 & 0 & 1 & 1 & 0 & 0 & -1 \\
 0 & 1 & 3 & 0 & 2 & 1 & -1 & 1 & -1 \\
 0 & 1 & 3 & 0 & 3 & 1 & -1 & 1 & -1 \\
 0 & 1 & 3 & 1 & 0 & 0 & 0 & 1 & -1 \\
 0 & 1 & 3 & 1 & 1 & 0 & 0 & 1 & -1 \\
 0 & 1 & 3 & 1 & 2 & 0 & 0 & 1 & -1 \\
 0 & 1 & 3 & 1 & 3 & 0 & 1 & 0 & -1 \\
\hline
\end{tabular}
\begin{tabular}{|ccccc|cccc|}
\hline
$a_2$ & $b_2$ & $a_3$ & $b_3$ & $a_4$ &  &  &  &  \\
\hline
 1 & 0 & 0 & 0 & 0 & 0 & 0 & 0 & 0 \\
 1 & 0 & 0 & 0 & 1 & 0 & 0 & 0 & 0 \\
 1 & 0 & 0 & 0 & 2 & 0 & 0 & 0 & 0 \\
 1 & 0 & 0 & 0 & 3 & 0 & 1 & -1 & 0 \\
 1 & 0 & 0 & 1 & 0 & 0 & 0 & 0 & 0 \\
 1 & 0 & 0 & 1 & 1 & 1 & 0 & -1 & 0 \\
 1 & 0 & 0 & 1 & 2 & 1 & -1 & 0 & 0 \\
 1 & 0 & 0 & 1 & 3 & 1 & -1 & 0 & 0 \\
 1 & 0 & 1 & 0 & 0 & 0 & 0 & 0 & 0 \\
 1 & 0 & 1 & 0 & 1 & 0 & 0 & 0 & 0 \\
 1 & 0 & 1 & 0 & 2 & 0 & 1 & -1 & 0 \\
 1 & 0 & 1 & 0 & 3 & -1 & 1 & 0 & 0 \\
 1 & 0 & 1 & 1 & 0 & 0 & 0 & 0 & 0 \\
 1 & 0 & 1 & 1 & 1 & 0 & 0 & 0 & 0 \\
 1 & 0 & 1 & 1 & 2 & 1 & 0 & -1 & 0 \\
 1 & 0 & 1 & 1 & 3 & 1 & -1 & 0 & 0 \\
 1 & 0 & 2 & 0 & 0 & 0 & 0 & 0 & 0 \\
 1 & 0 & 2 & 0 & 1 & 0 & 1 & -1 & 0 \\
 1 & 0 & 2 & 0 & 2 & -1 & 1 & 0 & 0 \\
 1 & 0 & 2 & 0 & 3 & -1 & 1 & 0 & 0 \\
 1 & 0 & 2 & 1 & 0 & 0 & 0 & 0 & 0 \\
 1 & 0 & 2 & 1 & 1 & 0 & 0 & 0 & 0 \\
 1 & 0 & 2 & 1 & 2 & 0 & 0 & 0 & 0 \\
 1 & 0 & 2 & 1 & 3 & 1 & 0 & -1 & 0 \\
 1 & 0 & 3 & 0 & 0 & 0 & 0 & -1 & 1 \\
 1 & 0 & 3 & 0 & 1 & -1 & 0 & 0 & 1 \\
 1 & 0 & 3 & 0 & 2 & -1 & 0 & 0 & 1 \\
 1 & 0 & 3 & 0 & 3 & -1 & 0 & 0 & 1 \\
 1 & 0 & 3 & 1 & 0 & 0 & 0 & -1 & 1 \\
 1 & 0 & 3 & 1 & 1 & 0 & -1 & 0 & 1 \\
 1 & 0 & 3 & 1 & 2 & 0 & -1 & 0 & 1 \\
 1 & 0 & 3 & 1 & 3 & 0 & -1 & 0 & 1 \\
 1 & 1 & 0 & 0 & 0 & 0 & 0 & 0 & 0 \\
 1 & 1 & 0 & 0 & 1 & 0 & 0 & 0 & 0 \\
 1 & 1 & 0 & 0 & 2 & 0 & -1 & 1 & 0 \\
 1 & 1 & 0 & 0 & 3 & 0 & -1 & 1 & 0 \\
 1 & 1 & 0 & 1 & 0 & 0 & 0 & 0 & 0 \\
 1 & 1 & 0 & 1 & 1 & -1 & 0 & 1 & 0 \\
 1 & 1 & 0 & 1 & 2 & -1 & 0 & 1 & 0 \\
 1 & 1 & 0 & 1 & 3 & -1 & 1 & 0 & 0 \\
 1 & 1 & 1 & 0 & 0 & 0 & 0 & 0 & 0 \\
 1 & 1 & 1 & 0 & 1 & 0 & -1 & 1 & 0 \\
 1 & 1 & 1 & 0 & 2 & 0 & -1 & 1 & 0 \\
 1 & 1 & 1 & 0 & 3 & 1 & -1 & 0 & 0 \\
 1 & 1 & 1 & 1 & 0 & 0 & 0 & 0 & 0 \\
 1 & 1 & 1 & 1 & 1 & 0 & 0 & 0 & 0 \\
 1 & 1 & 1 & 1 & 2 & -1 & 0 & 1 & 0 \\
 1 & 1 & 1 & 1 & 3 & -1 & 0 & 1 & 0 \\
 1 & 1 & 2 & 0 & 0 & 0 & 0 & 1 & -1 \\
 1 & 1 & 2 & 0 & 1 & 0 & 0 & 1 & -1 \\
 1 & 1 & 2 & 0 & 2 & 1 & 0 & 0 & -1 \\
 1 & 1 & 2 & 0 & 3 & 1 & 0 & 0 & -1 \\
 1 & 1 & 2 & 1 & 0 & 0 & 0 & 1 & -1 \\
 1 & 1 & 2 & 1 & 1 & 0 & 1 & 0 & -1 \\
 1 & 1 & 2 & 1 & 2 & 0 & 1 & 0 & -1 \\
 1 & 1 & 2 & 1 & 3 & -1 & 1 & 1 & -1 \\
 1 & 1 & 3 & 0 & 0 & 0 & 0 & 1 & -1 \\
 1 & 1 & 3 & 0 & 1 & 1 & 0 & 0 & -1 \\
 1 & 1 & 3 & 0 & 2 & 1 & 0 & 0 & -1 \\
 1 & 1 & 3 & 0 & 3 & 1 & -1 & 1 & -1 \\
 1 & 1 & 3 & 1 & 0 & 0 & 0 & 1 & -1 \\
 1 & 1 & 3 & 1 & 1 & 0 & 0 & 1 & -1 \\
 1 & 1 & 3 & 1 & 2 & 0 & 1 & 0 & -1 \\
 1 & 1 & 3 & 1 & 3 & 0 & 1 & 0 & -1 \\
\hline
\end{tabular}
\begin{tabular}{|ccccc|cccc|}
\hline
$a_2$ & $b_2$ & $a_3$ & $b_3$ & $a_4$ &  &  &  &  \\
\hline
2 & 0 & 0 & 0 & 0 & 0 & 0 & 0 & 0 \\
 2 & 0 & 0 & 0 & 1 & 0 & 0 & 0 & 0 \\
 2 & 0 & 0 & 0 & 2 & 0 & 1 & -1 & 0 \\
 2 & 0 & 0 & 0 & 3 & 0 & 1 & -1 & 0 \\
 2 & 0 & 0 & 1 & 0 & 0 & 0 & 0 & 0 \\
 2 & 0 & 0 & 1 & 1 & 1 & 0 & -1 & 0 \\
 2 & 0 & 0 & 1 & 2 & 1 & 0 & -1 & 0 \\
 2 & 0 & 0 & 1 & 3 & 1 & -1 & 0 & 0 \\
 2 & 0 & 1 & 0 & 0 & 0 & 0 & 0 & 0 \\
 2 & 0 & 1 & 0 & 1 & 0 & 1 & -1 & 0 \\
 2 & 0 & 1 & 0 & 2 & 0 & 1 & -1 & 0 \\
 2 & 0 & 1 & 0 & 3 & -1 & 1 & 0 & 0 \\
 2 & 0 & 1 & 1 & 0 & 0 & 0 & 0 & 0 \\
 2 & 0 & 1 & 1 & 1 & 0 & 0 & 0 & 0 \\
 2 & 0 & 1 & 1 & 2 & 1 & 0 & -1 & 0 \\
 2 & 0 & 1 & 1 & 3 & 1 & 0 & -1 & 0 \\
 2 & 0 & 2 & 0 & 0 & 0 & 0 & -1 & 1 \\
 2 & 0 & 2 & 0 & 1 & 0 & 0 & -1 & 1 \\
 2 & 0 & 2 & 0 & 2 & -1 & 0 & 0 & 1 \\
 2 & 0 & 2 & 0 & 3 & -1 & 0 & 0 & 1 \\
 2 & 0 & 2 & 1 & 0 & 0 & 0 & -1 & 1 \\
 2 & 0 & 2 & 1 & 1 & 0 & -1 & 0 & 1 \\
 2 & 0 & 2 & 1 & 2 & 0 & -1 & 0 & 1 \\
 2 & 0 & 2 & 1 & 3 & 1 & -1 & -1 & 1 \\
 2 & 0 & 3 & 0 & 0 & 0 & 0 & -1 & 1 \\
 2 & 0 & 3 & 0 & 1 & -1 & 0 & 0 & 1 \\
 2 & 0 & 3 & 0 & 2 & -1 & 0 & 0 & 1 \\
 2 & 0 & 3 & 0 & 3 & -1 & 1 & -1 & 1 \\
 2 & 0 & 3 & 1 & 0 & 0 & 0 & -1 & 1 \\
 2 & 0 & 3 & 1 & 1 & 0 & 0 & -1 & 1 \\
 2 & 0 & 3 & 1 & 2 & 0 & -1 & 0 & 1 \\
 2 & 0 & 3 & 1 & 3 & 0 & -1 & 0 & 1 \\
 2 & 1 & 0 & 0 & 0 & 0 & 0 & 0 & 0 \\
 2 & 1 & 0 & 0 & 1 & 0 & 0 & 0 & 0 \\
 2 & 1 & 0 & 0 & 2 & 0 & 0 & 0 & 0 \\
 2 & 1 & 0 & 0 & 3 & 0 & -1 & 1 & 0 \\
 2 & 1 & 0 & 1 & 0 & 0 & 0 & 0 & 0 \\
 2 & 1 & 0 & 1 & 1 & -1 & 0 & 1 & 0 \\
 2 & 1 & 0 & 1 & 2 & -1 & 1 & 0 & 0 \\
 2 & 1 & 0 & 1 & 3 & -1 & 1 & 0 & 0 \\
 2 & 1 & 1 & 0 & 0 & 0 & 0 & 0 & 0 \\
 2 & 1 & 1 & 0 & 1 & 0 & 0 & 0 & 0 \\
 2 & 1 & 1 & 0 & 2 & 0 & -1 & 1 & 0 \\
 2 & 1 & 1 & 0 & 3 & 1 & -1 & 0 & 0 \\
 2 & 1 & 1 & 1 & 0 & 0 & 0 & 0 & 0 \\
 2 & 1 & 1 & 1 & 1 & 0 & 0 & 0 & 0 \\
 2 & 1 & 1 & 1 & 2 & -1 & 0 & 1 & 0 \\
 2 & 1 & 1 & 1 & 3 & -1 & 1 & 0 & 0 \\
 2 & 1 & 2 & 0 & 0 & 0 & 0 & 0 & 0 \\
 2 & 1 & 2 & 0 & 1 & 0 & -1 & 1 & 0 \\
 2 & 1 & 2 & 0 & 2 & 1 & -1 & 0 & 0 \\
 2 & 1 & 2 & 0 & 3 & 1 & -1 & 0 & 0 \\
 2 & 1 & 2 & 1 & 0 & 0 & 0 & 0 & 0 \\
 2 & 1 & 2 & 1 & 1 & 0 & 0 & 0 & 0 \\
 2 & 1 & 2 & 1 & 2 & 0 & 0 & 0 & 0 \\
 2 & 1 & 2 & 1 & 3 & -1 & 0 & 1 & 0 \\
 2 & 1 & 3 & 0 & 0 & 0 & 0 & 1 & -1 \\
 2 & 1 & 3 & 0 & 1 & 1 & 0 & 0 & -1 \\
 2 & 1 & 3 & 0 & 2 & 1 & 0 & 0 & -1 \\
 2 & 1 & 3 & 0 & 3 & 1 & 0 & 0 & -1 \\
 2 & 1 & 3 & 1 & 0 & 0 & 0 & 1 & -1 \\
 2 & 1 & 3 & 1 & 1 & 0 & 1 & 0 & -1 \\
 2 & 1 & 3 & 1 & 2 & 0 & 1 & 0 & -1 \\
 2 & 1 & 3 & 1 & 3 & 0 & 1 & 0 & -1 \\
\hline
\end{tabular}
\begin{tabular}{|ccccc|cccc|}
\hline
$a_2$ & $b_2$ & $a_3$ & $b_3$ & $a_4$ &  &  &  &  \\
\hline
 3 & 0 & 0 & 0 & 0 & 0 & 0 & 0 & 0 \\
 3 & 0 & 0 & 0 & 1 & 0 & 1 & -1 & 0 \\
 3 & 0 & 0 & 0 & 2 & 0 & 1 & -1 & 0 \\
 3 & 0 & 0 & 0 & 3 & 0 & 1 & -1 & 0 \\
 3 & 0 & 0 & 1 & 0 & 0 & 0 & 0 & 0 \\
 3 & 0 & 0 & 1 & 1 & 1 & 0 & -1 & 0 \\
 3 & 0 & 0 & 1 & 2 & 1 & 0 & -1 & 0 \\
 3 & 0 & 0 & 1 & 3 & 1 & 0 & -1 & 0 \\
 3 & 0 & 1 & 0 & 0 & 0 & 0 & -1 & 1 \\
 3 & 0 & 1 & 0 & 1 & 0 & 0 & -1 & 1 \\
 3 & 0 & 1 & 0 & 2 & 0 & 0 & -1 & 1 \\
 3 & 0 & 1 & 0 & 3 & -1 & 0 & 0 & 1 \\
 3 & 0 & 1 & 1 & 0 & 0 & 0 & -1 & 1 \\
 3 & 0 & 1 & 1 & 1 & 0 & -1 & 0 & 1 \\
 3 & 0 & 1 & 1 & 2 & 1 & -1 & -1 & 1 \\
 3 & 0 & 1 & 1 & 3 & 1 & -1 & -1 & 1 \\
 3 & 0 & 2 & 0 & 0 & 0 & 0 & -1 & 1 \\
 3 & 0 & 2 & 0 & 1 & 0 & 0 & -1 & 1 \\
 3 & 0 & 2 & 0 & 2 & -1 & 0 & 0 & 1 \\
 3 & 0 & 2 & 0 & 3 & -1 & 1 & -1 & 1 \\
 3 & 0 & 2 & 1 & 0 & 0 & 0 & -1 & 1 \\
 3 & 0 & 2 & 1 & 1 & 0 & 0 & -1 & 1 \\
 3 & 0 & 2 & 1 & 2 & 0 & -1 & 0 & 1 \\
 3 & 0 & 2 & 1 & 3 & 1 & -1 & -1 & 1 \\
 3 & 0 & 3 & 0 & 0 & 0 & 0 & -1 & 1 \\
 3 & 0 & 3 & 0 & 1 & -1 & 0 & 0 & 1 \\
 3 & 0 & 3 & 0 & 2 & -1 & 1 & -1 & 1 \\
 3 & 0 & 3 & 0 & 3 & -1 & 1 & -1 & 1 \\
 3 & 0 & 3 & 1 & 0 & 0 & 0 & -1 & 1 \\
 3 & 0 & 3 & 1 & 1 & 0 & 0 & -1 & 1 \\
 3 & 0 & 3 & 1 & 2 & 0 & 0 & -1 & 1 \\
 3 & 0 & 3 & 1 & 3 & 0 & -1 & 0 & 1 \\
 3 & 1 & 0 & 0 & 0 & 0 & 0 & 0 & 0 \\
 3 & 1 & 0 & 0 & 1 & 0 & 0 & 0 & 0 \\
 3 & 1 & 0 & 0 & 2 & 0 & 0 & 0 & 0 \\
 3 & 1 & 0 & 0 & 3 & 0 & 0 & 0 & 0 \\
 3 & 1 & 0 & 1 & 0 & 0 & 0 & 0 & 0 \\
 3 & 1 & 0 & 1 & 1 & -1 & 1 & 0 & 0 \\
 3 & 1 & 0 & 1 & 2 & -1 & 1 & 0 & 0 \\
 3 & 1 & 0 & 1 & 3 & -1 & 1 & 0 & 0 \\
 3 & 1 & 1 & 0 & 0 & 0 & 0 & 0 & 0 \\
 3 & 1 & 1 & 0 & 1 & 0 & 0 & 0 & 0 \\
 3 & 1 & 1 & 0 & 2 & 0 & 0 & 0 & 0 \\
 3 & 1 & 1 & 0 & 3 & 1 & -1 & 0 & 0 \\
 3 & 1 & 1 & 1 & 0 & 0 & 0 & 0 & 0 \\
 3 & 1 & 1 & 1 & 1 & 0 & 0 & 0 & 0 \\
 3 & 1 & 1 & 1 & 2 & -1 & 1 & 0 & 0 \\
 3 & 1 & 1 & 1 & 3 & -1 & 1 & 0 & 0 \\
 3 & 1 & 2 & 0 & 0 & 0 & 0 & 0 & 0 \\
 3 & 1 & 2 & 0 & 1 & 0 & 0 & 0 & 0 \\
 3 & 1 & 2 & 0 & 2 & 1 & -1 & 0 & 0 \\
 3 & 1 & 2 & 0 & 3 & 1 & -1 & 0 & 0 \\
 3 & 1 & 2 & 1 & 0 & 0 & 0 & 0 & 0 \\
 3 & 1 & 2 & 1 & 1 & 0 & 0 & 0 & 0 \\
 3 & 1 & 2 & 1 & 2 & 0 & 0 & 0 & 0 \\
 3 & 1 & 2 & 1 & 3 & -1 & 1 & 0 & 0 \\
 3 & 1 & 3 & 0 & 0 & 0 & 0 & 0 & 0 \\
 3 & 1 & 3 & 0 & 1 & 1 & -1 & 0 & 0 \\
 3 & 1 & 3 & 0 & 2 & 1 & -1 & 0 & 0 \\
 3 & 1 & 3 & 0 & 3 & 1 & -1 & 0 & 0 \\
 3 & 1 & 3 & 1 & 0 & 0 & 0 & 0 & 0 \\
 3 & 1 & 3 & 1 & 1 & 0 & 0 & 0 & 0 \\
 3 & 1 & 3 & 1 & 2 & 0 & 0 & 0 & 0 \\
 3 & 1 & 3 & 1 & 3 & 0 & 0 & 0 & 0 \\
\hline
\end{tabular}
\endgroup
\caption{The $256$ cases for the inputs to $\chi$ on a boundary.
The unlabeled output columns are $\frac{\chi(\sigma_0-\sigma_1)}{(-1)^{b_2}a_1}$, $\frac{\chi(\sigma_2)}{(-1)^{b_2}a_1}$, $\frac{-\chi(\sigma_3)}{(-1)^{b_2}a_1}$, and $\frac{\chi(\sigma_4)}{(-1)^{b_2}a_1}$.
The output of $\chi$ is $0$ on this boundary because all output rows sum to $0$.  }\label{ta:cocyclechecks}
\end{table}

\begin{remark}\label{re:evalcocycle}
From this, we already recover that $H_3(D_8)$ has an element of order $4$, since $[\chi]\in H^3(D_8,\Z/4\Z)$ determines an element of $\hom(H_3(D_8),\Z/4\Z)$ under the evaluation pairing, and this element sends $[\bar c]$ to $1$.
This is because $[r|r|r]$ and $[r|r^2|r]$ do not trigger $\chi$, but $\chi([r|r^3|r])=1$.
\end{remark}

\begin{proposition}\label{pr:transfertrick}
Let $n\geq4$, and let $H<\sg{n}$ be $D_8\times \sg{n-4}$, where $D_8$ embeds in the standard way acting on $\{1,2,3,4\}$, and $\sg{n-4}$ acts on $\{5,6,\dotsc,n\}$.
Let $\mathrm{tr}\co H_3(\sg{n})\to H_3(H)$ be the transfer map.
Let $\iota\co D_8\to \sg{n}$ be the standard inclusion, and let $\pi_1\co H\to D_8$ be the coordinate projection.
Then $[\chi]$ sends $\pi_{1*}\mathrm{tr}\iota_*[\bar c]$ to $\pm1$.
\end{proposition}

\begin{proof}
The tricky part is evaluating $\mathrm{tr}$ on $\iota_*[\bar c]$.
The complex $\bar F_*(\sg{n})$ is a free resolution of $\Z$ both as a $\Z \sg{n}$-module and as a $\Z H$-module.
We can easily express the transfer using this resolution for $H$, but to finish our computation, we need to move our chains over to $\bar F_*(H)$ using an augmentation-preserving map of $\Z H$-complexes $\bar F_*(\sg{n})\to \bar F_*(H)$.
Let $\rho\co \bar F_*(\sg{n})\to \bar F_*(H)$ be such a map.
Then for a cycle $[b]\in \bar F_*(\sg{n})$, the transfer of $[b]$ is
\[\mathrm{tr}([b])=[\rho(\sum_{g\in H\backslash \sg{n}} gb)].\]
Here $H\backslash \sg{n}$ denotes the set of right cosets of $H$ in $\sg{n}$, and the sum is over one representative from each coset.
It is standard that this is well defined; see for example Brown~\cite{Brown}, section~III.9.
To make this more explicit, we pick a specific map $\rho$.
Following Brown, we define $\rho\co \sg{n}\to H$ by picking a coset representative $\bar g$ for each coset $Hg$, and define $\rho(g)=g\bar g^{-1}$.
Then for a homogeneous simplex $\sigma=(g_0,g_1,\dotsc,g_m)\in \bar F_m(\sg{n})$, define $\rho(\sigma)=(\rho(g_0),\rho(g_1),\dotsc,\rho(g_m))$.
Since $H=D_8\times \sg{n-4}$ is such a nice subgroup of $\sg{n}$, it is possible to pick coset representatives explicitly; however, we do not need to do this.
It is enough to simply demand that, for $g\in H$, we have $\bar g=1$, so that $\rho$ restricts to the identity on $\bar F_*(H)\subseteq \bar F_*(\sg{n})$.

To proceed, we analyze the sum of the translates $gc$ as $g$ runs over the cosets of $H$.
Intuitively, the transfer of a cycle is the total preimage of that cycle under a covering map.
The cycle $\iota_*c$ is represented by a map from a lens space with fundamental group $\Z/4\Z$.
Each component of this preimage is a cover of that lens space.
We break the sum up into sum up into three parts, where the first part consists of components where the cover is a sphere, the second part consists of components where the cover is an $\R P^3$, and the third part consists of components where the cover is the lens space itself.

\begin{claim*}
For any $g\in \sg{n}$, we have $\chi(\pi_{1*}\rho(g\iota_*(1+r+r^2+r^3)c))=0$.
\end{claim*}

The point is that $(1+r+r^2+r^3)c$ is a boundary in $\bar F_*(D_8)$.
In fact, $(1+r+r^2+r^3)c=\partial b$, where
\[b=\sum_{i=1}^3\sum_{j=1}^3[r^i|r|r^2|r].\]
So
\[
\begin{split}
\chi(\pi_{1*}\rho(g\iota_*(1+r+r^2+r^3)c))&= \chi(\pi_{1*}\rho(g\iota_*\partial b)) = \chi(\pi_{1*}\rho(\partial g\iota_*b)) \\
&= \chi(\partial \pi_{1*}\rho(g \iota_*b))=0.
\end{split}
\]

\begin{claim*}
If $g\in\sg{n}$ with $g\iota(r^2)g^{-1}\in H$, then $2\chi(\pi_{1*}\rho(g\iota_*(1+r)c))=0$.
\end{claim*}
This is a consequence of the previous claim.
Let $a=g\iota(r^2)g^{-1}\in H$.
Since $\chi$ is a cocycle to a module with the trivial action, $\chi(\pi_1(a) x)=\chi(x)$ for any chain $x$.
Since $\rho$ is a chain map between $\Z H$-complexes, we know $\rho(a x)=a\rho(x)$ for any $x$.
Then
\[
\begin{split}
\chi(\pi_{1*}\rho(g\iota_*(r^2+r^3)c))&=\chi(\pi_{1*}\rho(g\iota(r^2)\iota_*(1+r)c))=\chi(\pi_{1*}\rho(ag\iota_*(1+r)c)) \\
&=\chi(\pi_1(a)\pi_{1*}\rho(g\iota_*(1+r)c))=\chi(\pi_{1*}\rho(g\iota_*(1+r)c)).
\end{split}
\]
So
\[
\begin{split}
2\chi(\pi_{1*}\rho(g\iota_*(1+r)c))&=\chi(\pi_{1*}\rho(g\iota_*(1+r)c))+\chi(\pi_{1*}\rho(g\iota_*(r^2+r^3)c))\\
&=\chi(\pi_{1*}\rho(g\iota_*(1+r+r^2+r^3)c))=0.
\end{split}
\]

\begin{claim*}
If $g\in \sg{n}$ with $g\iota(r)g^{-1}\in H$, then $\rho(g\iota_*c)=(\alpha_g)_*\iota_*c$, where $\alpha_g\co \sg{n}\to\sg{n}$ is $x\mapsto gxg^{-1}$.
\end{claim*}

Since $\rho$ is $H$-equivariant and $\chi$ is $H$-invariant, we may assume without loss of generality that $g$ is one of the coset representatives used in the definition of $\rho$.
Since $g\iota(r)g^{-1}\in H$, we have that, for any $k$, 
\[Hg\iota(r^k)=H(g\iota(r^k)g^{-1})g=Hg,\]
so that $g$ is also the coset representative for $g\iota(r^k)$.
Then for any $k$, we have
\[\rho(g\iota(r^k))=g\iota(r^k)\overline{g\iota(r^k)}^{-1}=g\iota(r^k)g^{-1}=\alpha_g(\iota(r^k)).\]
Then 
\[
\begin{split}
\rho(g\iota_*c)&=\rho(g([\iota(r)|\iota(r)|\iota(r)]+[\iota(r)|\iota(r^2)|\iota(r)]+[\iota(r)|\iota(r^3)|\iota(r)])) \\
&=\rho(g((1,\iota(r),\iota(r^2),\iota(r^3))+(1,\iota(r),\iota(r^3),1)+(1,\iota(r),1,\iota(r))))\\
&=(\rho(g),\rho(g\iota(r)),\rho(g\iota(r^2)),\rho(g\iota(r^3)))+(\rho(g),\rho(g\iota(r)),\rho(g\iota(r^3)),\rho(g)) \\
&\quad +(\rho(g),\rho(g\iota(r)),\rho(g),\rho(g\iota(r))) \\
&=(1,\alpha_g(\iota(r))),\alpha_g(\iota(r^2)),\alpha_g(\iota(r^3)))+(1,\alpha_g(\iota(r)),\alpha_g(\iota(r^3)),1) \\
&\quad +(1,\alpha_g(\iota(r)),1,\alpha_g(\iota(r))) \\
&=(\alpha_g)_*\iota_*c.
\end{split}
\]

\begin{claim*}
If $g\in \sg{n}$ with $g\iota(r)g^{-1}\in H$, but $g\notin H$, then $\chi(\pi_{1*}\rho(g\iota_*c))=0$.
\end{claim*}

Let $a=g\iota(r)g^{-1}$.
Since $a$ is conjugate to $\iota(r)$ in $\sg{n}$, $a$ is a $4$-cycle, so we can write it as $a=(a_1,a_2,a_3,a_4)$.
If $a\in \iota(D_8)$, then $a=f\iota(r)f^{-1}$ for some $f\in \iota(D_8)$.
(Simply put, $a_1,a_2,a_3,a_4$ are $1,2,3,4$, but up to some dihedral permutation of their order.)
Then $gf^{-1}$ is in the centralizer in $\sg{n}$ of $\iota(r)$, which is $\iota(r)\times \sg{n-4}<H$.
This implies that $g\in H$, contrary to our hypothesis.
So $a\notin \iota(D_8)$, but $a\in H$.
Since $H=\iota(D_8)\times \sg{n-4}$, and $4$-cycles only appear in the factors, we have that $a\in \sg{n-4}$.
In particular, $\pi_1(a)=1\in D_8$.
Then
\[
\begin{split}
\pi_{1*}(\alpha_g)_*\iota_*c &= \pi_{1*}([a|a|a]+[a|a^2|a]+[a|a^3|a])\\
&=[1|1|1]+[1|1|1]+[1|1|1].
\end{split}
\]
In particular, $\chi(\pi_{1*}\rho(g\iota_*c))=0$.

Now we can finish the proof of the proposition.
We have that $\langle \iota(r)\rangle$ acts on $H\backslash\sg{n}$ by multiplication on the right.
Every orbit has size $1$, $2$, or $4$, since its size divides the order of $\iota(r)=(1,2,3,4)$.
Notice that $Hg$ is in an orbit of size $1$ if and only if $g\iota(r)g^{-1}\in H$, and it is in an orbit of size $2$ if and only if $g\iota(r)g^{-1}\notin H$ but $g\iota(r^2)g^{-1}\in H$.
Let $S_i$ be the set of coset representatives for orbits of size $i$, for $i\in\{1,2,4\}$.
Then
\[\mathrm{tr}(\iota_*c)=\sum_{g\in S_1}\rho(g\iota_*c)+\sum_{g\in S_2}\rho(g\iota_*c)+\sum_{g\in S_4}\rho(g\iota_*c).\]
We partition $S_4$ into $\langle\iota(r)\rangle$-orbits of size $4$;
then $\sum_{g\in S_4}\rho(g\iota_*c)$ is a sum of pieces of the form
\[\rho(g\iota_*c)+\rho(gr\iota_*c)+ \rho(gr^2\iota_*c)+\rho(gr^2\iota_*c).\]
But this is $\rho(g(1+r+r^2+r^3)\iota_*c)$, and by one of the claims, $\chi\circ \pi_{1*}$ sends this to $0$.
So $\chi(\pi_{1*}\sum_{g\in S_4}\rho(g\iota_*c))=0$.

Similarly, we partition $S_2$ into $\langle\iota(r)\rangle$-orbits of size $2$.
So $\sum_{g\in S_4}\rho(g\iota_*c)$ is a sum of pieces of the form
\[\rho(g\iota_*c)+\rho(gr\iota_*c).\]
Then another claim tells us that $2\chi(\pi_{1*}\sum_{g\in S_2}\rho(g\iota_*c))=0$.
This means that $\chi(\pi_{1*}\sum_{g\in S_2}\rho(g\iota_*c))$ is a multiple of $2$ in $\Z/4\Z$.

Now we look at $\sum_{g\in S_1}\rho(g\iota_*c)$.
If $g=1$, then 
\[\chi(\pi_{1*}\rho(g\iota_*c))=\chi(\pi_{1*}(\alpha_1)_*\iota_*c)=\chi(c)=1,\]
as explained in Remark~\ref{re:evalcocycle}.
If $g\neq 1$, then $g\iota(r)g^{-1}\in H$, but $g\notin H$, so $\chi(\pi_{1*}\rho(g\iota_*c))=0$.
Putting this together, we see that 
\[
\begin{split}
\chi(\pi_{1*}\mathrm{tr}\iota_*c) &= \chi(\pi_{1*}(\sum_{g\in S_1}\rho(g\iota_*c)+\sum_{g\in S_2}\rho(g\iota_*c)+\sum_{g\in S_4}\rho(g\iota_*c))) \\
&= 1+0 +\chi(\pi_{1*}(\sum_{g\in S_2}\rho(g\iota_*c)))+0=\pm1,
\end{split}
\]
since the unevaluated term is a multiple of $2$.
\end{proof}

\begin{corollary}\label{co:tightH3}
Let $n\geq 6$.
The group $H_3(\sg{n})$ is isomorphic to $(\Z/2\Z)^2\oplus (\Z/3\Z)\oplus (\Z/4\Z)$.
It is generated by $[\bar c^{31}_1]$ and $[\bar c^{33}_{135}]$, which have order $2$, and $[\bar c^{36}_1]$, which has order $3$, and  $[\bar c^{37}_1+\bar c^{32}_{13}]$, which has order $4$.
\end{corollary}

\begin{proof}
We have shown that the upper bound for $H_3(\sg{n})$ in Corollary~\ref{co:H3upperbound} is tight.
Corollary~\ref{co:gensnontriv} shows that the generators are nontrivial, so that the orders are correct for generators of prime order.
Proposition~\ref{pr:transfertrick} shows that $H_3(\sg{n})$ contains at least one element of order $4$.
If there were no generators of order $4$, then this would be impossible, so we deduce that $[\bar c^{37}_1+\bar c^{32}_{13}]$ has order $4$.
\end{proof}

\begin{corollary}
Every class in $H_3(\sg{n})$ has a representative that is the image of the fundamental class of an orientable $3$-manifold.
\end{corollary}

\begin{proof}
It is an exercise to adapt this for smaller $n$, so we assume that $n\geq 6$.
It is enough to show this for generators, since representative $3$-manifolds can be combined by connected sums.
It is almost obvious that $[\bar c^{31}_1]$ is the image of the fundamental class of the projective space $\R P^3$ under the map sending the generator to $s_1$.
Likewise, it is almost obvious that $[\bar c^{33}_{135}]$ is the image of the fundamental class of the $3$-torus $(S^1)^3$ under the map sending a basis for the fundamental group to $s_1$, $s_3$, and $s_5$.
It takes a little more work, but it is straightforward to show that $[\bar c^{36}_1]$ is the image of the fundamental class of a lens space $L(3,1)$ with fundamental group $\Z/3\Z$, under the map that sends a generator to $s_1 s_2$.
Similarly, one can show that, up to classes of order $2$, $[\bar c^{37}_1+\bar c^{32}_{13}]$ is the image of the fundamental class of a lens space $L(4,1)$ with fundamental group $\Z/4\Z$ under the map sending a generator to $s_1 s_2 s_3$.
Alternatively, one can use $\iota_*c$ as an order-$4$ generator, and explain that it is also the image of the fundamental class of $L(4,1)$ under this map. 
\end{proof}

\bibliography{3sbib}
\bibliographystyle{acm}

\vspace{0.2cm}
\small
 \noindent \textbf{Matthew B. Day,} {\sc Department of Mathematical Sciences, 309 SCEN, University of Arkansas, Fayetteville, AR 72701, U.S.A.} \\
\noindent  e-mail: {\tt matthewd@uark.edu}
\\
\\
 \noindent \textbf{Trevor Nakamura}, {\sc formerly of the Department of Mathematical Sciences, University of Arkansas}
\\
\noindent  e-mail: {\tt trevornakamura1994@gmail.com}

\end{document}